\documentclass{amsart}

\usepackage[a4paper,top=3cm,bottom=3cm,left=3cm,right=3cm]{geometry}
\usepackage{amsmath}
\usepackage{amssymb}
\usepackage{paralist}
\usepackage{graphics} 
\usepackage{epsfig} 
\usepackage{graphicx}  \usepackage{epstopdf}
\usepackage[colorlinks=true]{hyperref}
\usepackage{xcolor}
\usepackage{bbm}

\hypersetup{urlcolor=blue, citecolor=blue}

\newtheorem{theorem}{Theorem}[section]

\newtheorem{lemma}{Lemma}[section]
\newtheorem{proposition}{Proposition}[section]

\theoremstyle{definition}
\newtheorem{definition}[theorem]{Definition}
\newtheorem{remark}{Remark}

\numberwithin{equation}{section}

\newcommand{\e}{\varepsilon}
\newcommand{\R}{\mathbb{R}}
\newcommand{\T}{\mathbb{T}}
\newcommand{\C}{\mathbb{C}}

\newcommand{\Z}{\mathbb{Z}}
\newcommand{\N}{\mathbb{N}}

\newcommand{\bral}{[\![}
\newcommand{\brar}{]\!]}

\newcommand{\cK}{\mathcal{K}}
\newcommand{\td}{\mathtt{d}}
\newcommand{\te}{\mathtt{e}}
\newcommand{\ii}{\mathrm{i}}
\renewcommand{\le}{\leqslant}
\renewcommand{\geq}{\geqslant}

\thanks{Filippo Giuliani has received funding from PRIN-20227HX33Z \textit{Pattern formation in nonlinear phenomena} - Funded by the European Union-Next Generation  EU, Miss. 4-Comp. 1-CUP D53D23005690006. The authors acknowledge partial support from INdAM-GNAMPA}

\begin{document}
	
\title[Growth of Sobolev norms for the quantum Euler equations] 
{Energy cascade and Sobolev norms inflation for the quantum Euler equations on tori}

\author[F.~Giuliani]{Filippo Giuliani}
\address[F.~Giuliani]{Dipartimento di Matematica, Politecnico di Milano \\ Piazza Leonardo Da Vinci 32 \\ 20133, Milano, Italy}
\email{filippo.giuliani@polimi.it}

\author[R.~Scandone]{Raffaele Scandone}
\address[R.~Scandone]{Dipartimento di Matematica e Applicazioni ``R. Caccioppoli'', Universit\`a degli Studi di Napoli ``Federico II" \\ Complesso Universitario Monte S.~Angelo, via Cintia \\ 80126, Napoli, Italy}
\email{raffaele.scandone@unina.it}

	\begin{abstract}
In this paper we prove the existence of solutions to the quantum Euler equations on $\T^d$, $d\geqslant 2$, with almost constant mass density, displaying energy transfers to high Fourier modes and polynomially fast-in-time growth of Sobolev norms above the finite-energy level. These solutions are uniformly far from vacuum, suggesting that weak turbulence in quantum hydrodynamics is not necessarily related to the occurrence of vortex structures. 

{In view of possible connections with instability mechanisms for the classical compressible Euler equations, we also keep track of the dependence on the semiclassical parameter, showing that, at high regularity, the time at which the Sobolev norm inflations occur is uniform when approaching the semiclassical limit.}

Our construction relies on a novel result of Sobolev instability for the plane waves of the cubic nonlinear Schr\"odinger equation (NLS), 
{which is connected to the quantum Euler equations through the Madelung transform}. More precisely, we show the existence of smooth solutions to NLS, which are small-amplitude perturbations of a plane wave and undergo a polynomially fast $H^s$-norm inflation for $s>1$. The proof is based on a partial Birkhoff normal form procedure, involving the normalization of non-homogeneous Hamiltonian terms. 
\end{abstract}

\maketitle

	\tableofcontents
	
	\section{Introduction and main results}
	In this paper we consider a quantum fluid model in dimension $d\geqslant 2$, with periodic boundary conditions, and we investigate the existence of smooth solutions undergoing arbitrarily large growth of high order Sobolev norms as time evolves. This issue is deeply connected with the \emph{weak turbulence} theory for quantum systems, as discussed later in the introduction. The model under analysis is given by the following compressible quantum Euler equations
	\begin{equation}\label{eq:QHD}
		\begin{cases}
			\partial_t \rho+\operatorname{div}(\rho v)=0, \\
			\partial_t (\rho v)+\operatorname{div}(\rho v \otimes v)+ \nabla p(\rho)=\dfrac{\e^2}{2}\rho\,\nabla \left( \dfrac{\Delta \sqrt{\rho}}{\sqrt{\rho}}\right), 
		\end{cases}\qquad t\in\R,\;x\in \T^d:=(\R/ 2\pi\Z)^d,
	\end{equation}
	where the unknowns $(\rho,v):\R_t\times \T^d\to \R^+\times \R^d$ are interpreted respectively as mass density and velocity field, and $p(\rho):=\frac12\rho^2$ is a barotropic pressure term. We assume that the velocity field is irrotational, namely $\operatorname{rot} v=0$. Here $\e>0$ is the semiclassical parameter.
	
System \eqref{eq:QHD}, also known as \emph{quantum hydrodynamic} (QHD) system, arises in the description of hydrodynamic models where the quantum statistics becomes relevant at mesoscopic/macroscopic scales, which typically happens when the thermal de Broglie wavelength is comparable with the inter-particle distance. It is widely used in several physical contexts, including superfluidity \cite{Khal}, quasi-particles condensates \cite{carci}, quantum plasmas \cite{Haas}, and nanoscale semiconductor devices \cite{Jungel}. 

The third order nonlinear term appearing in the right hand side of the momentum equation encodes the quantum correction to the classical compressible Euler equations. Accordingly, the wave-type dispersion relation $\Omega(k)\approx|k|$ of the Euler equations is modified into a Schr\"odinger-type dispersion relation $\Omega(k)\approx \e|k|^2+|k|,$ as can be readily verified by linearizing \eqref{eq:QHD} around a constant equilibrium state. The QHD system actually falls into a more general class of dispersive hydro\-dynamic models, provided by the Euler-Korteweg system, which we briefly discuss in Section \ref{sec:op}.
	
The total mass $\mathcal{M}:=\|\rho\|_{L^1}$, and the total energy
	\begin{equation}\label{def:energy}
		\mathcal{E}:=\frac12\int_{\T^d} \e|\nabla\sqrt{\rho}|^2+\e\rho \,|v|^2+\e^{-1}\rho^2\,dx
	\end{equation}
	are formally conserved along the flow of \eqref{eq:QHD}. Moreover, since the quantum correction can be written in divergence form as
	\begin{equation}\label{eq:qc}
	{\textstyle{\dfrac{\e^2}{2}}}\rho\,\nabla \left( \dfrac{\Delta \sqrt{\rho}}{\sqrt{\rho}}\right)=\e^2\,\nabla\cdot\big({\textstyle{\frac14}}\nabla^2\rho-\nabla\sqrt{\rho}\otimes \nabla\sqrt{\rho}\big),
\end{equation}
	also the total momentum $\int_{\mathbb{T}^d} \rho \,v \,dx$ is formally conserved.
	
The formulation of system \eqref{eq:QHD} requires the mass density to be everywhere non-vanishing. In presence of vacuum regions, the situation is indeed more involved: to begin with, the velocity field $v$ is not well-defined on $\{\rho=0\}$, and one needs to carefully define the correct hydrodynamic variables; in addition, the irrotationality condition $\operatorname{rot} v=0$ outside vacuum is still compatible with the presence of quantum vortices, namely solutions whose velocity field has non-zero circulations. For a comprehensive discussion on the mathematical framework of quantum fluid models in presence of vacuum, we refer to \cite{AM-review} and references therein. 

In this work we consider only the non-vacuum regime, 
aiming to show that the mechanism of growth of Sobolev norms, and related weakly turbulent behaviors in quantum hydrodynamics, are not necessarily caused by the presence of vortices. To do that, we exploit the connection between the QHD system \eqref{eq:QHD} and the cubic NLS equation, for which solutions exhibiting energy cascade have been proved to exist (we mention the seminal paper \cite{CKSTT}, and we refer to Section \ref{ss_11} for a more detailed overview).

We will prove the existence of solutions to the QHD system, whose mass density varies slightly around a constant density and which undergo a polynomially fast norm inflation. At sufficiently high regularity, the upper bound on the time at which the arbitrarily large $H^s$ norm inflations occur is uniform when $\e\to 0$. We also show the occurrence of energy cascades to high frequencies over exponential times. 
Our results are a consequence of a novel theorem on Sobolev instability for plane wave solutions to the cubic NLS -- see Theorem \ref{thm:growth}.

 The connection between quantum hydrodynamic systems and NLS equations has been deeply investigated in the literature (see the discussion in Section \ref{ss_11}), mainly in the context of well-posedness/stability results. In the breakthrough paper \cite{Merle} by Merle-Rapha\"el-Rodnianski-Szeftel (see also the recent work \cite{Cao} by Cao Labora-G\'omez Serrano-Shi-Staffilani), blow-up profiles for a class of defocusing, \emph{supercritical} NLS equations is obtained by showing that the presence of dispersion does not alter the compression mechanism for a family of smooth, self-similar solutions to the classical compressible Euler equations. In the present paper, in the context of wave turbulence, somehow we go in the opposite direction, transporting a dynamical information (energy transfers) from a defocusing, \emph{subcritical} NLS to the corresponding hydrodynamic system.

\smallskip


Concerning the non-vacuum regime of the QHD system, local well-posedness for system \eqref{eq:QHD} (and for a more general class of Euler-Korteweg systems) on $\T^d$, at sufficiently high Sobolev regularity ($H^s(\T^d)$ with $s>2+\frac{d}{2}$), has been proved by Berti-Maspero-Murgante \cite{BMM}. In the case of irrational tori, long time existence is discussed in  Feola-Iandoli-Murgante \cite{FIM} and Bambusi-Feola-Montalto \cite{BFM}. In the Euclidean setting, the solution theory for the QHD system (and related quantum fluids models) has been widely investigated \cite{AHMHyp,HIME,AM17S,AM16D,AS17,LIMAR,AHSP,Jungel,HLLA,DFM,ADAM,NISU,AP,AMSMS}, see also \cite{AMHZ-Rims} and references therein. In particular, existence of global-in-time, finite energy weak solutions has been proved by Antonelli-Marcati \cite{AMQ3,AMQ2}, also allowing for vacuum regions.
	
Comparing with the case of classical compressible Euler equations, where various mechanisms of breakdown of regularity are known (see e.g.~the seminal paper \cite{Sideris} by Sideris, and the monograph \cite{Dafe} by Dafermos), the solution theory for the QHD system is more satisfactory, as a consequence of the enhanced dispersion induced by the quantum correction \eqref{eq:qc}. Nevertheless, the interaction between nonlinear dispersive waves may still lead to long-time ``instability'' effects, which typically emerge with a transfer to high frequencies of certain energy densities. The study of such phenomena falls into the realm of the aforementioned \emph{weak turbulence} theory, whose general aim is to provide a statistical description of interacting nonlinear waves outside equilibrium, see e.g.~the monograph \cite{Nazarenko} by Nazarenko. From a mathematical point of view, a possible \emph{deterministic} way to encode weak turbulent mechanisms for the QHD system \eqref{eq:QHD} is to show the existence of solutions which exhibit arbitrarily large growth over time of the Sobolev norm
\begin{equation}\label{eq:sob_norm}
\|(\rho,v)\|_{M^s}:=\|\rho\|_{H^s}+\|v\|_{H^{s-1}},
\end{equation}
or in alternative of the quantity
\begin{equation}\label{eq:lambda_norm}
	\|(\sqrt{\rho},\Lambda)\|_{M^s}:=\|\sqrt{\rho}\|_{H^s}+\|\sqrt{\rho}v\|_{H^{s-1}},
\end{equation}
at regularity $s>1$, namely above the finite-energy level $s=1$. The choice of the norm \eqref{eq:sob_norm} is natural from an analytical perspective; on the other side, in view of the expression \eqref{def:energy} for the (conserved) total energy, a growth in-time of \eqref{eq:lambda_norm} for $s>1$ directly implies the so-called \emph{direct energy cascade}, namely a transfer to high frequencies of the energy density. In this paper we provide some results in this direction, also highlighting the dependence on the semiclassical parameter $\e$. Apart from its own interest, we believe this is a starting point in order to connect  instability mechanisms of the compressible Euler equations with the weak turbulence theory for quantum hydrodynamics, in the semiclassical limit as $\e\to 0$.

As already discussed, we only seek for solutions without vacuum regions. To this aim, let us fix a mass level $\mathcal{M}=(2\pi)^dm$, and consider pairs $(\rho,v)$ such that $\rho$ is a small-amplitude perturbation of the constant density $\rho_{\mathrm{const}}(x)=m$. Our first main result shows that, within this class of hydrodynamic states, we can actually find smooth solutions to the QHD system \eqref{eq:QHD} displaying arbitrarily large growth of $\|(\rho,v)\|_{M^s}$  as time evolves, for any $s>1$. More precisely, we have the following.

		\begin{theorem}\label{th:main_qhd}
		Let $d\geqslant 2$, $m,\sigma >0$, $\e\in(0,1]$, $s>1$. There exists $C>0$, independent on $\e$, such that the following holds. For every $\mathcal{K}>0$ large enough, there exists a  smooth solution $(\rho,v)$ to system \eqref{eq:QHD}, defined on $[0,T]\times\T^d$ for some 
		\begin{equation}\label{eq:T_pol}
			0<T\leqslant C\e\Big(\frac{\mathcal{K}}{\e}\Big)^{\frac{3(1+\sigma)}{s-1}},
			\end{equation}
		satisfying the estimates
		\begin{gather}\label{eq:growth_K}
		 \frac{\|(\rho,v)(T) \|_{M^s}}{\| (\rho,v)(0) \|_{M^s}}\geqslant \mathcal{K},\\
\label{close_weak_norm}
		\sup_{t\in [0, T]} \|\rho(t)-m\|_{ L^\infty}\leqslant C\,\Big(\frac{\mathcal{K}}{\e}\Big)^{-\frac{\sigma}{s-1}}.
	\end{gather}
	
	\end{theorem}
Estimates \eqref{eq:T_pol} and \eqref{eq:growth_K} show that the growth of $\|(\rho,v)\|_{M^s}$ can happen \emph{polynomially fast} in the growth-ratio $\mathcal{K}$. Moreover, the bound \eqref{close_weak_norm} guarantees (for $\mathcal{K}$ sufficiently large) that the mass density $\rho$ is uniformly bounded away from zero, {and its variation from a fixed constant density can be actually made arbitrarily small}. In particular, vacuum regions are absent, and no quantum vortices emerge up to time $T$. 
As anticipated, the Sobolev norms inflation is purely related to the weak interaction between nonlinear dispersive waves. Concerning the dependence on the semiclassical parameter $\e$, observe that by choosing $\sigma>0$ sufficiently small we can obtain, at regularity $s>4$, an upper bound on the time of growth $T$ which is uniform (and actually infinitesimal) as $\e\to 0$.



\smallskip

Next we state our second main result, showing the occurrence of \emph{direct energy cascades} for the QHD system \eqref{eq:QHD}, namely the existence of solutions with arbitrarily large growth in time of the $M^s$-norm of the hydrodynamic state $(\sqrt{\rho},\Lambda)$.

	\begin{theorem}\label{th:main_cascade}
Let $d\geqslant 2$, $m,\sigma>0$, $\e\in[0,1)$, $s>1$. There exist $C>0$, independent on $\e$, and a universal constant $c>0$ such that the following holds. For every $\mathcal{K}>0$ large enough, there exists a  smooth solution $(\rho,v)$ to system \eqref{eq:QHD}, defined on $[0,T]\times\T^d$ for some 
\begin{equation}\label{eq:T_exp}
	0<T\leqslant C\e^{-9}\,e^{(\e^{-1}\mathcal{K})^{\frac{c}{s-1}}},
\end{equation}
satisfying the estimates
\begin{gather}\label{eq:growth_K_exp}
	\frac{\|(\sqrt{\rho},\Lambda)(T) \|_{M^s}}{\| (\sqrt{\rho},\Lambda)(0) \|_{M^s}}\geqslant \mathcal{K},\\
	\label{close_weak_norm_exp}
	\sup_{t\in [0, T]} \|\rho(t)-m\|_{L^\infty}\leqslant C\,\Big(\frac{\mathcal{K}}{\e}\Big)^{-\frac{\sigma}{s-1}}.
\end{gather}
\end{theorem}

Comparing the above result with Theorem \ref{th:main_qhd}, we can notice the worse estimate on the time of growth, which is now exponential rather than polynomial in the growth-ratio $\mathcal{K}$, and blow-up (uniformly in $s>1$) when the semiclassical parameter $\e$ goes to zero. This is related to some technical issues arising when dealing with the quantity $\Lambda$. Note that, a priori, $\Lambda=\sqrt{\rho} v$ could stay bounded in $M^s$ even though $\|\sqrt{\rho}\|_{H^s}$ and $\|v\|_{H^{s-1}}$ are growing, a scenario which can not be excluded (at least for polynomial times) by our techniques. We also mention that \eqref{close_weak_norm_exp} can be actually replaced by an exponential bound (in the growth-ratio $\mathcal{K}$), which we do not write explicitly for the sake of concreteness.

\smallskip

As anticipated above, Theorems \ref{th:main_qhd}, \ref{th:main_cascade} rely on a novel result for NLS, that is Theorem \ref{thm:growth} below, which concerns the Sobolev instability of plane waves. 
We refer to the next section for further discussion on the results of Sobolev instabilities for NLS equations on compact manifolds.

\subsection{Connection with NLS equation and fast norm inflation close to plane waves}\label{ss_11}
The proofs of Theorems \eqref{th:main_qhd} and \eqref{th:main_cascade} are both based on the relation between the QHD system \eqref{eq:QHD} and a suitable \emph{semilinear} Schr\"odinger equation. More precisely, assuming that $u:\R_t\times\T^d\to \C$ is a solution to the cubic NLS equation
\begin{equation}\label{eq:NLS_stand}
	\mathrm{i}\varepsilon\partial_tu=-{\textstyle{\dfrac{\varepsilon^2}{2}}}\Delta u+ |u|^2u,
\end{equation}
and writing $u$ as $u=|u|e^{\mathrm{i}\theta}$, then its \emph{Madelung transform}
\begin{equation}\label{eq:madelung}
\rho:=|u|^2,\;v:=\varepsilon\nabla\theta,
\end{equation}
formally is a solution to system \eqref{eq:QHD}, with $\operatorname{rot}v=\operatorname{rot}\nabla\theta=0$. 

The connection between quantum hydrodynamic models and wave-functions dynamics (governed by Schr\"odinger-type equations), originally observed in the seminal work by Madelung \cite{Madelung}, has been widely exploited in the mathematical literature (see e.g.~\cite{CDS,AM-review} and references therein). 

The derivation of system \eqref{eq:QHD} from the NLS \eqref{eq:NLS_stand}, through the Madelung transform, 
can be readily justified if $u$ is sufficiently smooth and not vanishing at any point (see the explicit computation in the proof of Theorem \ref{th:main_qhd}). In presence of vacuum regions, and for low-regularity wave-functions, a more refined analysis is instead required \cite{AMQ3,AMQ2,AP,AM-review}. We also point out that, in general, it is unknown whether the solutions to \eqref{eq:QHD} provided by the NLS equation are unique (see \cite{AMZ1,AMZ2} for some recent results in this direction). In fact, reconstructing a wave-function from the hydrodynamic state requires the circulations of $v$ around vacuum regions to be quantized  \cite{PR-quanta,Bianchini}, a condition which is a priori only \emph{formally} preserved (even in the smooth case) in presence of vacuum regions.

As already discussed, we focus here on the non-vacuum regime, and more precisely we seek for solutions to the QHD system with mass density $L^{\infty}$-close to the constant density $\rho_{\mathrm{const}}=m$. Accordingly, on the Schr\"odinger side we consider small amplitude perturbations, up to a constant (in space) phase of a \emph{plane wave}
\begin{equation}\label{eq:plane_wave_original}
\mathtt{p}_{m,k}(t, x):=\sqrt{m}\,e^{\mathrm{i} k\cdot x-\ii\big(\e\frac{|k|^2}{2}+\frac{m}{\e}\big) t},  \qquad k\in \Z^2\,,
\end{equation}
that is the solution to \eqref{eq:NLS_stand} with initial data $\sqrt{m}e^{\ii k}$. 
The orbits $\mathtt{p}_{m,k}(t)$ correspond, through the Madelung transform, to the stationary solutions $(m, k)$ of the QHD system.

We then consider wave-functions of the form 
\begin{equation}\label{eq:smap}
u(t,x)=\sqrt{m}\,e^{\mathrm{i} (k\cdot x+\varphi(t))}+r(t,x),
\end{equation}
with $\varphi(t)\in\T$ and $\|r\|_{L_{t,x}^{\infty}}$ sufficiently small. 
In particular, $u(t)$ is $L^{\infty}$-close to the embedded one dimensional torus supporting the plane wave orbits in the phase space.
Observe that this yield the desired condition $\rho>0$, but does not give any information on the closedness of $v(x)$ to the constant velocity field $k$.


For technical reasons, it is useful to measure the vicinity to the plane wave using a stronger norm than $L^{\infty}$. More precisely, we consider the Wiener space
\begin{equation}\label{eq:def_wiener}
	\ell^1(\T^d):=\left\{ u\colon \T^d\to \mathbb{C} : \| u \|_{\ell^1(\T^d)}:=\|\widehat{u}\|_{\ell^1(\Z^d)}<+\infty \right\},
\end{equation}
which is a Banach algebra with respect to the pointwise product. Note indeed that the $\ell^1$ norm controls the $L^{\infty}$ norm, as expressed by \eqref{eq:emb}, and also the $L^2$ norm (in view of the embedding $\ell^{1}(\Z^d)\hookrightarrow \ell^2(\Z^d)$ and  the Plancharel identity).

Our next result shows that the NLS equation \eqref{eq:NLS_stand} admits smooth solutions of the form \eqref{eq:smap}, with $\|r\|_{L_t^{\infty}\ell_x^1}$ sufficiently small, displaying arbitrarily large, \emph{polynomially} fast growth of the $H^s$ norm, $s>1$, as time evolves. A result describing such behavior for solutions close to quasi-periodic finite gap orbits, has been first proved by Guardia-Hani-Haus-Maspero-Procesi in \cite{GuardiaHHMP19}, with \emph{exponential} upper bounds on the time of growth. Here we also show that, by retaining the exponential upper bound, we can additionally obtain the extra smallness assumption \eqref{waa} for $\|r\|_{L_t^{\infty}\ell_x^1}$ in terms of the $H^1$ norm of the solution, which in our construction is of size $\exp(\mathcal{K})$. We explicitly consider only the two-dimensional case, the same result then holds in dimension $d\geqslant 3$ by considering solutions which are only dependent upon two of the spatial coordinates.

\begin{theorem}\label{thm:growth}
	Fix $m,\sigma >0$, $\e\in(0,1]$, $s>1$, $k\in\Z^2$. There exist $C:=C(m)>0$ such that the following holds. For every $\mathcal{K}>0$ large enough, there exists a global, smooth solution $u(t)$ of the cubic NLS equation \eqref{eq:NLS_stand} on $\T^2$ such that
	\begin{gather}
		\label{eq:grka}
		\frac{\|u(T) \|_{H^s}}{\| u(0) \|_{H^s}}\geqslant \mathcal{K}  \qquad \mathrm{for\,\,\,\,\,some} \qquad 0<T\leqslant \frac{C \e}{s-1}\, \mathcal{K}^{\frac{3(1+\sigma)}{s-1}},\\
		\label{eq:clsm}
		\sup_{t\in [0, T]} \,\,\inf_{\varphi\in \T} \| u(t)-\mathtt{p}_{m, k}(0)\,e^{\mathrm{i} \varphi} \|_{\ell^1}\leqslant  C \mathcal{K}^{-\frac{\sigma}{s-1}}.
	\end{gather}
	Moreover, given $\delta_*>0$, and allowing $C$ to depend also on $\delta_*$, we can construct a global, smooth solution $u(t)$ to \eqref{eq:NLS_stand} satisfying
	\begin{equation}\label{wa}
		\frac{\|u(T) \|_{H^s}}{\| u(0) \|_{H^s}}\geqslant \mathcal{K}  \qquad \mathrm{for\,\,\,\,\,some} \qquad 0<T\leqslant C\e^{-9}\,e^{\mathcal{K}^{\frac{c}{s-1}}},
	\end{equation}
for a suitable universal constant $c>0$, the bound \eqref{eq:clsm}, and the additional estimate
	\begin{equation}\label{waa}
		\sup_{t\in [0, T]} \,\left(\,\inf_{\varphi\in \T} \| u(t)-\mathtt{p}_{m,k}(0)\,e^{\mathrm{i} \varphi} \|_{\ell^1}\right)\,\, \| u(t)\|_{H^1}\leqslant \delta_*.
	\end{equation}
	
\end{theorem}

For $\mathcal{K}$ large enough, estimate \eqref{eq:clsm} implies in particular absence of vacuum, at least up to the time $T$ when the growth of Sobolev norm is attained. As already discussed, this readily implies that the Madelung transform of $u$ provides a smooth solution to  the QHD system \eqref{eq:QHD} on $[0,T]\times\T^2$. In addition, estimates \eqref{eq:clsm} and \eqref{waa} allow, respectively, to obtain suitable equivalences between the $H^s$-norm of the wave-function $u$ and the $M^s$-norm of the corresponding hydrodynamic states $(\rho,v)$ (Proposition \ref{pr:equiv1}) and $(\sqrt{\rho},\Lambda)$ (Proposition \ref{pr:equiv2}), eventually leading to the existence of weakly turbulent solutions to the QHD system \eqref{eq:QHD} (when $d=2$, and then lifted to higher dimension), as stated in Theorems \ref{th:main_qhd} and \ref{th:main_cascade}.

\begin{remark}
We point out that in dimension $d=1$ the cubic NLS equation is completely integrable, and its infinite conservation laws provide uniform in time bounds for all the Sobolev norms of the solutions. Accordingly, we expect stability for the 1d QHD system in the non-vacuum regime.
\end{remark}
 
 \begin{remark}
We mention the work \cite{Faou} by Faou-Gauckler-Lubich, where the authors provide a Nekhoroshev stability result in Sobolev regularity for solutions starting close to a large set of plane waves.
 We cannot really compare Theorem \ref{thm:growth} and this result, since in \cite{Faou} the solutions considered start close to plane waves in the strong $H^s$ norm. 
\end{remark}

\medskip\noindent\textbf{Overview on results of growth of Sobolev norms for NLS.} Theorem \ref{thm:growth} falls into the category of results of existence of solutions to NLS equations undergoing an arbitrarily large growth of Sobolev norms, initiated by the works of Bourgain \cite{Bou95,Bou96} and Kuksin \cite{Kuk}. In 2010 this study has seen a significant progress after the seminal work by Colliander-Keel-Staffilani-Takaoka-Tao \cite{CKSTT} concerning the existence of solutions to the cubic NLS on $\T^2$ with arbitrarily small data and arbitrarily large $H^s$ norms, with $s>1$, at later times. 
Hani \cite{Hani} proved the existence of solutions with initial data arbitrarily close to plane waves in $H^s$ exhibiting a $H^s$-norm explosion with $s\in (0, 1)$.
Results similar to \cite{CKSTT} have been proved by Guardia for NLS equations with convolution potentials, Haus-Procesi \cite{HPquintic} for the quintic NLS equation, and Guardia-Haus-Procesi \cite{GuardiaHP16} for any NLS equation with analytic nonlinearities. Hani-Pausader-Tzvetkov-Visciglia \cite{RT2} considered the cubic NLS on the product space $\R\times\T^2$, showing the existence of orbits with unbounded Sobolev trajectories (the analogous result on compact manifold remains an open question). Recently, Maspero-Murgante \cite{MasMur} proved lower bounds on the growth of Sobolev norms for solutions to a class of fractional NLS equations with one derivative in the nonlinearity.

Upper bounds on the time $T$ at which the solutions achieve a prescribed growth of norms have been provided for the first time by Guardia-Kaloshin in \cite{GuardiaK12}.
More precisely, they showed that solutions $u(t)$ to the cubic NLS satisfying  
\[
\| u(0) \|_{H^s}\leqslant \mu, \qquad\| u(T)\|_{H^s}\geqslant \mathcal{K}, \qquad \mu\ll 1, \qquad \mathcal{K}\gg 1,
\]
can be constructed in such a way that the time $T$ has an exponential-type upper bound in terms of $\cK/\mu$, that is
\[
0<T \leqslant \exp\left({\left({\cK}/{\mu}\right)^c}\right) \qquad \mathrm{for\,\,some\,\,}c>0.
\] They also show that, if one does not assume smallness on the $H^s$ norm of the initial datum, then one can construct solutions $u(t)$ exhibiting an arbitrarily large growth, say
\[
\| u(T) \|_{H^s}\geqslant \cK \| u(0)\|_{H^s} \quad  \mathrm{for\,\, an\,\, arbitrarily\,\, large}\,\,\cK>0,
\]
with the time $T$ satisfying a \emph{polynomial} bound in $\cK$, namely
\[
0<T\leqslant \cK^c\qquad \mathrm{for\,\,some\,\,}c>0.
\]
Moreover, in this case, the solutions can be chosen to have an arbitrarily small $L^2$ norm.
The result provided by Theorem \ref{thm:growth} is of this type, but it concerns solutions arbitrarily close in the $\ell^1$-norm (and thus also in the $L^2\cap L^{\infty}$ norm) to plane waves. The existence of such solutions has been previously proved in \cite{GuardiaHHMP19}, which considered the more general problem of solutions starting close to finite gap quasi-periodic solutions of the cubic NLS on $\T^2$. The upper bounds on the time of the growth provided in  \cite{GuardiaHHMP19} are \emph{exponential} in the factor of the growth $\mathcal{K}$, while in Theorem \ref{thm:growth} these bounds are \emph{polynomial} in $\mathcal{K}$. 
This improvement has been obtained also by the first author in \cite{Giu} for cubic NLS equations, with and without the presence of convolution potentials, on both rational and irrational two dimensional tori.

Finally, we mention that Sobolev norms inflation for NLS solutions has been first shown in the irrational setting by the first author and Guardia \cite{GiuGu}.

\subsection{Open problems and perspectives}\label{sec:op}
The QHD system \eqref{eq:QHD} admits stationary states of the form $(m,k)$, $m>0$, $k\in\Z^d$, representing a constant mass density at equilibrium (in a reference frame moving at velocity $k$). A challenging open problem concerns the possible \emph{orbital} instability, in the $M^s$-topology with $s>1$, of these stationary solutions. This would require to improve Theorem \ref{thm:growth} (possibly with weaker upper bounds on the time of growth) by showing that the orbits with growing $H^s$ Sobolev norms, with $s>1$, may be additionally chosen to start close, in the $H^s$-norm, to a plane wave.
%

We also point out that we do not have any a priori information on which between $\|\rho\|_{H^s}$ and $\|v\|_{H^{s-1}}$ (respectively between $\|\sqrt{\rho}\|_{H^s}$ and $\|\Lambda\|_{H^{s-1}}$ ) actually grow in time. This remains an interesting open question, which requires new ideas and insights from classical turbulence theory in fluid dynamics.

It would be possible to extend Theorem \ref{thm:growth} to analytic non-linearities, in the same spirit of the aforementioned paper \cite{GuardiaHP16}. In turn, this would allow to generalize Theorems \ref{th:main_qhd} and \ref{th:main_cascade} to the case of analytic (barotropic) pressures. Another extension is possible in the framework of irrational tori, using the ideas of \cite{GiuGu}.

As already mentioned, the QHD system belongs to a larger class of fluid dynamics models, provided by the Euler-Korteweg system
\begin{equation}\label{eq:EK}
	\begin{cases}
		\partial_t \rho+\operatorname{div}(\rho v)=0, \\
		\partial_t (\rho v)+\operatorname{div}(\rho v \otimes v)+ \nabla p(\rho)=\rho\nabla \operatorname{div}\Big((k(\rho)\nabla\rho)-{\textstyle{\frac{1}{2}}}k'(\rho)|\nabla\rho|^2\Big),
	\end{cases}
\end{equation}
where $k(\rho)\geqslant 0$ is the so-called capillarity coefficient. System \eqref{eq:EK} arises, among the others, in the description of capillarity effects in diffuse interfaces, see e.g.~ \cite{Kort,SBG}. The choice $k(\rho)=\frac{1}{4\rho}$ leads to system \eqref{eq:QHD}.  As for QHD, the Euler-Korteweg system \eqref{eq:EK} is also related, through the Madelung transform, to a wave-function dynamics, given in this case by the \emph{quasi-linear} NLS 
\begin{equation}\label{eq:quasilinear}
	\ii \partial_t u={\textstyle{-\frac{1}{2}}} \Delta  u-\operatorname{div}\big(\tilde{k}(|u|^2) \nabla|u|^2\big)u+{\textstyle{-\frac{1}{2}}}\tilde{k}^{\prime}(|u|^2)|\nabla|u|^2|^2u+|u|^2u,
\end{equation}
where $\tilde{k}(\rho)=k(\rho)-\frac{1}{4\rho}$. Observe that with the choice $k(\rho)=\frac{1}{4\rho}$ \eqref{eq:quasilinear} reduce to the \emph{semi-linear} NLS equation \eqref{eq:NLS_stand}. It would be interesting to extend Theorem \ref{thm:growth} to quasi-linear NLS of the form \eqref{eq:quasilinear}, allowing then for weak-turbulence type results for the Euler-Korteweg system.

Finally, in the context of weak turbulence theory, an important challenge is to derive an effective statistical description (at certain time-frequency scales) of the interaction between non-linear waves in system \eqref{eq:QHD}. For the NLS equation, such effective model is rigorously provided by a suitable quantum version of the Boltzmann kinetic equation \cite{Hani-Deng}. A similar analysis for the QHD system is still missing, even at formal level, and can not be directly deduced from the corresponding results for the NLS. Such analysis would be crucial for a deeper understanding of the turbulence mechanisms for quantum fluid systems.

\medskip\noindent\textbf{Acknowledgements}. Filippo Giuliani has received funding from PRIN-20227HX33Z 
\textit{Pattern formation in nonlinear phenomena - Funded by the European Union-Next Generation} EU, Miss. 4-Comp. 1-CUP D53D23005690006. The authors acknowledge partial support from INdAM-GNAMPA.

\medskip\noindent\textbf{Notation.}
We write $A\lesssim B$ when $A\leqslant CB$ for some positive constant $C$. When $C$ depends on some parameter $k$, we write $A\lesssim_{k}B$. We set $\langle x\rangle:=\sqrt{1+x^2}$. The Fourier transform $\widehat{f}$ of a function $f\in\T^d$, given by $\widehat{f}(k)=(2\pi)^{-d}\int_{\T^d}f(x)e^{-ik\cdot x}dx$, will be occasionally also denoted by $\mathcal{F}f$. The symbol $J^s$, for $s\in\R$, denotes the multiplier operator given by $\widehat{J^sf}(k):=\langle k\rangle^s\widehat{f}(k)$. For $d\in\N^+$, $s\geq 0$ we denote by $H^s(\T^d)$ the classical Sobolev space of functions $f\in L^2(\T^d)$ such that $J^sf \in L^2(\T^d)$. 

\section{Proof of the main results}
In this section we prove our main results, Theorems \ref{th:main_qhd} and \ref{th:main_cascade}, on the norm inflation and the direct energy cascade for the QHD system. It will be sufficient to explicitly consider only the two-dimensional case, as 2d-solutions can be embedded into higher dimension -- see \eqref{eq:embed} below. We will use the weak turbulence result for the NLS equation \eqref{eq:NLS_stand}, provided by Theorem \ref{thm:growth}, whose proof is postponed to the next section. We will also exploit suitable equivalence results (see Propositions \ref{pr:equiv1} and \ref{pr:equiv2} below) between the $H^s$-norm of the wave-function $u$ and the $M^s$-norm of the hydrodynamic states $(\rho,v)$ and $(\sqrt{\rho},\Lambda)$, respectively.

\begin{proposition}\label{pr:equiv1}
Fix $s>1$, $m>0$, $\e\in(0,1]$ and $\varphi\in\T$. Given $u=|u|e^{\mathrm{i}\theta}\in H^s(\T^2)$, with
\begin{equation}\label{eq:clm}
\|u-\sqrt{m}e^{\mathrm{i}\varphi}\|_{L^{\infty}}\leqslant{\textstyle{\frac{\sqrt{m}}{2}}},
\end{equation}
and setting $\rho:=|u|^2$, $v:=\e\nabla\theta$, we have the estimate
	\begin{equation}\label{eq:equi}
	M^{-1}\|(\rho,v)\|_{M^{s}}\leqslant \|u\|_{H^s}\leqslant M\e^{-1}\|(\rho,v)\|_{M^{s}},
\end{equation}
for some constant $M:=M(s,m)>1$.
\end{proposition}

		\begin{remark}
	It is worth pointing out that a condition as \eqref{eq:clm}, which guarantees that the wave function is (strictly) non-vanishing, is unavoidable in order to prove the equivalence of norms \eqref{eq:equi}.
		\end{remark}

A result in the same spirit of Proposition \ref{pr:equiv1} have been obtained in \cite[Lemma 2.1]{FIM}, in the context of long-time stability for the QHD system on irrational tori. We also mention \cite[Theorem 1.6]{Wegner}, where the author proves a local bi-Lipschitz equivalence between Schr\"odinger and QHD norms on $\R$, also allowing for non-trivial boundary conditions at infinity. 

In the above results, the constants in the equivalence of norms may depend on the $H^s$-norm of the wave-function. In view of our application to the existence of solutions with \emph{large} $H^s$-norm, here it is crucial instead that the constants in \eqref{eq:equi} depend only on the $L^{\infty}$-norm of $u$, which we are able to control \emph{uniformly in time} for the solutions to NLS constructed in Theorem \ref{thm:growth}.

\begin{proof}[Proof of Proposition \ref{pr:equiv1}]
We first consider the case $\e=1$. Observe that $u$ and $e^{-\ii\varphi}u$ have the same $H^s$-norm, and lead to the same hydrodynamic variables $(\rho,v)$. Without loss of generality, we can then set $\varphi=0$.

In view of \eqref{eq:clm}, we have $\Re (u)\geqslant\frac{\sqrt{m}}{2}>0$ and $\Im(u)\leqslant\frac{\sqrt{m}}{2}$. In particular, $\theta$ is given (up to an integer multiple of $2\pi$) by 
$$\theta=\arctan\left(\frac{\Im(u)}{\Re(u)}\right)\in{\textstyle{\big[\!-\frac{\pi}{4},\frac{\pi}{4}\big]}}.$$
We have the bound
\begin{equation}\label{eq:vt-temp}
\|v\|_{H^{s-1}}=\|\nabla\theta\|_{H^{s-1}}\lesssim\|\theta\|_{H^{s}}\lesssim_{s} \|\tan\theta\|_{H^s},
\end{equation}
where we used the composition Lemma \ref{le:compo} with $h(x)=\arctan(x)$ in the last step (observe that $\|\tan\theta\|_{L^{\infty}}\leqslant 1$). Moreover, since $|\cos\theta|\geqslant\frac{\sqrt{2}}{2}>0$, Lemma \ref{le:compo} yields $\|\sec\theta\|_{H^s}\lesssim_s\|\cos\theta\|_{H^s}$. Using \eqref{eq:vt-temp} and the Leibniz rule \eqref{eq:bil_sym}  we then obtain
\begin{equation}\label{eq:vt}
	\begin{split}
\|v\|_{H^{s-1}}&\lesssim_s \|\tan\theta\|_{H^s}\lesssim_s \|\sin\theta\|_{H^s}\|\sec\theta\|_{L^{\infty}}+\|\sin\theta\|_{L^{\infty}}\|\sec\theta\|_{H^s}\\
&\lesssim_s\|\sin\theta\|_{H^s}
+\|\cos\theta\|_{H^s}\lesssim\|e^{\ii\theta}\|_{H^s}.
	\end{split}
\end{equation}
 On the other side, denoting by $\theta_0:=\theta-\int_{\T^2}\theta(x)dx$ the projection of $\theta$ onto the space of zero mean functions, we have
\begin{equation}\label{eq:tv}
\|e^{\ii\theta}\|_{H^s}=\|e^{\ii\theta_0}\|_{H^s}\leqslant\|\cos\theta_0\|_{H^s}+\|\sin\theta_0\|_{H^s}\lesssim_{s} 1+\|\theta_0\|_{H^s}\lesssim 1+\|\nabla\theta\|_{H^{s-1}}=1+\|v\|_{H^{s-1}},
\end{equation}
where we used Lemma \ref{le:compo} in the third step, and Poincar\'e inequality in the fourth step.

Next, using the fractional Leibniz rule \eqref{eq:bil_sym} we get
\begin{equation}\label{eq:est-rho-sm}
\|\rho\|_{H^s}=\|u\overline{u}\|_{H^s}\lesssim_s\|u\|_{H^s}\|\overline{u}\|_{L^{\infty}}+\|u\|_{L^{\infty}}\|\overline{u}\|_{H^s}\lesssim_{s,m} \|u\|_{H^s}.
\end{equation}
Moreover, \eqref{eq:clm} and the bound $||a|-|b||\leqslant|a-b|$ yields $\frac12\sqrt{m}\leqslant \sqrt{\rho} \leqslant \frac32\sqrt{m}$. In particular, the composition Lemma \ref{le:compo} and \eqref{eq:est-rho-sm} give
\begin{gather}\label{eq:esum}
\|\sqrt{\rho}\|_{H^s}\lesssim_{s,m}\|\rho\|_{H^s}\lesssim_{s,m} \|u\|_{H^s},\\
\label{eq:erum}
\|\rho^{-1/2}\|_{H^s}\lesssim_{s,m}\|\rho\|_{H^s}\lesssim_{s,m} \|u\|_{H^s}.
\end{gather}
Observe now that
$$\|v\|_{H^{s-1}}\lesssim_{s}\|e^{\ii\theta}\|_{H^s}\lesssim_s\|\sqrt{\rho}e^{\ii\theta}\|_{H^s}\|\rho^{-1/2}\|_{L^{\infty}}+\|\sqrt{\rho}e^{\ii\theta}\|_{L^{\infty}}\|\rho^{-1/2}\|_{H^s}\lesssim_{s,m} \|u\|_{H^s},$$
where we used \eqref{eq:vt} in the first step, the Leibniz rule \eqref{eq:bil_sym} in the second step, and estimate \eqref{eq:erum} in the last step.   

The bound above, together with \eqref{eq:est-rho-sm}, yields the former inequality in \eqref{eq:equi}. We are left to prove the latter. In fact, we have
\begin{equation*}
\begin{split}
\|u\|_{H^s}&=\|\sqrt{\rho}e^{\ii\theta}\|_{H^s}\lesssim_{s}\|\sqrt{\rho}\|_{H^s}\|e^{\ii\theta}\|_{L^{\infty}}+\|\sqrt{\rho}\|_{L^{\infty}}\|e^{\ii\theta}\|_{H^s}\\
&\lesssim_s \|\sqrt{\rho}\|_{H^s}+\|\sqrt{\rho}\|_{L^{\infty}}(1+\|v\|_{H^{s-1}})\lesssim_{s,m} \|(\rho,v)\|_{M^s}+\|\sqrt{\rho}\|_{L^{\infty}}\lesssim_{s,m}\|(\rho,v)\|_{M^s},
\end{split}
\end{equation*}
where we used the fractional Leibniz rule \eqref{eq:bil_sym} in the second step, \eqref{eq:tv} in the third step, estimate \eqref{eq:esum} in the fourth step, and the Sobolev embedding $H^s(\T^2)\hookrightarrow L^{\infty}(\T^2)$ in the last step. Estimate \eqref{eq:equi} is proved when $\e=1$. The case $\e\in(0,1)$ then readily follows by
\begin{gather*}
	M^{-1}\|(\rho,v)\|_{M^s}=M^{-1}\|\rho\|_{H^s}+M^{-1}\e\|\nabla\theta\|_{H^{s-1}}\leqslant\|u\|_{H^s},\\
	\|u\|_{H^s}\leqslant M\|\rho\|_{H^s}+M\e^{-1}\|\e\nabla\theta\|_{H^{s-1}}\leqslant M\e^{-1}\|(\rho,v)\|_{M^s}.
\end{gather*}
The proof is complete.
\end{proof}

\begin{proposition}\label{pr:equiv2}
Fix $s>1$, $m>0$, $\e\in(0,1]$ and $\varphi\in\T$. There exist $M:=M(s,m)>1$ and $\delta:=\delta(s,m)>0$ such that the following holds. Given $u=|u|e^{\mathrm{i}\theta}\in H^s(\T^2)$, with
	\begin{equation}\label{eq:small1e}
\|u-\sqrt{m}e^{\ii\varphi}\|_{\ell^1}\langle\|u\|_{H^1}\rangle\leqslant\delta,
	\end{equation}
 and setting $\sqrt{\rho}:=|u|$, $\Lambda:=\e\sqrt{\rho}\,\nabla\theta$, we have the estimate
	\begin{equation}\label{eq:equi2}
		M^{-1}\|(\sqrt{\rho},\Lambda)\|_{M^{s}}\leqslant \|u\|_{H^s}\leqslant M\e^{-1}\|(\sqrt{\rho},\Lambda)\|_{M^{s}}.
	\end{equation}
\end{proposition}

\begin{remark}
For $\delta>0$ sufficiently small, condition \eqref{eq:small1e} is more restrictive than \eqref{eq:clm}, and such smallness assumption turns out to be crucial when estimating $\|u\|_{H^s}$ in terms of the $M^s$-norm of the hydrodynamic state $(\sqrt{\rho},\Lambda)$.
\end{remark}

	\begin{proof}[Proof of Proposition \ref{pr:equiv1}]
		Arguing as before, it is sufficient to prove the thesis when $\e=1$, $\varphi=0$. Let us set $\widetilde{\delta}:=\langle\|u\|_{H^1}\rangle^{-1}\delta\leqslant\delta$. For $\delta\leqslant\frac{\sqrt{m}}{2}$, estimate \eqref{eq:emb} and the bound $\big||a|-|b|\big|\leqslant|a-b|$ give $\frac12\sqrt{m}\leqslant\sqrt{\rho}\leqslant\frac32\sqrt{m}$. The fractional Leibniz rule \eqref{eq:bil_sym} yields also $\rho:=u\overline{u}\in H^s(\T^2)$, with 
	$$\|\rho\|_{H^s}\lesssim \|u\|_{L^{\infty}}\|u\|_{H^s}\lesssim_{s,m}\|u\|_{H^s}.$$	
	Applying Lemma \ref{le:compo}, respectively with $h(x)=\sqrt{x}$ and $h(x)=x^{-1/2}$, we then obtain
	\begin{gather}
		\label{eq:estirho}	\|\sqrt{\rho}\|_{H^s}\lesssim_{s,m} \|u\|_{H^s},\\
		\label{eq:estinv}
		\|\rho^{-\frac12}\|_{H^s}\lesssim_{s,m}\|u\|_{H^s}.
	\end{gather}

Set now $\phi:=e^{\ii\theta}$. A direct computation gives the identity
	\begin{equation}\label{eq:idmade}
		\overline{\phi}\,\nabla u=\nabla\sqrt{\rho}+i\Lambda.
	\end{equation}		
	 Using \eqref{eq:estinv} and the fractional Leibniz rule \eqref{eq:bil_sym} we get
	\begin{equation}\label{eq:bopf}
		\|\phi\|_{H^s}\lesssim \|u\|_{H^s}\|\rho^{-\frac12}\|_{L^{\infty}}+\|\rho^{-\frac12}\|_{H^s}\|u\|_{L^{\infty}}\lesssim_{s,m} \|u\|_{H^s}.
	\end{equation}
	Since $\Lambda=\Im(\overline{\phi}\,\nabla u)$, the asymmetric fractional Leibniz rule \eqref{eq:asbi} then yields
	\begin{equation}\label{eq:mad}
		\|\Lambda\|_{H^{s-1}}\leqslant\|\overline{\phi}\,\nabla u\|_{H^{s-1}}\lesssim_s\|\phi\|_{L^{\infty}}\|u\|_{H^s}+\|\phi\|_{H^s}\|J^{-1}\nabla u\|_{L^{\infty}}.
	\end{equation}
	In addition, observe that
	\begin{equation}\label{eq:lin-wiener}
		\|J^{-1}\nabla u\|_{\ell^1}=\|\langle k\rangle^{-1}k\,\mathcal{F}(u-\sqrt{m})\|_{\ell^1}\lesssim \|u-\sqrt{m}\|_{\ell^1}\leqslant\widetilde{\delta}.
	\end{equation}
	Using \eqref{eq:bopf}, \eqref{eq:lin-wiener} and \eqref{eq:emb}, we deduce from \eqref{eq:mad} the bound
	$$\|\Lambda\|_{H^{s-1}}\lesssim_{s,m} \|u\|_{H^s},$$
	which together with \eqref{eq:estirho} yields the former inequality in \eqref{eq:equi}. We are left to prove the latter. To this aim, we start by observing that
	\begin{gather}
		\label{eq:small_re}\|\Re u-\sqrt{m}\|_{\ell^1}={\textstyle{\frac{1}{2}}}\|(u-\sqrt{m})+(\overline{u}-\sqrt{m})\|_{\ell^1}\leqslant\widetilde{\delta}.\\
		\label{eq:small_im}\|\Im u\|_{\ell^1}={\textstyle{\frac{1}{2}}}\|(u-\sqrt{m})-(\overline{u}-\sqrt{m})\|_{\ell^1}\leqslant\widetilde{\delta}.
	\end{gather}
	Owing to \eqref{eq:small_re}, we can apply Proposition \ref{pr:comp-wiener} with $F(z)=1/z$, $z_0=\sqrt{m}$, so that choosing $\delta$ sufficiently small we get
	\begin{equation}\label{eq:small_inre}\|(\Re\psi)^{-1}\|_{\ell^1}\lesssim_{m}1.
	\end{equation}
	Using \eqref{eq:small_im}, \eqref{eq:small_inre} and the algebra property \eqref{eq:algebra}, we then deduce
	\begin{equation}\label{eq:small_tan}
		\Big\|\frac{\Im\psi}{\Re\psi}\Big\|_{\ell^1}\lesssim_m \widetilde{\delta}.
	\end{equation}
	Since $\Re u\geqslant\frac{\sqrt{m}}{2}>0$, as follow from \eqref{eq:emb} and \eqref{eq:small_re}, we can write $$\phi=e^{i\arctan\big(\frac{\Im u}{\Re u}\big)}.$$ 
	In view of \eqref{eq:small_tan}, we can apply again Proposition \ref{pr:comp-wiener} with $F(z)=e^{i\arctan{z}}-1$, $z_0=0$, and choosing $\delta$ possibly smaller we obtain
	\begin{equation}\label{eq:small_polar}
		\|\phi-1\|_{\ell^1}\lesssim_m\widetilde{\delta}.
	\end{equation}

Next we have
	\begin{equation}\label{eq:bou_hs}
		\begin{split}
			\|u\|_{H^{s}}&\lesssim\|u\|_{L^2}+\|\phi\overline{\phi}\nabla u\|_{H^{s-1}}\\
			&\lesssim_s \|\sqrt{\rho}\|_{L^2}+\|\phi\|_{L^{\infty}}\|\overline{\phi}\,\nabla u\|_{H^{s-1}}+\|\phi\|_{H^s}\|J^{-1}(\overline{\phi}\,\nabla u)\|_{L^{\infty}}\\
			&\lesssim_{s,m} \|\sqrt{\rho}\|_{L^2}+\|\nabla\sqrt{\rho}\|_{H^{s-1}}+\|\Lambda\|_{H^{s-1}}+\|u\|_{H^s}\|J^{-1}(\overline{\phi}\,\nabla u)\|_{L^{\infty}}\\
			&\lesssim \|(\sqrt{\rho},\Lambda)\|_{M^{s}}+\|u\|_{H^s}\|J^{-1}(\overline{\phi}\,\nabla u)\|_{L^{\infty}},
		\end{split}
	\end{equation}
	where we used the asymmetric product estimate \eqref{eq:asbi} in the second step, identity \eqref{eq:idmade} together with \eqref{eq:bopf} in the third step, and the definition \eqref{eq:sob_norm} of $M^s$-norm in the last step. Moreover, in view of \eqref{eq:emb} we have
	\begin{equation}\label{eq:split_uno}
		\|J^{-1}(\overline{\phi}\,\nabla u)\|_{L^{\infty}}\leqslant \|J^{-1}\nabla u\|_{\ell^1}+ \|J^{-1}\big((\overline{\phi}-1)\,\nabla u\big)\|_{\ell^1}.
	\end{equation}
	The first term in the r.h.s.~of \eqref{eq:split_uno} is estimated as in \eqref{eq:lin-wiener}. For the latter, we have
	$$\|J^{-1}\big((\overline{\phi}-1)\nabla u\big)\|_{\ell^1}=\big\|\langle k\rangle^{-1}\big(\mathcal{F}(\overline{\phi}-1)*\widehat{\nabla u}\,\big)\big\|_{\ell^1}$$
	$$\lesssim \|\langle k\rangle^{-1}\|_{\ell^{2,\infty}}\,\|\mathcal{F}(\overline{\phi}-1)\|_{\ell^1}\,\|\widehat{\nabla u}\|_{\ell^2}\lesssim \widetilde{\delta}\|u\|_{H^1},$$
	where we used Young inequality in Lorentz spaces (see e.g.~\cite[Section 1.1]{Danchin}) in the second step, and estimate \eqref{eq:small_polar} in the last step. We then deduce
	\begin{equation}\label{eq:jdw}
		\|J^{-1}(\overline{\phi}\nabla u)\|_{L^{\infty}}\lesssim_m\widetilde{\delta}\langle\|u\|_{H^1}\rangle=\delta.
	\end{equation}
	Combining \eqref{eq:bou_hs} with \eqref{eq:jdw} we get
	\begin{equation}\label{eq:ed}\|u\|_{H^s}\lesssim_{s,m} \|(\sqrt{\rho},\Lambda)\|_{M^s}+\delta\|u\|_{H^s},
	\end{equation}
	and choosing $\delta>0$ small enough (depending on $s,m$) so that $\delta\|\psi\|_{H^s}$ can be absorbed in the l.h.s.~of \eqref{eq:ed}, we eventually obtain the desired inequality.
\end{proof}

We are now able to prove our main results.

\begin{proof}[Proof of Theorem \ref{th:main_qhd}]
We start with the case $d=2$. Let $M=M(s,m)$ as in Proposition \ref{pr:equiv1}, and set $\widetilde{\mathcal{K}}:=\e^{-1}M^2\mathcal{K}$. Let $u$ be a global, smooth solution to the NLS equation \eqref{eq:NLS_stand} such that
\begin{equation}\label{eq:tegs}
\frac{\|u(T)\|_{H^s}}{\|u(0)\|_{H^s}}\geqslant \widetilde{\mathcal{K}},
\end{equation}
for some
$$T\lesssim_m \frac{\e}{s-1}\, {\widetilde{\mathcal{K}}}^{\frac{3(1+\sigma)}{s-1}}\lesssim_{s,m}\e\Big(\frac{\mathcal{K}}{\e}\Big)^{\frac{3(1+\sigma)}{s-1}},$$
and
\begin{equation}\label{eq:clmm}
\sup_{t\in[0,T]}\inf_{\varphi\in\T}\|u(t)-\sqrt{m}e^{\ii\varphi}\|_{\ell^1}\lesssim_m\widetilde{\mathcal{K}}^{-\frac{\sigma}{s-1}}\lesssim_{m}\Big(\frac{\mathcal{K}}{\e}\Big)^{-\frac{\sigma}{s-1}},
\end{equation}
as provided by Theorem \ref{thm:growth} (in the case $k=0$).

Let $\rho:=|u|^2$. The bound $\big||a|-|b|\big|\leqslant|a-b|$, \eqref{eq:emb} and \eqref{eq:clmm} yield
\begin{equation}\label{p24}
	\begin{split}
		\sup_{t\in[0,T]}\|\rho(t)-m\|_{L^{\infty}}&\leqslant\sup_{t\in[0,T]}\|\sqrt{\rho}(t)+\sqrt{m}\|_{L^{\infty}}\|\sqrt{\rho}(t)-\sqrt{m}\|_{L^{\infty}}\\
		&\lesssim_m\sup_{t\in[0,T]}\inf_{\varphi\in\T}\|u(t)-\sqrt{m}e^{\ii\varphi}\|_{L^{\infty}}\lesssim_{m}\Big(\frac{\mathcal{K}}{\e}\Big)^{-\frac{\sigma}{s-1}},
	\end{split}
\end{equation}
which proves \eqref{close_weak_norm}. In particular, for $\mathcal{K}$ sufficiently large, we deduce that on the interval $[0,T]$ the mass density $\rho$ is smooth and non-vanishing, and we can write $u=\sqrt{\rho}\,e^{\ii\theta}$ for some smooth $\theta:[0,T]\times \T^2\to\R$. Setting $v:=\e\nabla\theta$, a direct check then shows that the pair $(\rho,v)$ is a smooth solution to system \eqref{eq:QHD} (with $\rho>0$ and $\operatorname{rot}v=0$), defined on $[0,T]\times\T^2$. For convenience of the reader, we explicitly write down the computation. Substituting $u=\sqrt{\rho}\,e^{\ii\theta}$ into the NLS equation \eqref{eq:NLS_stand}, simplifying the factor $e^{\ii\theta}$, and separating real and imaginary parts we obtain
\begin{equation}\label{eq:int-qhd}
\begin{cases}
\partial_t\sqrt{\rho}+\e\nabla\!\sqrt{\rho}\cdot\nabla\theta+\frac{\e}{2}\sqrt{\rho}\Delta\theta=0,\\
\e\sqrt{\rho}\partial_t\theta+{\textstyle{\frac{1}{2}}}\sqrt{\rho}|v|^2+\rho\sqrt{\rho}={\textstyle{\frac{\e^2}{2}}}\Delta\sqrt{\rho}.
\end{cases}
\end{equation}
Multiplying the first equation in \eqref{eq:int-qhd} by $2\sqrt{\rho}$ we get
$$\partial_t\rho+\e\nabla\rho\cdot\nabla\theta+\e\rho\Delta\theta=0,$$ 
namely the continuity equation $\partial_t\rho+\operatorname{div}(\rho v)=0$. In the second identity in \eqref{eq:int-qhd} we can divide by $\sqrt{\rho}$, taking gradients and then multiply by $\rho$, obtaining
\begin{equation}\label{eq:memo}
\rho\partial_t v+{\textstyle{\frac{1}{2}}}\rho\nabla|v|^2+\nabla P(\rho)={\textstyle{\frac{\e^2}{2}}}\rho\nabla\Big(\frac{\Delta\sqrt{\rho}}{\sqrt{\rho}}\Big).
\end{equation}
Using the continuity equation we can rewrite \eqref{eq:memo} as 
$$\partial_t(\rho v)+v\operatorname{div}(\rho v)+{\textstyle{\frac{1}{2}}}\rho\nabla|v|^2+\nabla P(\rho)={\textstyle{\frac{1}{2}}}\rho\nabla\Big(\frac{\Delta\sqrt{\rho}}{\sqrt{\rho}}\Big),$$
which in view of the identity
$$\operatorname{div}(\rho v\otimes v)=v\operatorname{div}(\rho v)+{\textstyle{\frac{1}{2}}}\rho\nabla|v|^2$$
eventually yields the momentum equation in the QHD system \eqref{eq:QHD}.

Finally, for $\mathcal{K}$ large enough, we can exploit the equivalence of norms \eqref{eq:equi}, which combined with estimate \eqref{eq:tegs} gives
$$\frac{\|(\rho,v)(T)\|_{M^s}}{\|(\rho,v)(0)\|_{M^s}}\geqslant \frac{\e}{M^2}\frac{\|u(T)\|_{H^s}}{\|u(0)\|_{H^s}}\geqslant \frac{\e}{M^2}\,\widetilde{\mathcal{K}}=\mathcal{K},$$
proving the growth of Sobolev norms \eqref{eq:growth_K}, and concluding the proof in the case $d=2$.

When $d\geqslant 3$, we proceed as follows. Let $(\widetilde{\rho},\widetilde{v})$ be a smooth solution to the QHD system \eqref{eq:QHD}, defined on $[0,T]\times\T^2$, and satisfying estimates \eqref{eq:growth_K}-\eqref{close_weak_norm}. Let us set
\begin{equation}\label{eq:embed}
\rho(t,x):=\widetilde{\rho}(t,x_1,x_2),\; v(t,x):=(\widetilde{v}(t,x_1,x_2),0,\ldots,0),\quad \,x\in\T^d,
\end{equation}
which defines a smooth solution $(\rho,v)$ to \eqref{eq:QHD} on $[0,T]\times \T^d$. Since $\|(\rho,v)\|_{M^s(\T^d)}=\|(\widetilde{\rho},\widetilde{v})\|_{M^s(\T^2)}$ and $\|\rho-m\|_{L^{\infty}(\T^d)}=\|\widetilde{\rho}-m\|_{L^{\infty}(\T^2)}$, the pair $(\rho,v)$ satisfies the desired estimates \eqref{eq:growth_K}-\eqref{close_weak_norm}.
\end{proof}

\begin{proof}[Proof of Theorem \ref{th:main_cascade}]
	As before, it is sufficient to consider the case $d=2$, the solution in higher dimension being then given by \eqref{eq:embed}. 
	
	Let $M=M(s,m)$ and $\delta_*=\delta(s,m)$ as in Proposition \ref{pr:equiv2}, and set $\widetilde{\mathcal{K}}:=\e^{-1}M^2\mathcal{K}$. Let $u$ be a smooth solution to the NLS equation \eqref{eq:NLS_stand} such that
	\begin{equation}\label{fin:uno} \frac{\|u(T)\|_{H^s}}{\|u(0)\|_{H^s}}\geqslant \widetilde{\mathcal{K}}
		\end{equation}
	for some
	$$ T\lesssim_{m,\delta_*} \e^{-9}\,e^{\widetilde{\mathcal{K}}^{\frac{c}{s-1}}}\lesssim_{s,m} \e^{-9} e^{(\e^{-1}\mathcal{K})^{\frac{c}{s-1}}},$$
	and 
	\begin{gather}
	\label{fin:due}	\sup_{t\in[0,T]}\inf_{\varphi\in\T}\|u(t)-\sqrt{m}e^{\ii\varphi}\|_{\ell^1}\lesssim_m\widetilde{\mathcal{K}}^{-\frac{\sigma}{s-1}}\lesssim_{m}\Big(\frac{\mathcal{K}}{\e}\Big)^{-\frac{\sigma}{s-1}},\\
\label{fin:tre}\sup_{t\in [0, T]} \,\left(\,\inf_{\varphi\in \T} \| u(t)-\sqrt{m}\,e^{\mathrm{i} \varphi} \|_{\ell^1}\right)\,\, \| u(t)\|_{H^1}\leqslant \delta_{*},
\end{gather}
as provided by Theorem \ref{thm:growth}.
	
Using \eqref{fin:due}, and arguing as in the proof of Theorem \ref{th:main_qhd}, we deduce that the Madelung transform $(\rho,v)$ of $u$ defines a smooth solution to the QHD system on $[0,T]\times\T^2$, satisfying estimate \eqref{close_weak_norm_exp}. Moreover, \eqref{fin:due} and \eqref{fin:tre} allow to exploit the equivalence of norms \eqref{eq:equi2}, which in view of \eqref{fin:uno} gives
	$$\frac{\|(\sqrt{\rho},\Lambda)(T)\|_{M^s}}{\|(\sqrt{\rho},\Lambda)(0)\|_{M^s}}\geqslant \frac{\e}{M^2}\frac{\|u(T)\|_{H^s}}{\|u(0)\|_{H^s}}\geqslant \frac{\e}{M^2}\widetilde{\mathcal{K}}=\mathcal{K},$$
which proves \eqref{eq:growth_K_exp} and concludes the proof. 
\end{proof}

\section{Growth of Sobolev norms for the NLS equation}\label{sec:growth}
In this section we prove Theorem \ref{thm:growth} on the Sobolev norms inflation for solutions to the Schr\"odinger equation \eqref{eq:NLS_stand}. Considering as new time variable $\tau=\frac{\e}{2} t$, and renaming it again $t$, we can equivalently consider the defocusing cubic NLS equation
\begin{equation}\label{cubicNLS}
	\mathrm{i} \partial_t u=-\Delta u+2 \e^{-2} |u|^2 u \qquad u=u(t, x),\quad t\in \R, \quad x\in \T^2.
\end{equation}
As is well known, the above equation possesses a Hamiltonian structure given by
\begin{equation}\label{cubicNLSham}
	H(u, \bar{u})= \int_{\T^2} |\nabla u|^2 dx+ \e^{-2} \int_{\T^2} |u|^4\,dx.
\end{equation}
In particular, equation \eqref{cubicNLS} can be rewritten as the following Hamiltonian system
\[
\begin{cases}
	\mathrm{i}\partial_t u=\nabla_{\overline{u}} H(u, \bar{u}),\\
	-\mathrm{i}\partial_t \bar{u}=\nabla_{{u}} H(u, \bar{u}),
\end{cases} 
\] 
where $\nabla_{\overline{u}}=(\partial_{\Re(u)}+\mathrm{i}\partial_{\Im(u)})/2$, $\nabla_{{u}}=(\partial_{\Re(u)}-\mathrm{i}\partial_{\Im(u)})/2$
are the Wirtinger gradients.

The solution to equation \eqref{cubicNLS} with initial datum $\mathtt{p}_{m, k}(0, x)=\sqrt{m} \,e^{\mathrm{i} k\cdot x}$\footnote{Thanks to the phase translation invariance we can assume the amplitude of the plane wave to be real and positive.}, $m>0$, $k\in \Z^2$, is the \emph{plane wave}
\begin{equation}\label{planek}
	\mathtt{p}_{m, k}(t, x)=\sqrt{m}\,e^{\mathrm{i} (k\cdot x-\Omega(k) t)}, \qquad \Omega(k)=|k|^2+2 \e^{-2} m\,.
\end{equation}
For the sake of clarity, we rewrite the statement of Theorem \ref{thm:growth}, taking into account the extra $\e$ factor coming from the time-scaling $t\to\frac{\e}{2}t$ used to derive equation \eqref{cubicNLS}\footnote{Comparing \eqref{planek} with \eqref{eq:plane_wave_original}, we note also that the plane wave is modified according to the time rescaling.}.


\begin{theorem}\label{thm:weak}
Fix $m,\sigma >0$, $\e\in(0,1]$, $s>1$, $k\in\Z^2$. For every $\mathcal{K}>0$ large enough, there exists a global, smooth solution $u(t)$ of the cubic NLS equation \eqref{cubicNLS} on $\T^2$ such that
	\begin{gather}
		\label{bound:growth}
		\frac{\|u(T) \|_{H^s}}{\| u(0) \|_{H^s}}\geqslant \mathcal{K}  \qquad \mathrm{for\,\,\,\,\,some} \qquad 0<T \lesssim_m \frac{\e^2}{s-1}\, \mathcal{K}^{\frac{3(1+\sigma)}{s-1}},\\
\label{bound:growth2}
		\sup_{t\in [0, T]} \,\,\inf_{\varphi\in \T} \| u(t)-\mathtt{p}_{m, k}(0)\,e^{\mathrm{i} \varphi} \|_{\ell^1}\lesssim_m \mathcal{K}^{-\frac{\sigma}{s-1}}.
	\end{gather}
	Moreover, given $\delta_*>0$, there exists a suitable universal constant $c>0$ and a global, smooth solution $u(t)$ to \eqref{cubicNLS}  satisfying
	\begin{equation}\label{bound:growth_exp}
		\frac{\|u(T) \|_{H^s}}{\| u(0) \|_{H^s}}\geqslant \mathcal{K}  \qquad \mathrm{for\,\,\,\,\,some} \qquad 0<T\lesssim_{m,\delta_*} {\e^{-8}}\,e^{\mathcal{K}^{\frac{c}{s-1}}},
	\end{equation}
the bound \eqref{bound:growth2}, and the additional estimate
	\begin{equation}\label{lastbound}
		\sup_{t\in [0, T]} \,\left(\,\inf_{\varphi\in \T} \| u(t)-\mathtt{p}_{m,k}(0)\,e^{\mathrm{i} \varphi} \|_{\ell^1}\right)\,\, \| u(t)\|_{H^1}\leqslant \delta_*.
	\end{equation}
	
\end{theorem}

The rest of this section is devoted to the proof of Theorem \ref{thm:weak}. For convenience of the reader we explain the main strategy and ideas of the proof.

\medskip

{\noindent\textbf{Idea and plan of the proof.} Preliminary, we show that we can reduce to the case of plane waves with Fourier support localized at the zero-th mode. Then, we perform a symplectic reduction of the plane wave to a fixed point. Since we are considering a defocusing NLS equation this equilibrium turned out to be elliptic. This is a suitable setting to apply Birkhoff normal form methods. More precisely, these methods are used to derive an effective finite dimensional model which possesses orbits that move energy to high modes. The solutions of the cubic NLS equation undergoing the prescribed growth of norms are approximations (in the $\ell^1$ norm) of such orbits.

The reduction of the plane wave to a point is obtained by restriction to a certain mass level, which for NLS corresponds to fix the $L^2$ norm of the wave-function, and the passage to suitable coordinates \eqref{coord}. The restricted system is still Hamiltonian with Hamiltonian given in Lemma \ref{lem:hampw}. 

In order to apply a normal form procedure it is useful to finding coordinates in which the linearized problem at the elliptic equilibrium is diagonal. We show that, as a map of the phase space, the diagonalizing change of variables has a particular pseudo differential structure. More precisely, we show that this map is identity plus a smoothing operator. This is one of the key ingredients of the proof and we come back on this later.

Since the aforementioned change of coordinates is symplectic, the Hamiltonian structure is preserved and the new Hamiltonian $\mathcal{H}$ is given in \eqref{HamCal}. We observe that such Hamiltonian is an analytic function with a full jet, because of the presence of the analytic function $\alpha$ in \eqref{def:alpha}. Usually, one would like to normalize the first orders, hence dealing with homogenous Hamiltonian terms of increasing degree at each step of the normal form procedure. This procedure turns out to be not optimal to obtain the improved (polynomial) time estimates \eqref{bound:growth}, hence we need a more refined argument.

 We normalize/eliminate terms in the Hamiltonian which are not homogeneous. This is possible thanks to the special form of the Hamiltonian. The critical term is represented by $\mathcal{K}_1$ in \eqref{calK}. In Lemma \ref{lem:stimeBNF} we prove that this kind of Hamiltonians are closed under the adjoint action $\{ \mathcal{H}^{(2)}, \cdot \}$ and hence they can be normalized/eliminated by constructing changes of coordinates as time-one flow maps of Hamiltonians of the same form. Notably, only $3$-wave resonances are involved in this procedure. 
The main goal of the normal form procedure is to highlight the fact that the effective Hamiltonian terms of degree three and four are close (in some sense) to the resonant Hamiltonian of degree four of the cubic NLS equation, from which one can derive the Toy model system of \cite{CKSTT}.
We achieve our aim by:
\begin{enumerate}
	\item a partial elimination of the Hamiltonian terms of order three.
	\item a partial normalization the Hamiltonian of degree four.
\end{enumerate}
To obtain the polynomial upper bounds \eqref{bound:growth} on the time $T$ one needs good lower bounds for $3$-wave and $4$-wave non resonant interactions. These bounds are needed to estimate the normalizing changes of coordinates and control the size of the remainder of the normal form procedure. These are given in Lemmata \ref{lem:smalldiv}, \ref{lem:smalldiv2}. 
These results are consequence of the properties of the $\Lambda$ set (we refer to Section \ref{sec:Lambdaset} for a detailed discussion on $\Lambda$), an appropriate scaling of this set, by a parameter $q\gg 1$, and the fact that we only need to {partially} normalize the Hamiltonian. These are in fact the main ideas borrowed from \cite{Giu}.
Since the above mentioned bounds depend in a nice way on the scaling parameter $q$, we are able to modulate the size of the remainder of the Birkhoff normal form. 

Actually, one needs also to show that certain Hamiltonian terms, which are not touched by the partial normal form procedure, generate small vector fields and hence they can be directly considered as perturbative terms. Here we use the fact that the abovementioned diagonalizing change of coordinates is identity plus smoothing remainder (see the estimate for $\mathcal{R}_*$ in the proof of Proposition \ref{prop:wbnf}). 

In Birkhoff and rotating coordinates the Hamiltonian has the form
\[
\mathbf{H}=\mathcal{N}+R(t)
\]
where $\mathcal{N}$ is in \eqref{def:N} and $R(t)$ is a time dependent Hamiltonian. The Hamiltonian $\mathcal{N}$ restricted to a suitable finite dimensional subspace $\mathcal{U}_{\Lambda}$ has the same dynamics of the Toy model of \cite{CKSTT} (this is proved in Section \ref{sec:TM}), in particular it possesses an orbit transferring energy to higher modes. Then, most of the remaining part of the proof consists in showing that there are solutions of the full Hamiltonian $\mathbf{H}$ which shadows, in the weak $\ell^1$-norm, the energy cascade orbit of $\mathcal{N}$ (see Section \ref{sec:approx}). Such solutions display the growth of Sobolev norms, namely they satisfy an inequality as \eqref{eq:grka}. 
It remains to come back to the original coordinates and undo the symplectic reduction of the plane wave. Thus, it is required to check that the $H^s$ norms in different coordinates are all equivalent.

Finally, our method is robust enough to construct weakly turbulent solutions satisfying the additional smallness assumption \eqref{waa}, at the cost of passing from polynomial to exponential upper bounds \eqref{bound:growth_exp} for the time at which the prescribed growth is achieved.}

\medskip

\noindent\textbf{Reduction to the case $k=0$.} First we show that it is enough to prove Theorem \ref{thm:weak} in the case $k=0\in \Z^2$.
Using the Fourier representation $u=\sum_{n\in\Z^2} u_n\,e^{\mathrm{i} n \cdot x}$ we can write equation \eqref{cubicNLS} as the following infinite dimensional ODEs system
\[
\begin{cases}
	\mathrm{i} \dot{u}_n=\partial_{\overline{u_n}} H (u, \overline{u})= |n|^2 u_n+2 \e^{-2} \sum_{n_1-n_2+n_3=n} u_{n_1} \overline{u_{n_2}} u_{n_3}\,, \\
	-\mathrm{i} \dot{\overline{u_n}}=\partial_{u_n} H (u, \overline{u})= |n|^2 \overline{u_n}+2 \e^{-2} \sum_{n_1-n_2+n_3=-n} \overline{u_{n_1}} {u_{n_2}} \overline{u_{n_3}}\,.
\end{cases}\qquad n\in \Z^2,
\]
By setting
\begin{equation}\label{psiu}
	\psi_n=u_{n+k}\,e^{\mathrm{i}  (|k|^2+2 k\cdot n) t} \qquad n\in \Z^2,
\end{equation}
the equation for the $\psi_n$'s reads as (we omit the equation for the complex conjugate)
\begin{equation}\label{eq:psi}
	\mathrm{i} \dot{\psi}_n= |n|^2 \psi_n+2 \e^{-2} \sum_{\substack{n_1-n_2+n_3=n}} \psi_{n_1} \overline{\psi_{n_2}} \psi_{n_3}\, \qquad n\in \Z^2.
\end{equation}

We observe that $\psi(t)=\sum_{n\in \Z^2} \psi_n(t)\,e^{\mathrm{i} n\cdot x}$ is a solution of \eqref{cubicNLS} which is localized on the mode $0$ if $u(t)$ is localized on the mode $k$. Hence, from now on, we restrict our analysis to the case of a plane wave \eqref{planek} with $k=0$.

\smallskip

\noindent\textbf{Restriction to a dilated sub-lattice.}  Let $q>1$ be an integer. We consider the subspace
\begin{equation}\label{Sq}
	S_q=\{ \psi\in L^2(\T^2) : \psi(x_1, x_2)=\psi(x_1+2\pi q^{-1}, x_2+2\pi q^{-1}) \},
\end{equation}
that is the set of functions which are $2\pi q^{-1}$ periodic in each variable. We observe that if $\psi\in S_q$ then its Fourier coefficients $\psi_n$ satisfies
\[
\psi_n=0 \qquad \mathrm{if}\,\,n\notin q \Z^2,
\]
namely $\psi$ is Fourier supported on the dilated sub-lattice $q \Z^2$ of $\Z^2$.

It is easy to see that such subspaces are left invariant by the flow of the equation \eqref{eq:psi}, because of the $x$-traslation invariance. 

At some point we will restrict to a subspace as $S_q$ in order to avoid certain $3$-wave resonant interactions. 
Since we will perform several changes of coordinates it is useful to observe that the transformations of the phase space which preserve momentum (equivalently, the $x$-translation invariance) leave invariant the subspaces $S_q$. In particular, we will consider symplectic changes of coordinates $\varphi$ generated as time-one flow map of some Hamiltonian $F$ which Poisson commutes with the momentum Hamiltonians 
\begin{equation}\label{momenta}
	\mathcal{M}_k(\psi, \bar{\psi})=\mathrm{i} \int_{\T^2}  \psi_{x_k}\, \,\overline{\psi}\,dx, \qquad k=1, 2.
\end{equation}
The fact that $\{  F, \mathcal{M}_k \}=0$ for $k=1, 2$ (and $\{ \mathcal{M}_1, \mathcal{M}_2 \}=0$) implies that
\begin{equation}\label{flowscomm}
	\varphi\circ \Phi_{\mathcal{M}_k}^{\tau}\circ \Phi^{s}_{\mathcal{M}_j}=\Phi_{\mathcal{M}_k}^{\tau}\circ \Phi^{s}_{\mathcal{M}_j}\circ \varphi, \qquad \forall k=1, 2,\,\,\forall \tau, s\in \R,
\end{equation}
where $\Phi_{\mathcal{M}_1}^{\tau}(\psi(x))=\psi(x_1+\tau, x_2)$, $\Phi_{\mathcal{M}_2}^{\tau}(\psi(x))=\psi(x_1, x_2+\tau)$ are the flows of the momentum Hamiltonians. 

Therefore, if $\psi\in S_q$ then
\[
\varphi(\psi (\cdot + 2\pi q^{-1}))=\Phi_{\mathcal{M}_1}^{2\pi q^{-1}} \Phi_{\mathcal{M}_2}^{2\pi q^{-1}} (\varphi(\psi(\cdot)))\stackrel{\eqref{flowscomm}}{=}\varphi(\Phi_{\mathcal{M}_1}^{2\pi q^{-1}} \Phi_{\mathcal{M}_2}^{2\pi q^{-1}} \psi(\cdot))\stackrel{\psi\in S_q}{=}\varphi(\psi(\cdot))\,.
\]
This shows that $\varphi$ must preserve the subspaces $S_q$.

\smallskip

\noindent\textbf{Symplectic reduction of the plane wave.}
The equation \eqref{eq:psi} possesses the Hamiltonian function
\begin{equation}\label{hamNLS}
	H(\psi, \overline{\psi})=\int_{\T^2}  |\nabla \psi|^2\,dx+\e^{-2} \int_{\T^2} |\psi|^4\,dx
\end{equation}
associated to the symplectic structure $-\mathrm{i} d\psi\wedge d \overline{\psi}$.

From now on we denote by
\[
H^s=H^s(\T^2; \mathbb{C})
\]
and we will frequently use the identification with the corresponding sequence space that we denote by the same symbol, hence
\begin{equation}\label{def:Hs}
	H^s=\left\{ \psi \colon \Z^2 \to \mathbb{C} : \| \psi\|_s^2=\sum_{n\in \Z^2} |\psi_n|^2\,\langle n \rangle^{2s}<+\infty \right\}.
\end{equation}
We will consider the Hamiltonian functions defined on the phase space
\begin{equation}\label{Rs}
	\mathbf{R}_s:=\{ (\psi^+, \psi^-)\in H^s\times H^s : \psi^+=\overline{\psi^-} \}, \qquad s> 1.
\end{equation}

We introduce a set of coordinates which allows to symplectically reduce the plane wave to a fixed point. We consider the set of variables $(\alpha, \theta, z, \bar{z})$ defined by
\begin{equation}\label{coord}
	\psi_0=\alpha \,e^{\mathrm{i} \theta}, \qquad  \psi-\psi_0=z\,e^{\mathrm{i} \theta},
\end{equation}
where $\alpha\in \R$, $\theta\in \T$. Without loss of generality we can assume that $\alpha>0$. Thus, we can symplectically reduce the plane wave to a fixed point by restricting to the mass level $\{ \| \psi\|_{L^2}^2=4 \pi^2 m\}$, where $m$ is the square of the amplitude of the plane wave. This is equivalent to fix
\begin{equation}\label{def:alpha}
	{2\pi}\alpha={2\pi}\alpha(z)=\sqrt{4\pi^2 {m}-\| z \|_{L^2}^2}.
\end{equation}
We shall check that we remain in the domain of validity of the change of coordinates \eqref{coord}, namely that the solutions that we construct meet the condition $\| z \|_{L^2}^2<4\pi^2 m$ over their lifespan.

As it has been observed in \cite{Faou}, there is a factorization of the dynamics of the variables $(\alpha, \theta)$ and $(z, \bar{z})$ on the mass level. In particular, the evolution of $\alpha$ is determined by \eqref{def:alpha}, while the equation for $\theta$ is given by
\begin{equation}\label{eq:theta}
	\dot{\theta}=-\frac{ \e^{-2}}{2\pi^2 \alpha}\,\Re\left( \sum_{n_1-n_2+n_3=0} z_{n_1}\,\overline{z_{n_2}}\,z_{n_3}  \right),
\end{equation}
where we adopted the convention $z_0=\overline{z_0}=\alpha$.
The system of equations for $(z, \bar{z})$ still possesses a Hamiltonian structure.

%
%


\begin{lemma}\label{lem:hampw}
	The restricted Hamiltonian \eqref{hamNLS} in the coordinates $(z, \bar{z})$ (see \eqref{coord}), with respect to the symplectic form $-\mathrm{i} d z \wedge d \overline{z}$, reads as
	\begin{equation}\label{ham:pw}
		H(z, \bar{z})=H^{(2)}(z, \bar{z})+\e^{-2}\,H^{(4)}(z, \bar{z})+\e^{-2}\,K(z, \bar{z})
	\end{equation}
	where
	\begin{equation}\label{H}
		\begin{aligned}
			H^{(2)}(z, \bar{z})&= \int_{\T^2} |\nabla z|^2 dx+2 {} m \e^{-2}  \| z \|_{L^2}^2+ {2} m  \e^{-2} \int_{\T^2} \Re(z^2)\,dx,\\
			H^{(4)}(z, \bar{z})&=-{4\pi^2}\sum_{n\neq 0} |z_n|^4+{4\pi^2}\sum_{\substack{n_1-n_2+n_3-n_4=0,\\ n_1\neq n_2, n_4}} z_{n_1} \overline{z_{n_2}} z_{n_3} \overline{z_{n_4}}\\
			& + \left( 8\pi^2-\frac{3}{4\pi^2} \right){} \| z \|_{L^2}^4-\frac{1}{2\pi^2} \| z \|_{L^2}^2 \int_{\T^2} \Re(z^2)\,dx,\\
			K(z, \bar{z})&=2{} \alpha  \int_{\T^2} (z+\bar{z}) |z|^2\,dx.
		\end{aligned}
	\end{equation}
	
\end{lemma}
\begin{proof}
	Recalling the Hamiltonian in the original coordinates \eqref{hamNLS}, the thesis follows by the following direct computation
	\begin{align*}
		\int_{\T^2} |\alpha+z|^4 \,dx&={m} \int_{\T^2} (z^2+\overline{z}^2)\,dx +2m \| z \|_{L^2}^2-\frac{3}{4\pi^2} \| z \|_{L^2}^4\\
		&+ 2 \alpha  \int_{\T^2} (z+\bar{z}) |z|^2\,dx-\frac{\| z \|_{L^2}^2}{4\pi^2} \int_{\T^2} (z^2+\bar{z}^2)\,dx+\int_{\T^2} |z|^4\,dx
	\end{align*}
	and the fact that
	\begin{equation}\label{latte}
		\int_{\T^2} |z|^4\,dx=8\pi^2 \| z \|_{L^2}^4-4\pi^2\sum_{n\neq 0} |z_n|^4+4\pi^2\sum_{\substack{n_1-n_2+n_3-n_4=0,\\ n_1\neq n_2, n_4}} z_{n_1} \overline{z_{n_2}} z_{n_3} \overline{z_{n_4}}.
	\end{equation}
	
\end{proof}

\begin{remark}\label{rem:Sq}
	We point out that the Hamiltonian $H$ in \eqref{ham:pw} Poisson commutes with the (reduced) momentum Hamiltonians \eqref{momenta}, and then the subspaces $S_q$ as in \eqref{Sq} are preserved.
\end{remark}

\begin{remark}
	Differently from the NLS Hamiltonian \eqref{hamNLS}, the $L^2$ norm is not a conserved quantity for the Hamiltonian system \eqref{ham:pw}.
\end{remark}

\begin{remark}
	The way we write the term \eqref{latte}, and consequently $H^{(4)}$, turns out to be useful for the derivation of the effective finite dimensional system whose dynamics is, up to Gauge transformations, the same of the Toy model introduced in \cite{CKSTT}. 
\end{remark}

\noindent\textbf{Diagonalization of the linearized vector field.} One of our main goal is to derive  from the Hamiltonian \eqref{ham:pw} an effective finite dimensional system which resembles the dynamics of the Toy model introduced in \cite{CKSTT}. To do that, we will apply a normal form procedure.

It is convenient to find a set of coordinates which makes diagonal (with respect to the Fourier basis) the linear part of the vector field. 

The linearized problem of the Hamiltonian $H$ at the equilibrium $z=0$ on the $(z_n, \overline{z_{-n}})$-directions is determined by the following matrix
\[
A_n= \begin{pmatrix}
	|n|^2+2  m \e^{-2} & 2 {} m \e^{-2}\\
	-2 {} m \e^{-2} & -|n|^2-2  m \e^{-2}
\end{pmatrix}.
\]
The eigenvalues of $A_n$ are $\pm \omega(n)$ with
\begin{equation}\label{freq:pw}
	\omega(n)=\sqrt{ |n|^4+4 {\e^{-2}} m |n|^2} \qquad n\in \Z^2\setminus\{0\}.
\end{equation}
Following \cite{Faou}, the matrices
\begin{equation}\label{matrixSn}
	\begin{aligned}
		S_n=&\frac{1}{\sqrt{(\omega(n)+ |n|^2) (\omega(n)+ |n|^2+4{} m \e^{-2})}}\begin{pmatrix}
			|n|^2+ 2{} m \e^{-2}+\omega(n) & -2 {} m \e^{-2}\\
			-2 {} m \e^{-2} &  |n|^2+ 2{} m \e^{-2}+\omega(n)
		\end{pmatrix},\\
		S^{-1}_n=&\frac{1}{\sqrt{(\omega(n)+  |n|^2) (\omega(n)+ |n|^2+4{} m \e^{-2})}}\begin{pmatrix}
			|n|^2+ 2{} m \e^{-2}+\omega(n) & 2 {} m \e^{-2}\\
			2 {} m \e^{-2} & |n|^2+ 2{} m \e^{-2}+\omega(n)
		\end{pmatrix},
	\end{aligned}
\end{equation}
provide a diagonalizing change of basis for $A_n$, namely
\begin{equation}\label{SAS}
	S_n^{-1} A_n S_n= \begin{pmatrix}
		\omega(n) & 0\\
		0 & -\omega(n)
	\end{pmatrix}.
\end{equation}
Starting from this, we construct a symplectic change of variables on the phase space which makes diagonal the linearized vector field at the equilibrium and we study its pseudo differential structure. This analysis is one of the key ingredient to obtain good estimates on the remainder of the normal form procedure.

First we need some preliminary definitions on linear symplectic operators on $H^s\times H^s$. By abuse of notation, we rename $H^s$ the space \eqref{def:Hs} where the sequences $z$ are defined on $\Z^2\setminus\{0\}$, instead of $\Z^2$.

Given a linear map $A\colon H^s\times H^s \to H^s\times H^s$ we use the following notation
\begin{equation}\label{repre}
	A (z, w)=\sum_{n\in \Z^2\setminus\{0\}} A_n \begin{pmatrix}
		z_n\\ w_n
	\end{pmatrix}\,e^{\mathrm{i n \cdot x}}, \qquad A_n=\begin{pmatrix}
		a^{(1, 1)}_n & a^{(1, 2)}_n\\ a^{(2, 1)}_n & a_{n}^{(2, 2)}
	\end{pmatrix}.
\end{equation}

\begin{definition}
	A linear map $A\colon H^s\times H^s \to H^s\times H^s$ is said to be \emph{real to real} if $A\colon \mathbf{R}_s\to \mathbf{R}_s$ (see \eqref{Rs}) for all $s\geq 0$. 
\end{definition}

\begin{remark}
	On the real subspace $\mathbf{R}=\{ (z, w) : w=\bar{z} \}$ we have
	\[
	A (z, \bar{z})=\sum_{n\in \Z^2\setminus\{0\}} A_n \begin{pmatrix}
		z_n\\ \overline{z_{-n}}
	\end{pmatrix}\,e^{\mathrm{i n \cdot x}}.
	\]
\end{remark}

\begin{remark}
	It is easy to see that, using the representation \eqref{repre}, $A$ is real to real if and only if
	\begin{equation}\label{realtoreal}
		a^{(2, 1)}_n=\overline{a_{-n}^{(1, 2)}}, \qquad a^{(2, 2)}_n=\overline{a_{-n}^{(1, 1)}}.
	\end{equation}
\end{remark}

We denote by $E$ the diagonal matrix
\[
E=\begin{pmatrix}
	\mathrm{I} & 0\\ 0 & -\mathrm{I}
\end{pmatrix},
\]
where $\mathrm{I}\colon H^s\to H^s$ is the identity map. We note that
\[
E (z, w)=\sum_{n\in \Z^2\setminus\{0\}} E_n\,   \begin{pmatrix}
	z_n\\ w_n
\end{pmatrix}\,e^{\mathrm{i} n\cdot x}, \qquad E_n=\begin{pmatrix}
	1 & 0\\
	0 & -1
\end{pmatrix}.
\]

\begin{definition}
	A linear map $A\colon H^s\times H^s\to H^s\times H^s$ is symplectic if 
	\begin{equation}\label{rel:symp}
		A^*\,E\, A= E,
	\end{equation}
	where $A^*$ is the adjoint of $A$.
\end{definition}

\begin{remark}
	By noticing that
	\[
	A^*_n=\begin{pmatrix}
		\overline{a_{n}^{(1, 1)}} & \overline{a_n^{(2, 1)}}\\
		\overline{a_n^{(1, 2)}} & \overline{a_n^{(2, 2)}}
	\end{pmatrix}
	\]
	we see that \eqref{rel:symp} is equivalent to asking that
	\begin{equation}\label{sympl}
		\begin{cases}
			|a_n^{(1, 1)}|^2-|a_n^{(2, 1)}|^2=1,\\
			|a_n^{(2, 2)}|^2-|a_n^{(1, 2)}|^2=1,\\
			\overline{a_n^{(1, 1)}}\,a_n^{(1, 2)}=\overline{a_n^{(2, 1)}} a_n^{(2, 2)},\\
			\overline{a_n^{(1, 2)}} a_n^{(1, 1)}=\overline{a_n^{(2, 2)}} a_n^{(2, 1)} .
		\end{cases}
	\end{equation}
\end{remark}

Now we are in position to state the following proposition, which ensures the existence of a real to real, symplectic change of coordinates of the phase space which makes diagonal (in Fourier basis) the linearized vector field of $H$ at the origin.

\begin{proposition}\label{prop:S}
	Let $s\geq 0$. There exists a real to real, symplectic, linear change of coordinates $S\colon H^s\times H^s\to H^s\times H^s$ such that 
	\begin{equation}\label{quadpar}
		(H^{(2)}\circ S)(w, \overline{w})=\mathcal{H}^{(2)} (w, \overline{w}):=\sum_{n\in \Z^2\setminus\{0\}} 4\pi^2 \omega(n) |w_n|^2,
	\end{equation}
	where $\omega(n)$ are the linear frequencies of the reduced plane wave \eqref{freq:pw}.
	Moreover, $S=\mathrm{I}+\Psi$, where $\mathrm{I}$ is the identity on $H^s\times H^s$ and  $
	\Psi\colon H^s\times H^s\to H^{s+2}\times H^{s+2}$.	The inverse $S^{-1}=\mathrm{I}+\widetilde{\Psi}$, where $\widetilde{\Psi}$ has the same properties of $\Psi$.
\end{proposition}

\begin{remark}
	In the coordinates $(w, \bar{w})$ the equations for the linearized problem at the origin reads as
	\[
	\begin{cases}
		\mathrm{i}\dot{w}_n=\omega(n) w_n\\
		-\mathrm{i} \dot{\overline{w_n}}= \omega(n) \overline{w_n}.
	\end{cases} \qquad n\in \Z^2\setminus\{0\}.
	\]
\end{remark}

\begin{proof}

	We consider the linear change of variables $S$ defined by 
	\[
	(
	f,
	g
	)=
	S (q, p)=\sum_{n\in \Z^2\setminus\{0\}} S_n \begin{pmatrix}
		q_n\\ \overline{p_{-n}}
	\end{pmatrix}\,e^{\mathrm{i} n\cdot x}, \qquad (q, p)\in H^s\times H^s,
	\]
	where the matrix $S_n$ in \eqref{matrixSn} can be written as
	\[
	S_n=\begin{pmatrix}
		\mathtt{d}(n) & \mathtt{e}(n)\\
		\mathtt{e}(n) & \mathtt{d}(n)
	\end{pmatrix}
	\]
	with
	\begin{equation}\label{dnen}
		\begin{aligned}
			\mathtt{d}(n):=\sqrt{1+\frac{4{} m \e^{-2}}{\omega(n)+ |n|^2}}+\mathtt{e}(n), \qquad
			\mathtt{e}(n):=-\frac{2{} m \e^{-2}}{\sqrt{(\omega(n)+ |n|^2)(\omega(n)+ |n|^2+4 {} m \e^{-2})}}.
		\end{aligned}
	\end{equation}

	We observe that
	\begin{equation}\label{ed}
		\mathtt{d}(n), \mathtt{e}(n)\in \R, \qquad \td(n)^2-\te(n)^2=1, \qquad	\mathtt{d}(n)=\mathtt{d}(-n), \qquad \mathtt{e}(n)=\mathtt{e}(-n).
	\end{equation}
	Therefore, the linear map $S$ fulfills the conditions \eqref{realtoreal} and \eqref{sympl} and hence it is real to real and symplectic.
	We now show that
	\begin{equation}\label{giala2}
		|\mathtt{d}(n)-1|\le \frac{2 {} m \e^{-2}}{|n|^2}, \qquad  |\mathtt{e}(n)| \le \frac{2 {} m \e^{-2}}{|n|^2}.
	\end{equation}
	The bound for $\te(n)$ is evident. Then the bound for $\td(n)$ comes from the fact that
	\[
	\left|\sqrt{1+\frac{4{} m \e^{-2}}{\omega(n)+ |n|^2}}-1\right|=\left|\frac{4{} m \e^{-2}}{\omega(n)+ |n|^2}\right|\, \,\left|\frac{1}{1+\sqrt{1+\frac{4{} m \e^{-2}}{\omega(n)+ |n|^2}}}\right|\le \frac{2 {} m \e^{-2}}{|n|^2}.
	\]	
	Since
	\[
	S_n \begin{pmatrix}
		q_n\\ \overline{p_{-n}}
	\end{pmatrix}- \begin{pmatrix}
		q_n\\ \overline{p_{-n}}
	\end{pmatrix}=\begin{pmatrix}
		(\mathtt{d}(n)-1) q_n+\mathtt{e}(n) \overline{p_{-n}}\\
		\mathtt{e}(n) q_n+	(\mathtt{d}(n)-1) \overline{p_{-n}}
	\end{pmatrix},
	\]
	by the bounds \eqref{giala2} we deduce that $S-\mathrm{I}\colon H^s\times H^s\to H^{s+2}\times H^{s+2}$.
	We note that
	\[
	S^{-1}_n=\begin{pmatrix}
		\widetilde{\mathtt{d}}(n) & \widetilde{\mathtt{e}}(n)\\
		\widetilde{\mathtt{e}}(n) & \widetilde{\mathtt{d}}(n)
	\end{pmatrix}
	\]
	with
	$
	\widetilde{\mathtt{d}}(n)=\mathtt{d}(n), \widetilde{\mathtt{e}}(n)=-\mathtt{e}(n).
	$
	Therefore the inverse of $S$ has the same properties of $S$.

	Now we prove \eqref{quadpar}. Recalling \eqref{H}, we can write 
	\begin{equation*}
	\begin{split}
H^{(2)}(z, \bar{z})&=\sum_{n\in \Z^2\setminus\{0\}} ( |n|^2+2m \e^{-2} ) |z_n|^2+\e^{-2} m\sum_{n\in \Z^2\setminus\{0\}} (z_n z_{-n}+\overline{z_n z_{-n}})\\
	&=\frac{1}{2} \sum_{n\in \Z^2\setminus\{0\}} E_n A_n \begin{pmatrix}
		z_n \\ \overline{z_{-n}}
	\end{pmatrix} \cdot \begin{pmatrix}
		\overline{z_n} \\ {z_{-n}}
	\end{pmatrix}.
	\end{split}
	\end{equation*}
	Since $S$ is symplectic we have
	\begin{align*}
		\frac{1}{4\pi^2}\,H^{(2)}(z, \overline{z}) &=\frac{1}{4\pi^2}\,H^{(2)} \circ S (w, \overline{w})=\frac{1}{2} \sum_{n\in \Z^2\setminus\{0\}} S_n^*\,E_n A_n S_n \begin{pmatrix}
			w_n \\ \overline{w_{-n}}
		\end{pmatrix} \cdot \begin{pmatrix}
			\overline{w_n} \\ {w_{-n}}
		\end{pmatrix}\\
		&=\frac{1}{2} \sum_{n\in \Z^2\setminus\{0\}} S_n^*\,E_n A_n S_n \begin{pmatrix}
			w_n \\ \overline{w_{-n}}
		\end{pmatrix} \cdot \begin{pmatrix}
			\overline{w_n} \\ {w_{-n}}
		\end{pmatrix} \\
		&=\frac{1}{2} \sum_{n\in \Z^2\setminus\{0\}} \,E_n\,S_n^{-1} A_n S_n \begin{pmatrix}
			w_n \\ \overline{w_{-n}}
		\end{pmatrix} \cdot \begin{pmatrix}
			\overline{w_n} \\ {w_{-n}}
		\end{pmatrix}\\
		&\stackrel{\eqref{SAS}}{=}\frac{1}{2} \sum_{n\in \Z^2\setminus\{0\}} \,E_n\,\begin{pmatrix}
			\omega(n) & 0\\ 0 & -\omega(n)
		\end{pmatrix} \begin{pmatrix}
			w_n \\ \overline{w_{-n}}
		\end{pmatrix} \cdot \begin{pmatrix}
			\overline{w_n} \\ {w_{-n}}
		\end{pmatrix}\\
		&=\sum_{n\in \Z^2\setminus\{0\}} \omega(n) |w_n|^2,
	\end{align*}
	where we denoted by $S_n^*$ the adjoint of $S_n$.
	
\end{proof}
\begin{remark}
	From the proof of the above theorem we deduce the following relation between the Fourier coefficients of the old and new variables, respectively $(z, \bar{z})$ and $(w, \bar{w})$, which will be useful in the following,
	\begin{equation}\label{change}
		\begin{aligned}
			z_n&=\td(n)\, w_n+\te(n)\, \overline{w_{-n}}, \qquad \overline{z_{-n}}=\te (n) \,w_{n}+\td(n)\, \overline{w_{-n}}, \\
			w_n&=\mathtt{d}(n) z_n-\mathtt{e}(n) \overline{z_{-n}}, \qquad \overline{w_{-n}}=-\te(n) z_n+\td(n)\,\overline{z_{-n}},
		\end{aligned}\qquad n\in \Z^2\setminus\{0\},
	\end{equation}
	where $\mathtt{d}(n), \mathtt{e}(n)$ are defined in \eqref{dnen}. In particular, the $L^2$ norm in the new variables reads as
	\begin{equation}\label{zgw}
	\begin{aligned}
		\| z \|_{L^2}^2&=4\pi^2 \left( \sum_{n\neq 0} (\mathtt{d}(n)^2+\mathtt{e}(n)^2) |w_n|^2+2 \sum_{n\neq 0} \mathtt{d}(n) \mathtt{e}(n) \Re(w_n w_{-n})\right)=:\mathtt{G}(w).\\
		\| w \|_{L^2}^2&=4\pi^2 \left(\sum_{n\neq 0} (\mathtt{d}(n)^2+\mathtt{e}(n)^2) |z_n|^2-2 \sum_{n\neq 0} \mathtt{d}(n) \mathtt{e}(n) \Re(z_n z_{-n})\right)
	\end{aligned}
	\end{equation}
\end{remark}

\medskip

\noindent\textbf{The Hamiltonian structure in the new coordinates.}	By applying the symplectic change of coordinates $S$ given in Proposition \ref{prop:S}, the new Hamiltonian reads as
\begin{equation}\label{HamCal}
	\begin{aligned}
		\mathcal{H}(w, \overline{w})&=(H\circ S)\, (w, \overline{w})=\mathcal{H}^{(2)}(w, \overline{w})+\mathcal{H}^{(4)}(w, \overline{w})+\mathcal{K}(w, \overline{w})\\
		\mathcal{H}^{(2)} (w, \overline{w})&=\sum_{n\in \Z^2\setminus\{0\}} 4\pi^2 \omega(n) |w_n|^2, \qquad \mathcal{H}^{(4)}=\e^{-2}\,H^{(4)}\circ S, \qquad \mathcal{K}=\e^{-2}\,K\circ S.
	\end{aligned}
\end{equation}
The plane wave is now reduce to the origin $w=0$, which is an elliptic equilibrium with linear frequencies of oscillation $\omega(n)$ in \eqref{freq:pw}.

We observe that since $S$ is linear, $\mathcal{H}^{(4)}$ is homogeneous of degree $4$ as $H^{(4)}$. We point out that the Hamiltonian $K$ in \eqref{H} is analytic, but not homogeneous, and it has the following form
\[
K(z, \overline{z})=\alpha(z) \,G^{(3)}(z, \overline{z})=f(\| z \|_{L^2}^2)\, G^{(3)}(z, \overline{z}),
\]
where 
\begin{equation}\label{def:f}
f(x)= \frac{1}{2\pi}\sqrt{4\pi^2 m-x}
\end{equation}
 and $G^{(3)}$ is the homogeneous Hamiltonian of degree $3$ 
\[
G^{(3)}(z, \bar{z})=2{} \int (z+\bar{z}) |z|^2\,dx.
\]
Then, in the new coordinates,
\begin{equation}\label{calK}
	\begin{aligned}
		\mathcal{K}(w, \overline{w}) =\e^{-2} K(S(w, \overline{w}))=\mathcal{K}_1(w, \overline{w})+\mathcal{K}_2(w, \overline{w}),\\
		\mathcal{K}_1(w, \overline{w})=\e^{-2}\,f(\|  w \|_{L^2}^2) \,\mathcal{G}^{(3)}(w, \overline{w}), \qquad \mathcal{K}_2(w, \overline{w})=\e^{-2}\, \Theta(w)\, \mathcal{G}^{(3)}(w, \overline{w}),
	\end{aligned}
\end{equation}
where $\mathcal{G}^{(3)}=G^{(3)}\circ S$ is a homogeneous Hamiltonian of degree $3$ and (recall the definition of $\mathtt{G}$ in \eqref{zgw})
\begin{equation}\label{cincin}
	\begin{aligned}
		\Theta(w)&=f(\mathtt{G}(w))-f(\|  w \|_{L^2}^2)=g(w)\,\,\mathcal{P}^{(2)}(w, \overline{w}),\\
		g(w)&=\left(\int_0^1 f'( (1-\tau)\|  w \|_{L^2}^2+\tau \mathtt{G}(w))\,  \,d\tau\right),\\
		\mathcal{P}^{(2)}(w, \overline{w})&=\mathtt{G}(w)-\| w \|_{L^2}^2\\
		&\stackrel{\eqref{ed}}{=}8\pi^2 \left( \sum_{n\neq 0} \mathtt{e}(n)^2\, |w_n|^2+ \sum_{n\neq 0} \mathtt{d}(n) \mathtt{e}(n) \Re(w_n w_{-n})\right).
	\end{aligned}
\end{equation}

We observe that $g$ is well defined on $\| w\|_{L^2}^2<4\pi^2 m$. Actually, in the following we will work on even smaller domains of $w$.

\begin{remark}
	The vector field of the Hamiltonian $\mathcal{K}_1$ has a zero at the origin of order two. Hence, close to the origin, it is the part of $\mathcal{K}$ affecting most the dynamics.
\end{remark}

The special form of the Hamiltonian $\mathcal{K}$ plays a major role in proving the polynomial bound on the time $T$ in \eqref{bound:growth}. We study the Hamiltonians of this form in Lemma \ref{lem:stimeBNF} below.

\subsection{The \texorpdfstring{$\Lambda$}{} set}\label{sec:Lambdaset}

In this section we state the main properties of the Fourier support of the initial data which gives rise to suitable approximate solutions to \eqref{ham:pw} displaying an energy cascade behavior. This set is usually called $\Lambda$ and it has been first introduced in \cite{CKSTT}. 

In this paper we consider a slight modification of the set $\Lambda$ constructed in \cite{CKSTT, GuardiaHHMP19, Hani}.

The set $\Lambda\subset\Z^2$ can be decomposed as the disjoint union of sets $\Lambda_i$, called \emph{generations},
\[
\Lambda=\Lambda_1\cup\dots\cup \Lambda_{\mathtt{N}}.
\]
Each generation possesses $2^{\mathtt{N}-1}$ elements.
The number of generations $\mathtt{N}$ is a parameter of the problem which will be fixed depending on the prescribed growth to achieve. 

We say that a quartet $(n_1, n_2, n_3, n_4)$ is a \emph{nuclear family} if $n_1, n_3\in \Lambda_i$ and $n_2, n_4\in \Lambda_{i+1}$ for some $i=1, \dots, \mathtt{N}-1$ and they form a non-degenerate rectangle in the $\Z^2$-lattice. 
\begin{remark}
	We observe that the vertices of rectangles in $\Z^2$ are $4$-wave resonances for the NLS equation, but not with respect to the linear frequencies of oscillation $\omega(n)=|n|^2+O(1)$ of the plane wave. 
\end{remark}

We recall the properties of the set $\Lambda$ constructed in \cite{CKSTT, GuardiaHHMP19} and we rephrase an additional property introduced in \cite{Hani} (see also the remark below):
\begin{itemize}
	\item[$(P1)$] (Closure): If $n_1, n_2, n_3\in\Lambda$ are three vertices of a rectangle then the fourth vertex belongs to $\Lambda$ too.
	\item[$(P2)$] (Existence and uniqueness of spouse and children): For each $1\le i\le \mathtt{N}-1$  and every $n_1\in\Lambda_i$ there exists a unique spouse $n_3\in \Lambda_i$ and unique (up to trivial permutations) children $n_2, n_4\in\Lambda_{i+1}$ such that $(n_1, n_2, n_3, n_4)$ is a nuclear family in $\Lambda$.
	\item[$(P3)$] (Existence and uniqueness of parents and sibling): For each $1\le i\le \mathtt{N}-1$  and every $n_2\in\Lambda_{i+1}$ there exists a unique sibling $n_4\in\Lambda_{i+1}$ and unique (up to trivial permutations) parents $n_1, n_3\in\Lambda_{i}$ such that $(n_1, n_2, n_3, n_4)$ is a nuclear family in $\Lambda$.
	\item[$(P4)$] (Non-degeneracy): A sibling of any mode $m$ is never equal to its spouse.
	\item[$(P5)$] (Faithfulness): Apart from nuclear families, $\Lambda$ contains no other rectangles.
	{ \item[$(P6)$] If four points $n_1, n_2, n_3, n_4$ in $\Lambda$ satisfy $n_1-n_2+n_3-n_4=0$ then either the relation is trivial or such points form a family.}
	\item[$(P7)$] There no exist $n_1, n_2\in \Lambda$ such that 
	\[
	\sigma_1 |n_1|^2+\sigma_2 |n_2|^2+\sigma_3 |n_3|^2=0, \qquad \sigma_1, \sigma_2, \sigma_3\in \{\pm\},
	\]
	for any $n_3\in \Z^2$.
	Namely there are no $3$-waves $NLS$ resonant interactions involving two modes in $\Lambda$.
\end{itemize}
\begin{remark}
	The {Property} VIII$_{\Lambda}$ in \cite{Hani} says that no pair of modes from $\Lambda$ form a rectangle with the zero mode. Since we made a symplectic reduction which eliminated the zero mode, this is equivalent to the  property $(P7)$.
\end{remark}

We scale the set $\Lambda\subset \Z^2$ by a large factor $q\in \mathbb{N}$ and we denote the scaled set again with $\Lambda\subset q  \Z^2$. The parameter $q$ will be used to place the rectangles in $\Lambda$ close to $4$-wave resonances of the plane wave (see Lemma \ref{lem:U0}). The following theorem ensures the existence of the set $\Lambda$ satisfying all the properties that we need.

\begin{theorem}\label{thm:gen}
	Fix any $\tilde{\eta}>0$ small and let $s>1$. Then, there exists $\alpha>0$ large enough such that for any $\mathtt{N}>0$ large enough and any $q\in\mathbb{N}$, there exists a set $\Lambda\subset  q  \Z^2$ with
	\[
	\Lambda:=\Lambda_1\cup\dots\cup\Lambda_\mathtt{N},
	\]
	which satisfies conditions $(P1)-(P7)$ and also
	\begin{equation*}
		\frac{\sum_{n\in \Lambda_{\mathtt{N}-2}}|n|^{2s}}{\sum_{n\in \Lambda_{3}}|n|^{2s}}\gtrsim \,2^{(s-1)(\mathtt{N}-4)}\,.
	\end{equation*}
	Moreover, we can ensure that each generation $\Lambda_i$ has $2^{\mathtt{N}-1}$ disjoint frequencies satisfying 
	\begin{equation}\label{bound:gen}
		\frac{\sum_{n\in \Lambda_{j}}|n|^{2s}}{\sum_{n\in \Lambda_{i}}|n|^{2s}}\lesssim\,e^{s \mathtt{N}} ,
	\end{equation}
	for any $1\le i< j\le \mathtt{N}$, and 
	\begin{equation}\label{def:R}
		C^{-1}\,q\, R\le  |{n}| {\le C\,\,q\,3^\mathtt{N}\, R}, \qquad \forall {n}\in {\Lambda}_i, \qquad i=1, \dots, \mathtt{N},
	\end{equation}
	where $C>0$ is independent of $\mathtt{N}$ and
	$R=R(\mathtt{N})$ satisfies
\[
e^{\alpha^\mathtt{N}}\le R \le e^{2(1+\tilde{\eta})\alpha^\mathtt{N}}.
\]

\end{theorem}

{\begin{proof}
	This theorem has been proved first in \cite{CKSTT} for a set satisfying $(P1)-(P5)$. In \cite{GuardiaHHMP19} and \cite{Hani} it has been proved that the set $\Lambda$ can be constructed to satisfy also, respectively, $(P6)$ and $(P7)$. All these properties are invariant under the scaling $n \mapsto q\,n$, $q\in \mathbb{N}$. 
\end{proof}}

\subsection{Normal form}

In this section we partially normalize the Hamiltonian \eqref{HamCal}. In order to identify the terms that we need to remove / normalize we need some preliminary definitions.

\smallskip

By Remark \ref{rem:Sq}, the fact that the change of coordinates $S$ constructed in Proposition \ref{prop:S} is momentum preserving and the discussion below Theorem \ref{thm:weak}, we have that the subspaces $S_q$ in \eqref{Sq} are preserved. Hence we can work in a space of functions whose frequencies are supported on the dilated lattice 
\[
q\,(\Z^2\setminus\{0\})=:\Z^2_q,
\] where $q\in\mathbb{N}$ will be chosen opportunely. We notice that the set $\Lambda$ can be constructed to belong to $\Z^2_q$.

\smallskip

Given a homogenous Hamiltonian $H^{(N)}$ of degree $N$ we write
\[
H^{(N)}(w, \overline{w})=\sum_{\pi(\sigma, n)=0} [H^{(N)}]_{n_1 \dots n_N}^{\sigma_1 \dots \sigma_N } w_{n_1}^{\sigma_1} \dots w_{n_N}^{\sigma_{n_N}}, \qquad \pi(\sigma, n):=\sum_{i=1}^N \sigma_i n_i=0,
\]
where the indices $n_i\in \Z^2_q$, $[H^{(N)}]_{n_1 \dots n_N}^{\sigma_1 \dots \sigma_N }$ are complex coefficients, $\sigma_i\in \{ \pm \}$ and
\[
w_n^{+}:=w_n, \quad w_n^{-}:=\overline{w_n}.
\]
The condition $\pi(\sigma, n)=0$ encodes the information that the Hamiltonian $H^{(N)}$ commutes with the momentum Hamiltonians $\mathcal{M}_k$, $k=1, 2$ in \eqref{momenta}. We also define
\[
\bral H^{(N)} \brar:=\sup_{(\sigma_i, n_i)} | [H^{(N)}]^{\sigma_1 \dots \sigma_{N}}_{n_1 \dots n_{N}}|.
\]

\begin{definition}\label{hom:ham}
Given a homogenous Hamiltonian ${H}^{(N)}$ of degree $N$ we denote by ${H}^{(N, d)}$, with $0\le d\le N$, the terms supported on the set
\begin{equation}\label{def:A}
	\begin{aligned}
		\mathcal{A}_{N, d}:=\Big\{ & (n_1, \dots, n_N)\in (\Z^2_q)^N : 
		\,\, \text{exactly}\,\, d\,\, \text{integer vectors $n_i$ belong to}\,\, (\Z^2_q)\setminus \Lambda \Big\}.
	\end{aligned}
\end{equation}
Thus, ${H}^{(N)}=\sum_{d=0}^N {H}^{(N, d)}$. Moreover, for $0\le k\le N$, we use the following notations
\[
\mathcal{A}_{N, \le k}=\cup_{d=0}^k \mathcal{A}_{N, d}, \qquad \mathcal{A}_{N, \geq k}=\cup_{d=k}^N \mathcal{A}_{N, d}.
\]
We denote by $H^{(N, \le k)}$ and $H^{(N, \geq k)}$ respectively the terms supported on $\mathcal{A}_{N, \le k}$ and $\mathcal{A}_{N, \geq k}$.
\end{definition}

The following remarks are fundamental for the application of the normal form procedure.

\begin{remark}
Since $\mathcal{H}^{(2)}$ defines a diagonal (with respect to the Fourier basis) vector field, the adjoint action $\{ \mathcal{H}^{(2)}, \cdot \}$ is diagonal on monomials. Then, given a Hamiltonian $H^{(N, k)}$, the Hamiltonian $\{ \mathcal{H}^{(2)}, H^{(N, k)}\}$ is of the same type of $H^{(N, k)}$. In other words, $\{ \mathcal{H}^{(2)}, \cdot \}$ leaves invariant the support of the Hamiltonian.
\end{remark}

\begin{remark}\label{rem:finite}
Fixed $\{\sigma_i\}_{i=1}^N$, the set of $N$-tuples $(n_1, \dots, n_N)$ such that
\[
(n_1, \dots, n_N)\in \mathcal{A}_{N, \le 1}, \qquad \sigma_1 n_1+\dots+\sigma_N n_N=0
\]
is finite. Thus, a Hamiltonian of the form $H^{(N, \le 1)}$ defines a finite dimensional system of ODEs.
\end{remark}

\begin{remark}\label{rem:fund}
We point out that the vector field of homogenous Hamiltonians of the form $H^{(N, k)}$ with $k\neq 1$ leaves invariant the finite dimensional subspace
\[
\mathcal{U}_{\Lambda}=\{ w_n=0 \,\,\,n\notin \Lambda \}.
\]
In particular, the vector field of $H^{(N, \geq 2)}$ vanishes on $\mathcal{U}_{\Lambda}$. Hence, the dynamics of the restriction of $H^{(N, 0)}+ H^{(N, \geq 2)}$ to $\mathcal{U}_{\Lambda}$ is determined only by $H^{(N, 0)}$.
\end{remark}

The following lemma is classical, its proof is a consequence of Young's inequality, and will be use repeatedly to estimate the vector field of homogenous Hamiltonians.

\begin{lemma}\label{lem:young}
Let 
\[
F^{(d+1)}(w, \overline{w})=\sum_{\substack{\pi(\sigma, n)=0}} [F^{(d+1)}]^{\sigma_1 \dots \sigma_{d+1}}_{n_1 \dots n_{d+1}}\,w^{\sigma_1}_{n_1}\dots w^{\sigma_{d+1}}_{n_{d+1}}
\]
be a homogenous Hamiltonian of degree $d+1$ preserving momentum. Then
\[
\| X_F^{(d+1)}(w) \|_{\ell^1}\lesssim \bral F^{(d+1)} \brar \| w\|_{\ell^1}^d \qquad \forall w\in \ell^1.
\]
Moreover, for all $w, w'\in \ell^1$,
\begin{align*}
	\| X_{F^{(d+1)}}(w)-X_{F^{(d+1)}}(w') \|_{\ell^1} &\lesssim \bral {F^{(d+1)}} \brar (\| w \|_{\ell^1}+\| w' \|_{\ell^1})^{d-1} \,\| w-w' \|_{\ell^1},\\
	\| X_{F^{(d+1)}}(w)-X_{F^{(d+1)}}(w')-D X_{F^{(d+1)}}(w) [w-w'] \|_{\ell^1} &\lesssim  \bral {F^{(d+1)}} \brar \| w \|_{\ell^1}^{d-2} \| w-w' \|^2_{\ell^1},
\end{align*}
where $D X_{F^{(d+1)}}(w)$ denotes the differential of the vector field at $w$.
\end{lemma}

\noindent\textbf{Resonant interactions.} To provide estimates for the normalizing changes of coordinates $\Gamma$ and for the remainder $\mathcal{R}$ of the normal form procedure of Proposition \ref{prop:wbnf} we need to provide lower bounds for non-resonant $3$-wave and $4$-wave interactions between the linear frequencies of the reduced plane wave $\omega(n)$ in \eqref{freq:pw}. Since we only need to partially normalize the Hamiltonian, it is sufficient to our purposes to consider the interactions among linear frequencies of oscillations $\omega(n_1), \omega(n_2), \dots, \omega(n_k)$ where $(n_1, \dots, n_k)\in\mathcal{A}_{k, 1}$, for $k=3, 4$ (recall \eqref{def:A}). Since the Hamiltonian terms are momentum preserving, the only mode which is out of the $\Lambda$ set is given as a linear combination of the modes in $\Lambda$.

\medskip

We first observe that, by \eqref{freq:pw},
\begin{equation}\label{exp}
\omega(n)-|n|^2=2 {\e^{-2}} m+r_n,  \qquad n\in \Z^2_q,
\end{equation}
where $r_n$ satisfies the bound
\[
|r_n|\lesssim \frac{m^2}{\e^4 |n|^{2}} .
\]
We consider $q>0$ such that for a sufficiently small constant $c_0>0$ (depending on the prescribed growth ratio $\mathcal{K}$)
\begin{equation}\label{cond:qeps}
\delta:=\e^{-1} q^{-1}<c_0.
\end{equation}
{Under this assumption 
\begin{equation}\label{om}
	|\omega(n)|\gtrsim  |n|^2\gtrsim  q^2 \qquad \forall n\in \Z_q^2.
\end{equation}}

\begin{lemma}{($3$-waves interaction)}\label{lem:smalldiv}
Assume \eqref{cond:qeps}.
For any $n_1, n_2\in \Lambda$ and $n_3\notin \Lambda$
\[
|\sigma_1 \omega(n_1)+\sigma_2\omega(n_2)+\sigma_3 \omega(n_3)|\gtrsim  q^2, \qquad \sigma_1, \sigma_2, \sigma_3\in \{\pm\}.
\]
\end{lemma}
\begin{proof}

We define the function
\[
\Omega(m)=\sigma_1 \omega(n_1)+\sigma_2\omega(n_2)+\sigma_3\omega(n_3), \qquad \Omega(0)=\sigma_1 |n_1|^2+\sigma_2 |n_2|^2+\sigma_3 |n_3|^2.
\]
We have that
\[
|\Omega'(m) |=\left|\sigma_1 \frac{2 m |n_1|^2}{ \omega(n_1)}+\sigma_2 \frac{2 m |n_2|^2}{ \omega(n_2)}+\sigma_3 \frac{2 m |n_3|^2}{ \omega(n_3)}\right|\le 6m.
\]
By $(P7)$ the quantity $\Omega(0)\neq 0$. We have that 
\[
\Omega(m)=\Omega(0)+\int_0^m \Omega'(\tau)\,d\tau,
\]
hence, since
\[
0\neq \Omega(0)=|n_1|^2-|n_2|^2+|n_1-n_2|^2=2|n_1|^2-2(n_1, n_2)\in q^2 \Z
\]
we obtain
$$
|\Omega(m)|\geq |\Omega(0)|-6m\geq q^2 \left( 1-\frac{6m}{ q^2} \right)\stackrel{\eqref{cond:qeps}}{\gtrsim} q^2.
$$
\end{proof}
\begin{lemma}{($4$-waves interaction)}\label{lem:smalldiv2}
Assume \eqref{cond:qeps}.
Let $n_1, n_2, n_3\in \Lambda$ be such that $-(n_1-n_2+n_3)\notin \Lambda$. Then
\[
| \omega(n_1)- \omega(n_2)+ \omega(n_3)-\omega(n_1-n_2+n_3)|\gtrsim  q^2.
\]
\end{lemma}

\begin{proof}
We define
\[
\Omega(m)=\omega(n_1)- \omega(n_2)+ \omega(n_3)-\omega(n_1-n_2+n_3), \qquad \Omega(0)=|n_1|^2-|n_2|^2+|n_3|^2-|n_1-n_2+n_3|^2.
\]
Reasoning as in Lemma \ref{lem:smalldiv} we have
\[
|\Omega'(m)|\le 8 m.
\]
We have that
$
\Omega(0)\in q^2\Z
$
and, by the closure property $(P1)$, this quantity is not zero.  Then 
\[
|\Omega(m)|\geq q^2-8 m= q^2 \left( 1-\frac{8 m}{ q^2} \right) \stackrel{\eqref{cond:qeps}}{\gtrsim} q^2.
\]
\end{proof}

If we Taylor expand the function $\alpha=\alpha(z)$ in \eqref{def:alpha} at $z=0$ we obtain an expansion of the Hamiltonian \eqref{HamCal} of the form
\[
\mathcal{H}=\mathcal{H}^{(2)}+\mathcal{H}^{(3)}+\mathcal{H}^{(4)}+\mathcal{H}^{(\geq 5)},
\]
where $\mathcal{H}^{(j)}$ are homogenous of degree $j$, while $\mathcal{H}^{(\geq 5)}$ is analytic with a zero at $z=0$ of order five.

Following the standard Birkhoff normal form procedure one would eliminate the Hamiltonian terms of degree three, and higher. Unfortunately, to obtain the polynomial time estimates \eqref{bound:growth} we cannot follow this strategy, because new \emph{non perturbative} Hamiltonian terms of degree four (and higher) would appear from the normalization step of $\mathcal{H}^{(3)}$. Instead, we need to normalize the whole term $\mathcal{K}_1$ in \eqref{calK}. This is not a homogenous Hamiltonian term, but it has a special structure, which is left invariant by the adjoint action $\{ \mathcal{H}^{(2)}, \cdot\}$ (see Remark \ref{rem:special} below). In the following lemma we analyze the Hamiltonians of this form.

\begin{lemma}\label{lem:stimeBNF}
Let $r>0$ be small enough.
Given a Hamiltonian of the form
\[
F(w, \overline{w})=f(\| w \|_{L^2}^2) \mathcal{F}^{(3)}(w, \overline{w})
\]
where $f$ is defined in \eqref{def:f} and 
\[
\mathcal{F}^{(3)}=\sum_{\substack{\mathcal{A}_{3, \le 1}\\ \pi(\sigma, n)=0}} [\mathcal{F}^{(3)}]^{\sigma_1 \sigma_2 \sigma_3}_{n_1 n_2 n_3} \,w^{\sigma_1}_{n_1} w^{\sigma_2}_{n_2} w^{\sigma_3}_{n_3}
\]
we have
\begin{equation}\label{bound:vec}
	\| X_F(w)\|_{\ell^1}\lesssim \bral \mathcal{F}^{(3)} \brar \| w \|_{\ell^1}^2 \qquad \forall w\in B_{\ell^1}(r).
\end{equation}
Moreover,
let $k\geq 3$ and 
\[
G^{(k)}=\sum_{\pi(\sigma, n)=0} [G^{(k)}]_{n_1 \dots n_k}^{\sigma_1 \dots \sigma_k}\,w_{n_1}^{\sigma_1}\dots w_{n_k}^{\sigma_k}
\]
be a homogeneous Hamiltonian of degree $k$. Then 
\begin{equation}\label{poisson}
	\{ F, G^{(k)} \}=f'(\| w \|_{L^2}^2) \mathcal{F}^{(3)}\, \{ \| w \|_{L^2}^2 , G^{(k)} \} +f(\| w \|_{L^2}^2)\{ \mathcal{F}^{(3)}, G^{(k)} \}
\end{equation}
and the following estimate holds
\begin{equation}\label{bound:vec2}
	\| X_{\{ F, G^{(k)}\}}(w) \|_{\ell^1}\lesssim \bral G^{(k)} \brar\,\bral \mathcal{F} \brar\,\|w\|^k_{\ell^1} \qquad \forall w\in B_{\ell^1}(r).
\end{equation}

\end{lemma}

\begin{remark}\label{rem:special}
We observe that, since $\mathcal{H}^{(2)}$ commutes with the $L^2$ norm, by \eqref{poisson}
\[
\{ F, \mathcal{H}^{(2)} \}=f(\| w \|_{L^2}^2)\{ \mathcal{F}^{(3)}, \mathcal{H}^{(2)} \}.
\]
\end{remark}

\begin{remark}
The restriction on the domain where we give the estimates \eqref{bound:vec}, \eqref{bound:vec2} is required to obtain uniform bounds on the derivatives of $f(x)$ at $x=0$. We remark that this is necessary also to remain in the domain of definition of the variables \eqref{coord}.
\end{remark}

\begin{proof}
The $w_n$-component of the vector field of $F$ is given, up to constant factors, by
\[
\partial_{\overline{w_n}} F=\left( f'(\| w \|_{L^2}^2)\,w_n\right)\, \mathcal{F}^{(3)}+f(\| w \|_{L^2}^2) \partial_{\overline{w_n}}{\mathcal{F}^{(3)}}.
\]
Since, for $r>0$ small enough,
\begin{equation}\label{billy}
	\sup_{w\in B_{\ell^1}(r)} |f(\| w \|_{L^2}^2)|\lesssim 1,\qquad \sup_{w\in B_{\ell^1}(r)} |f'(\| w \|_{L^2}^2)|\lesssim 1
\end{equation}
and by Lemma \ref{lem:young}
\[
|\mathcal{F}^{(3)}|\lesssim \bral \mathcal{F}^{(3)}\brar \| w \|^3_{\ell^1}, \qquad |\partial_{\overline{w_n}}\mathcal{F}^{(3)}|\le \|X_{\mathcal{F}^{(3)}} \|_{\ell^1}\lesssim \bral \mathcal{F}^{(3)}  \brar \| w \|^2_{\ell^1},
\]
we obtain the bound \eqref{bound:vec}.

The expression \eqref{poisson} comes from a direct computation.
The $w_n$-component of the vector field of $\{F, G\}$ is given, up to factors, by
\begin{align*}
	\partial_{\overline{w_n}}  \{ F, G^{(k)}\}&=f^{''}(\| w \|_{L^2}^2) w_n \{ \| w \|_{L^2}^2 , G^{(k)} \} \mathcal{F}^{(3)}+f'(\| w \|_{L^2}^2)\,\left(\partial_{\overline{w_n}} \{ \| w \|_{L^2}^2 , G^{(k)} \}\right) \mathcal{F}^{(3)}\\
	&+f'(\| w \|_{L^2}^2) \{ \| w \|_{L^2}^2 , G^{(k)} \} \left(\partial_{\overline{w_n}}\mathcal{F}^{(3)}\right)+f'(\| w \|_{L^2}^2)\,w_n\{ \mathcal{F}^{(3)}, G^{(k)} \}\\
	&+f(\| w \|_{L^2}^2) \,\left(\partial_{\overline{w_n}}\{ \mathcal{F}^{(3)}, G^{(k)} \}\right) .
\end{align*}
Since
\[
\sup_{w\in B_{\ell^1}(r)} |f''(\| w \|_{L^2}^2)|\lesssim 1
\]
and
\[
| \{ \| w \|_{L^2}^2 , G^{(k)}(w, \overline{w}) \}|\lesssim \bral G^{(k)} \brar \| w \|^k_{\ell^1} \qquad \forall w\in \ell^1
\]
we obtain the estimate \eqref{bound:vec2}.
\end{proof}

The aim of the normal form procedure is to eliminate the first order Hamiltonian terms which do not leave invariant the finite dimensional subspace
\[
\mathcal{U}_{\Lambda}=\{ w_n=0 \,\,\,n\notin \Lambda \}.
\]
Indeed, our goal is to show that the restriction of the Hamiltonian of degree $4$ on this set gives a system displaying the same dynamical features of the \emph{Toy model} introduced in \cite{CKSTT}, in particular the existence of an energy cascade orbit.

We observe that the Hamiltonian of degree four in \eqref{H}, and then in \eqref{HamCal}, is quite different with respect to the Hamiltonian of degree four of the cubic NLS equation, from which the Toy model is derived. For the moment let us focus on the expression of $H^{(4)}$ in \eqref{H} (ignoring the change of coordinates $S$). This differs from the Hamiltonian of the cubic NLS by the presence of the terms
\begin{equation}\label{latte2}
{} \| z \|_{L^2}^4, \qquad2 {} \| z \|_{L^2}^2 \int_{\T^2} \Re(z^2)\,dx.
\end{equation}
The first term in \eqref{latte2} Poisson commutes with the Hamiltonian of the Toy model
\begin{equation}\label{hamTM}
\begin{aligned}
	H_{TM}(z, \overline{z}):&=-{4\pi^2}\sum_{n\in \Lambda} |z_n|^4+{4\pi^2}\sum_{\substack{n_1-n_2+n_3-n_4=0,\\ n_1\neq n_2, n_4\\ n_i\in \Lambda}} z_{n_1} \overline{z_{n_2}} z_{n_3} \overline{z_{n_4}} \\
	&=-{4\pi^2}\sum_{n\in \Lambda} |z_n|^4+{4\pi^2}\sum_{\substack{(n_1, n_2, n_3, n_4)\,\mathrm{is}\\\mathrm{a\,\,nuclear\,\,family}}} z_{n_1} \overline{z_{n_2}} z_{n_3} \overline{z_{n_4}},
\end{aligned}
\end{equation}
where, for the second equality, we used the property $(P6)$. We will see that, for this reason, this term is harmless.

The second term in \eqref{latte2} turns out to be non resonant, hence it can be eliminated by a Birkhoff normal form transformation.

\smallskip

Recalling \eqref{calK}, \eqref{cincin}, we consider the following decomposition
\begin{equation}\label{K1}
\begin{aligned}
	\mathcal{K}_1(w, \overline{w})&=\mathcal{K}_1^{*}(w, \overline{w})+\mathcal{K}_1^{>}(w, \overline{w}),\\
	\mathcal{K}_1^{*}=\e^{-2}\,f(\| \Pi_{\Lambda} \cdot  \|_{L^2}^2)\,\mathcal{G}^{(3, \le 1)}, \qquad &
	\mathcal{K}_1^{>}=\e^{-2}\,f (\| \cdot \|_{L^2}^2) \,\mathcal{G}^{(3)}-\e^{-2}\,f(\| \Pi_{\Lambda} \cdot  \|_{L^2}^2)\,\mathcal{G}^{(3, \le 1)},
\end{aligned}
\end{equation}
\begin{equation}\label{K2}
\begin{aligned}
	\mathcal{K}_2(w, \overline{w})&=\mathcal{K}_2^{*}(w, \overline{w})+\mathcal{K}_2^{>}(w, \overline{w})\\
	\mathcal{K}_2^{*}&=\e^{-2}\,g(w)\,\mathcal{P}^{(2, 0)} \mathcal{G}^{(3, \le 1)}+\e^{-2}\,g(w)\, \mathcal{P}^{(2, 1)} \mathcal{G}^{(3, 0)}\\
	\mathcal{K}_2^{>}&=\e^{-2}\,g(w)\, \mathcal{P}^{(2, 2)} \mathcal{G}^{(3)}+\e^{-2}\,g(w)\,\mathcal{P}^{(2, \le 1)} \mathcal{G}^{(3, \geq 2)}+\e^{-2}\,g(w)\, \mathcal{P}^{(2, 1)} \mathcal{G}^{(3, 1)},
\end{aligned}
\end{equation}
where we denoted by $\Pi_{\Lambda}$ the Fourier projector
\[
\Pi_{\Lambda} w=\sum_{n\in\Lambda} w_n\,e^{\mathrm{i} n \cdot x}.
\]
The terms $\mathcal{K}_1^{>}$, $\mathcal{K}_2^{>}$ preserve the subspace $\mathcal{U}_{\Lambda}$. This fact is not trivial for $\mathcal{K}_1^{>}$, and it will be proved later (see Lemma \ref{lem:inv}).

The term $\mathcal{K}_1^*$ will be eliminated in the normal form procedure. The term $\mathcal{K}_2^*$ will be shown to be perturbative for our analysis. This is because, thanks to the bounds \eqref{giala2} and \eqref{def:R}, the coefficients of $\mathcal{P}^{(2, 0)}$ and $\mathcal{P}^{(2, 1)}$ can be considered small.

\begin{proposition}{(Partial normal form)}\label{prop:wbnf}
There exists a constant $c_0>0$ such that if (recall the definition of $\delta$ in \eqref{cond:qeps})
\[
\delta<c_0
\]
then the following holds.
There exists a constant $C>0$ such that
if $\eta_0\in (0, 1)$ satisfies 
\begin{equation}\label{condsmall}
	C \eta_0<1
\end{equation}
then for all $\eta\in (0, \eta_0)$ the following holds.
There exists a symplectic change of coordinates $\Gamma\colon B_{\ell^1}(\eta)\to B_{\ell^1}(2\eta)$ such that
\[
\mathcal{H}\circ \Gamma=\sum_{n\neq 0} \omega(n) |w_n|^2+\mathcal{K}^{>}+ \widetilde{H}_{TM}+\mathcal{H}^{(4, \geq 2)}+ \mathcal{R}
\]
where (recall \eqref{K1}, \eqref{K2}, \eqref{hamTM})
\begin{align*}
	\mathcal{K}^{>}(w, \overline{w})=\mathcal{K}_1^{>}(w, \overline{w})+\mathcal{K}_2^{>}(w, \overline{w}), \qquad
	 \widetilde{H}_{TM}(w, \overline{w})=H_{TM} (w, \overline{w})- {} \| \Pi_{\Lambda} w \|_{L^2}^4
\end{align*}
and $\mathcal{R}$ satisfies the following estimate
\begin{equation}\label{bound:remainder}
	\| X_{\mathcal{R}}(w)\|_{\ell^1}\lesssim \e^{-2} \delta^{2}  \| w \|_{\ell^1}^3 \qquad \forall w\in B_{\ell^1}(\eta).
\end{equation}
Moreover,
\begin{equation}\label{closetoId}
	\| \Gamma(w)-w\|_{\ell^1}\lesssim \delta^2 \| w \|_{\ell^1}^2 \qquad \forall w\in B_{\ell^1}(\eta).
\end{equation}

\end{proposition}

The crucial point is that, even if the vector field of the remainder $\mathcal{R}$ has a cubic estimate (and not of higher order as in the standard Birkhoff procedure) it involves the quantity $\delta^2$, which can be made arbitrarily small by tuning the internal parameter $q$. This bound is obtained thanks to the fact that we do not expand the Hamiltonian $\mathcal{K}$, but we deal with that in a single step of normal form.

\begin{remark}
The smallness condition \eqref{condsmall} is assumed to ensure that the domain of definition of the Birkhoff map $\Gamma$ is contained in the domain of definition of the coordinates \eqref{coord}. 
\end{remark}



\begin{proof}

\noindent\textbf{Step one.} We first eliminate the term
$
\mathcal{K}_1^{*}
$ from $\mathcal{K}_1$.
We consider a Hamiltonian (recall that $f$ in \eqref{def:f})
\[
F^{(3, \le 1)}(w, \overline{w})=f(\| \Pi_{\Lambda} w \|_{L^2}^2)\,\mathcal{F}^{(3, \le 1)}(w, \overline{w}), \quad \mathcal{F}^{(3, \le 1)}(w, \overline{w})=\sum_{\substack{\mathcal{A}_{3, \le 1}\\ \pi(\sigma, n)=0}} [\mathcal{F}^{(3, \le 1)}]^{\sigma_1 \sigma_2 \sigma_3}_{n_1 n_2 n_3} \,w^{\sigma_1}_{n_1} w^{\sigma_2}_{n_2} w^{\sigma_3}_{n_3}
\]
such that
\begin{equation}\label{latte3}
	\{  F^{(3, \le 1)}, \mathcal{H}^{(2)}\}+\mathcal{K}_1^{*}=f(\| \Pi_{\Lambda} \cdot \|_{L^2}^2) \,\left(\{  \mathcal{F}^{(3, \le 1)}, \mathcal{H}^{(2)}\}+\e^{-2}\,\mathcal{G}^{(3, \le 1)} \right)=0.
\end{equation}
We point out that, since
\begin{equation}\label{nonzero}
	\inf_{w\in B_{\ell^1}(\eta)} |f(\| \Pi_{\Lambda} w \|_{L^2}^2) |>0,
\end{equation}
the homological equation \eqref{latte3} is equivalent to
\[
\{  \mathcal{F}^{(3, \le 1)}, \mathcal{H}^{(2)}\}+\e^{-2}\,\mathcal{G}^{(3, \le 1)}=0.
\]
Hence we set
\[
[\mathcal{F}^{(3, \le 1)}]^{\sigma_1 \sigma_2 \sigma_3}_{n_1 n_2 n_3} = \frac{\e^{-2}\,[G^{(3, \le 1)}]^{\sigma_1 \sigma_2 \sigma_3}_{n_1 n_2 n_3}}{\sigma_1 \omega(n_1)+\sigma_2 \omega(n_2)+\sigma_3 \omega(n_3)}\,, \quad (n_1, n_2, n_3)\in\mathcal{A}_{3, \le 1}, \quad \sigma_1 n_1+\sigma_2 n_2+\sigma_3 n_3=0,
\]
where the denominators never vanish by Lemma \ref{lem:smalldiv}.
We have that $\bral G^{(3, \le 1)} \brar\lesssim 1$, then by using Lemma \ref{lem:smalldiv} we have
\begin{equation}\label{gratta}
	\bral \mathcal{F}^{(3, \le 1)} \brar\lesssim \delta^2.
\end{equation}
By Remark \ref{rem:finite} $F^{(3, \le 1)}$ defines a finite dimensional system of analytic ODEs. Since
\begin{equation}\label{b:vf3}
	\sup_{w\in  B_{\ell^1}(\eta)} \| X_{F^{(3, \le 1)}}(w)\|_{\ell^1}\lesssim \bral \mathcal{F}^{(3, \le 1)} \brar\, \eta^2\lesssim \delta^2 \eta^2 \qquad \forall w\in B_{\ell^1}(\eta),
\end{equation}
we can extend the local time of existence up to $1$ taking $\eta$ small enough.
Let us call $\phi_3$ the time one flow map of $F^{(3, \le 1)}$.
By adding and subtracting the term $\e^{-2} H^{(4, \le 1)}$ (which is part of $H^{(4)}$ in \eqref{H}) and using Lie series, the new Hamiltonian is given by 
\begin{align*}
	\mathcal{H}\circ \phi_3=\mathcal{H}^{(2)}+\mathcal{K}_1^{>}+\mathcal{K}_2^{>}+\e^{-2} H^{(4, \le 1)}+\mathcal{H}^{(4, \geq 2)}+\mathcal{R}_*,
\end{align*}
where $\mathcal{R}_*=\sum_{k=0}^3 \mathcal{R}_{*, k}$
\begin{align*}
	\mathcal{R}_{*, 0}&=\mathcal{K}_2^*,\\
	\mathcal{R}_{*, 1}&=\mathcal{H}^{(4, \le 1)}-\e^{-2} H^{(4, \le 1)},\\
	\mathcal{R}_{*, 2}&=\int_0^1 \{  F^{(3, \le 1)}, \mathcal{K}+\mathcal{H}^{(4)} \}\circ \phi^t_3\,dt,\\
	\mathcal{R}_{*, 3}&=\frac{1}{2}\int_0^1 (1-t) \{  F^{(3, \le 1)}, \{F^{(3, \le 1)}, \mathcal{H}^{(2)} \} \}\circ \phi^t_3\,dt.
\end{align*}

Regarding the term $\mathcal{R}_{*, 0}$, we first note that (recall the definition of $\mathcal{P}^{(2)}$ in \eqref{cincin}) 
\[
|\mathcal{P}^{(2)}(w, \overline{w})|\lesssim\,\| w \|_{L^2}^2.
\]
Thus for all $\tau\in [0, 1]$ and $w\in B_{\ell^1}(\eta)$
\[
|(1-\tau)\| w \|_{L^2}^2+\tau \mathtt{G}(w)|\lesssim \| w \|_{L^2}^2\lesssim\, \eta^2.
\]
By imposing the smallness condition \eqref{condsmall}, using a bound on the derivative of $f$ as in \eqref{billy} and the fact that the $\ell^1$ norm controls the $\ell^2$ norm we can ensure that (recall the definition of $g$ in \eqref{cincin})
\begin{equation}\label{b:g}
	| g(w) |\lesssim 1 \qquad \forall w\in B_{\ell^1}(\eta).
\end{equation}
By \eqref{cincin}, \eqref{giala2} and the fact that $\Lambda\subset \Z_q^2$ we have that, for all $w\in B_{\ell^1}(\eta)$,
\[
\| X_{\mathcal{P}^{(2,0)}} (w)\|_{\ell^1}\lesssim \delta^4 \| w \|_{\ell^1}^2, \qquad \| X_{\mathcal{P}^{(2,1)}} (w)\|_{\ell^1}\lesssim \delta^2 \| w \|_{\ell^1}^2.
\]
Since $\bral \e^{-2} \mathcal{G}^{(3)} \brar \lesssim \e^{-2}$, by Lemma \ref{lem:young} we have, for all $w\in B_{\ell^1}(\eta)$,
\[
\| X_{\mathcal{R}_{*, 0}}(w)\|_{\ell^1}\lesssim \e^{-2} \delta^2 \| w \|_{\ell^1}^5.
\]
Concerning the term $\mathcal{R}_{*, 1}$, let us write $H^{(4, \le 1)}(w, \overline{w})=\sum_{\mathcal{A}_{4, \le 1}} [H^{(4, \le 1)}]_{n_1 \dots n_4}^{\sigma_1 \dots \sigma_4} w_{n_1}^{\sigma_1} \dots w_{n_4}^{\sigma_4}$. Recalling Proposition \ref{prop:S} and \eqref{change}, let us denote by 
\[
q^+_n=\td(n) w_n+\te(n) \overline{w_{-n}}, \qquad {q^-_n}=\te (n) w_{-n}+\td(n) \overline{w_n}.
\]
To simplify the notation we write $C_{n_1 \dots n_4}^{\sigma_1 \dots \sigma_4}=[H^{(4, \le 1)}]_{n_1 \dots n_4}^{\sigma_1 \dots \sigma_4}$.
Then
\begin{align*}
	\e^2\mathcal{R}_{*, 1}&= H^{(4, \le 1)} ( S (w, \overline{w}) )-H^{(4, \le 1)}(w, \overline{w})\\
	&=\sum_{\substack{\mathcal{A}_{4, \le 1}\\\pi(\sigma, n)=0}} C_{n_1 \dots n_4}^{\sigma_1 \dots \sigma_4} q_{n_1}^{\sigma_1} \dots q_{n_4}^{\sigma_4}-\sum_{\substack{\mathcal{A}_{4, \le 1}\\\pi(\sigma, n)=0}} C_{n_1 \dots n_4}^{\sigma_1 \dots \sigma_4} w_{n_1}^{\sigma_1} \dots w_{n_4}^{\sigma_4}\\
	&=\sum_{\substack{\mathcal{A}_{4, \le 1}\\\pi(\sigma, n)=0}} C_{n_1 \dots n_4}^{\sigma_1 \dots \sigma_4} (q_{n_1}^{\sigma_1}-w_{n_1}^{\sigma_1}) q_{n_2}^{\sigma_2} q_{n_3}^{\sigma_3} w_{n_4}^{\sigma_4}+\sum_{\substack{\mathcal{A}_{4, \le 1}\\\pi(\sigma, n)=0}} C_{n_1 \dots n_4}^{\sigma_1 \dots \sigma_4} w_{n_1}^{\sigma_1} (q_{n_2}^{\sigma_2}-w_{n_2}^{\sigma_2})  q_{n_3}^{\sigma_3} w_{n_4}^{\sigma_4}\\
	&+\sum_{\substack{\mathcal{A}_{4, \le 1}\\\pi(\sigma, n)=0}} C_{n_1 \dots n_4}^{\sigma_1 \dots \sigma_4} w_{n_1}^{\sigma_1} w_{n_2}^{\sigma_2} (q_{n_3}^{\sigma_3}-w_{n_3}^{\sigma_3})   w_{n_4}^{\sigma_4}+\sum_{\substack{\mathcal{A}_{4, \le 1}\\\pi(\sigma, n)=0}} C_{n_1 \dots n_4}^{\sigma_1 \dots \sigma_4} w_{n_1}^{\sigma_1} w_{n_2}^{\sigma_2}  w_{n_3}^{\sigma_3} (q_{n_4}^{\sigma_4}-w_{n_4}^{\sigma_4}) .
\end{align*}
By \eqref{giala2} we have, for the indices $n_i$ involved in $\mathcal{R}_{*, 1}$ and for all $w\in \ell^1$,
\begin{align*}
	|q_{n_i}^{+}-w_{n_i}^{+}|&\le |\mathtt{d}(n_i)-1|\, |w_{n_i}| +|\mathtt{e}(n_i)|\,|\overline{w_{-n_i}}|\lesssim \delta^2 \| w \|_{\ell^1},\\
	|q_{n_i}^--w_{n_i}^-|&\le  |\mathtt{d}(n_i)-1|\, |\overline{w_{n_i}}| +|\mathtt{e}(n_i)|\,|{w_{-n_i}}|\lesssim \delta^2 \| w \|_{\ell^1}.
\end{align*}
By Young's inequality, for all $w\in B_{\ell^1}(\eta)$,
\[
\| X_{\mathcal{R}_{*, 1}}(w)\|_{\ell^1} \lesssim \e^{-2} \delta^2 \| w \|_{\ell^1}^3.
\]
By Lemma \ref{lem:stimeBNF} (using that the $\ell^1$ norm controls the $\ell^2$ norm and $\eta$ is small enough) and \eqref{gratta} we have the following estimates 
\[
\| X_{\mathcal{R}_{*, 2}}(w)\|_{\ell^1} \lesssim \e^{-2} \delta^2  \| w \|_{\ell^1}^3, \qquad
\| X_{\mathcal{R}_{*, 3}}(w)\|_{\ell^1}\lesssim \e^{-2} \delta^2  \| w \|_{\ell^1}^3, \qquad \forall w\in B_{\ell^1}(\eta).
\]

We conclude that
\[
\| X_{\mathcal{R}_*}(w)\|_{\ell^1}\lesssim \e^{-2} \delta^2  \| w \|_{\ell^1}^3 \qquad \forall w\in B_{\ell^1}(\eta).
\]
\noindent\textbf{Step two.} Recalling the Hamiltonian $H^{(4)}$ in \eqref{H} we consider the following splitting
\[
H^{(4, \le 1)}=H^{(4, 0)}+H^{(4, 1)}
\]
where
\begin{align*}
	H^{(4, 0)}&= \widetilde{H}_{TM}-h^{(4, 0)}, \qquad 
	 \widetilde{H}_{TM}(w, \overline{w})=H_{TM} (w, \overline{w}) +\left(8\pi^2-\frac{3}{4\pi^2}  \right) {} \| \Pi_{\Lambda} w \|_{L^2}^4,\\
	h^{(4, 0)}&=- \frac{1}{2\pi^2} \| w \|_{L^2}^2 \int_{\T^2} \Re(w^2)\,dx=-4\pi^2 \left(\sum_{n\in\Lambda} |w_n|^2\right) \left(\sum_{n\in\Lambda} (w_n w_{-n}+\overline{w_n w_{-n}})\right)
\end{align*}
and
\begin{align*}
	H^{(4, 1)}&=H_I^{(4, 1)}+H_{II}^{(4, 1)}\\
	H_I^{(4, 1)}(w, \overline{w})&={4\pi^2} \sum_{\substack{\mathcal{A}_{4, 1}\\ n_1-n_2+n_3-n_4=0}} w_{n_1} \overline{w_{n_2}} w_{n_3} \overline{w_{n_4}}\\
	H_{II}^{(4, 1)}(w, \overline{w})&=-{4\pi^2} \left(\sum_{n\neq 0, n\in \Lambda} |w_n|^2\right) \left(\sum_{\substack{n\neq 0,\\ n\notin \Lambda\,\,\mathrm{or}\,\,-n\notin\Lambda}} (w_n w_{-n}+\overline{w_n w_{-n}})\right).
\end{align*}
We now eliminate $H^{(4, 1)}$ and $h^{(4, 0)}$. We look for a homogeneous Hamiltonian of degree four
$
G^{(4, \le 1)}=G^{(4, 0)}+G^{(4, 1)}
$
such that
\[
\begin{cases}
	\{  G^{(4, 0)}, \mathcal{H}^{(2)}\}+\e^{-2}\,h^{(4, 0)}=0,\\
	\{  G^{(4, 1)}, \mathcal{H}^{(2)}\}+\e^{-2}\,H^{(4, 1)}=0.
\end{cases}
\]

We consider the splitting $G^{(4,  1)}=G_I^{(4,  1)}+G_{II}^{(4,  1)}$ where
\[
\begin{cases}
	\{  G_I^{(4, 1)}, \mathcal{H}^{(2)}\}+\e^{-2} H_I^{(4, 1)}=0,\\
	\{  G_{II}^{(4, 1)}, \mathcal{H}^{(2)}\}+\e^{-2} H_{II}^{(4, 1)}=0.
\end{cases}
\]
To solve the first equation we define
\[
(G^{(4, 1)}_I)_{n_1 \dots n_4}^{+-+-}=\frac{\e^{-2}\,[H_I^{(4, 1)}]^{+-+-}_{n_1 \dots n_4}}{ \omega(n_1)-\omega(n_2)+\omega(n_3)-\omega(n_4)},  \quad (n_1, n_2, n_3, n_4)\in\mathcal{A}_{4, 1}, \quad \sigma_1 n_1+\dots+\sigma_4 n_4=0.
\]
By Lemma \ref{lem:smalldiv2} the denominators above do not vanish.
Since $\bral H_I^{(4, 1)}\brar\lesssim 1$, by Lemma \ref{lem:smalldiv2} we have that
\[
\bral G^{(4, 1)}_I \brar \lesssim \delta^2.
\]

To solve the second equation we proceed in the same way. The denominators appearing in this case are simply $\pm \omega(n)$ (where we used that $\omega(n)=\omega(-n)$). Hence, by using that $|\omega(n)|\gtrsim \e q^2$ (see \eqref{om}), we have the same estimate on the generator
\[
\bral G^{(4, 1)}_{II}\brar \lesssim \delta^2.
\]
To eliminate $h^{(4, 0)}$ we reason exactly as for $H^{(4, 1)}_{II}$.
Then we conclude that $\bral G^{(4, \le 1) } \brar\lesssim \delta^2$ and
\begin{equation}\label{b:vf4}
	\sup_{w\in  B_{\ell^1}(\eta)} \| X_{G^{(4, \le 1)}}(w)\|_{\ell^1}\lesssim \bral {G}^{(4, \le 1)} \brar\, \eta^3\lesssim \delta^2 \eta^3 \qquad \forall w\in B_{\ell^1}(\eta).
\end{equation}
By Remark \ref{rem:finite} $G^{(4, \le 1)}$ defines a finite dimensional system of analytic ODEs. Arguing as in \emph{step one} we can consider the time-one flow map $\phi_4$ of $G^{(4, \le 1)}$ by taking $\eta$ small enough.
The new Hamiltonian can be computed by Lie series
\begin{align*}
	\mathcal{H}\circ \phi_3\circ \phi_4=\mathcal{H}^{(2)}+\mathcal{K}_1^{>}+\mathcal{K}_2^{>}+\e^{-2} \widetilde{H}_{TM}+\mathcal{H}^{(4, \geq 2)}+\mathcal{R},
\end{align*}

where
\begin{align*}
	\mathcal{R}&=\mathcal{R}_*+\widetilde{\mathcal{R}}+\widehat{\mathcal{R}},\\
	\widetilde{\mathcal{R}}&=\int_0^1 \{ G^{(4, \le 1)}, \e^{-2} H^{(4, \le 1)}+\mathcal{H}^{(4, \geq 2)}+ \mathcal{K}_1^{>}+\mathcal{K}_2^{>}+\mathcal{R}_* \}\circ \phi_4^t\,dt,\\
	\widehat{\mathcal{R}}&=\frac{1}{2} \int_0^1 (1-t) \{ G^{(4, \le 1)}, \{ G^{(4, \le 1)}, \mathcal{H}^{(2)}\}  \} \,\circ \phi_4^t\,dt.
\end{align*}
By Lemma \ref{lem:young} and Lemma \ref{lem:stimeBNF} we have the following estimates, for all $w\in B_{\ell^1}(\eta)$,
\[
\| X_{\widetilde{\mathcal{R}}}(w) \|_{\ell^1}\lesssim \e^{-2} \delta^4 \| w \|_{\ell^1}^4, \qquad \| X_{\widehat{\mathcal{R}}}(w) \|_{\ell^1}\lesssim \e^{-2} \delta^2 \| w \|_{\ell^1}^5.
\]
We conclude that, for $\eta$ small enough,
\[
\| X_{\mathcal{R}}(w)\|_{\ell^1}\lesssim \e^{-2} \delta^2  \| w \|_{\ell^1}^3 \qquad \forall w\in B_{\ell^1}(\eta),
\]
that is the estimate \eqref{bound:remainder}.
We set $\Gamma:=\phi_4 \circ \phi_3$. Then by the bounds \eqref{b:vf3} and \eqref{b:vf4} we obtain the estimate \eqref{closetoId}.

\end{proof}

\subsection{Effective finite dimensional model}\label{sec:TM}
We rename
\begin{equation}\label{def:N}
\begin{aligned}
	\mathcal{N}&=\mathcal{N}_0+\mathcal{N}_2,\\
	\e^2 \mathcal{N}_0&= \widetilde{H}_{TM}=H_{TM}+G \qquad G:=\left(8\pi^2-\frac{3}{4\pi^2}  \right)\| \Pi_{\Lambda} \cdot \|^4_{L^2},, \qquad  \mathcal{N}_2=\mathcal{K}^{>}+\mathcal{H}^{(4, \geq 2)}.
\end{aligned}
\end{equation}

For our purposes, it is convenient to get rid of the quadratic part of the Hamiltonian $\mathcal{H}^{(2)}$ by passing to the rotating coordinates
\begin{equation}\label{def:rotcoord}
w_n=r_n\,e^{\mathrm{i} \omega(n) t} \qquad n\in \Z^2\setminus\{0\}.
\end{equation}
We obtain an equation with a time-dependent Hamiltonian
\begin{equation}\label{bfH}
\mathbf{H}=\mathcal{N}+\mathcal{P}'(t)+\mathcal{Q}'(t)+\mathcal{R}'(t), 
\end{equation}
where
\begin{equation}\label{giugi}
\begin{aligned}
	\mathcal{P}'(t)&=\mathcal{N}_0'(t)-\mathcal{N}_0, \qquad \mathcal{N}_0'(t)=\mathcal{N}_0(r_n\,e^{\mathrm{i} \omega(n) t}),
	\\
	\mathcal{Q}'(t)&=\mathcal{N}_2'(t)-\mathcal{N}_2, \qquad \mathcal{N}_2'(t)=\mathcal{N}_2(r_n\,e^{\mathrm{i} \omega(n) t}),\\
	\mathcal{R}'(t)&=\mathcal{R}(r_n\,e^{\mathrm{i} \omega(n) t}).
\end{aligned}
\end{equation}

We recall the finite dimensional subspace
\[
\mathcal{U}_{\Lambda}=\{ w_n=0\,\,\,n\notin \Lambda \}.
\]
In the next lemma we prove that the Hamiltonian vector field $X_{\mathcal{N}}$ leaves invariant $\mathcal{U}_{\Lambda}$.
\begin{lemma}\label{lem:inv}
Let $\eta_0>0$ be the one given by Proposition \ref{prop:wbnf}. Then, for all $\eta\in (0, \eta_0)$, the following holds.
We have
$
X_{\mathcal{N}}\colon \mathcal{U}_{\Lambda}\cap B_{\ell^1}(\eta)\to  \mathcal{U}_{\Lambda}
$.

\end{lemma}
\begin{proof}
We need to prove that $X_{\mathcal{N}_2}$ vanishes on $\mathcal{U}_{\Lambda}$.

We first show that $X_{\mathcal{K}^{>}}$ vanishes on $\mathcal{U}_{\Lambda}\cap B_{\ell^1}(\eta)$. Recalling \eqref{K1}, \eqref{K2}, we have that
\begin{align*}
	\mathcal{K}^{>}&=\mathcal{K}_1^{>}+\mathcal{K}_2^{>}\\
	\mathcal{K}_1^{>}&=\underbrace{-f(\| \cdot  \|_{L^2}^2)\,\mathcal{G}^{(3, \geq 2)}}_A+\underbrace{[f(\|  \cdot  \|_{L^2}^2)-f(\| \Pi_{\Lambda} \cdot  \|_{L^2}^2)] \,\mathcal{G}^{(3, \le 1)}}_{B}\,.
\end{align*}
The vector field of $\mathcal{K}_2^{>}$ vanishes because it is the product of the function $g$ and homogenous Hamiltonians of degree $5$ of type $H^{(5, \geq 2)}$ (recall Definition \ref{hom:ham} and Remark \ref{rem:fund}).

Now we analyze the term $\mathcal{K}^>_1$. If $n\notin \Lambda$ then
\begin{align*}
	\partial_{\overline{w_n}} A &= f(\|  \cdot  \|_{L^2}^2)\,(\partial_{\overline{w_n}}\,\mathcal{G}^{(3, \geq 2)})+f'(\| \cdot \|_{L^2}^2)\, \mathcal{G}^{(3, \geq 2)} \,w_n,\\
	\partial_{\overline{w_n}} B &=-f'(\| \cdot \|_{L^2}^2) \mathcal{G}^{(3, \le 1)} \,w_n+[f(\| \Pi_{\Lambda} \cdot  \|_{L^2}^2)-f(\| \cdot  \|_{L^2}^2)] \,\partial_{\overline{w_n}} \,\mathcal{G}^{(3, 1)}.
\end{align*}
Therefore, the restriction of $\partial_{\overline{w_n}} A$ to $\mathcal{U}_{\Lambda}$ is zero, because $w_n=0$ and $\partial_{\overline{w_n}}\,\mathcal{G}^{(3, \geq 2)}$ vanishes on $\mathcal{U}_{\Lambda}$. Concerning $\partial_{\overline{w_n}} B$, the first term clearly vanishes because $w_n=0$, while for the second we use that
\[
f(\|  w  \|_{L^2}^2)-f(\| \Pi_{\Lambda} w  \|_{L^2}^2)=\frac{\sum_{n\notin \Lambda} |w_n|^2}{f(\| \Pi_{\Lambda} w  \|_{L^2}^2)+f(\| w  \|_{L^2}^2)}.
\]
The numerator vanishes on $\mathcal{U}_{\Lambda}$, while the denominator never vanishes because of condition \eqref{condsmall} and the fact that $\| w \|_{L^2}^2\le \| w \|^2_{\ell^1}\le \eta^2$. By Remark \ref{rem:fund} the vector field of $\mathcal{H}^{(4, \geq 2)}$ vanishes on $\mathcal{U}_{\Lambda}$. Thus we have the thesis.  
\end{proof}

We want to show that the Hamiltonian system defined by $\mathcal{N}$ possesses an energy cascade orbit, namely a solution that transfers energy from the modes of the first generation to the modes of the last one\footnote{For technical reasons we consider transfers of energy among the modes of $\Lambda_3$ and $\Lambda_{N-2}$ instead of $\Lambda_1$ and $\Lambda_N$}. 
We know from \cite{CKSTT} that this kind of orbit exists for the Hamiltonian system defined by $H_{TM}$ in \eqref{hamTM}. We observe that $\{ H_{TM}, \| \Pi_{\Lambda} \cdot \|_{L^2} \}=0$ and, since $G=h( \| \Pi_{\Lambda} \cdot \|_{L^2})$ for a smooth function $h$, $\{ G,  \| \Pi_{\Lambda} \cdot \|_{L^2}\}=0$.
Therefore $\| \Pi_{\Lambda} \cdot \|_{L^2}$ is a constant of motion of $\mathcal{N}$. Then we will construct an energy cascade orbit for $\mathcal{N}$ by considering a phase shift of the orbit of the Toy model.

We first reduce considerably the number of degrees of freedom of the restriction of $\mathcal{N}$ to $\mathcal{U}_{\Lambda}$ by using a symmetry of the system.

\begin{lemma}[Intragenerational equality]
Consider the subspace
\[
\widetilde{\mathcal{U}_{\Lambda}}:=\left\{ r\in \mathcal{U}_{\Lambda} : r_{n}=r_{n'}\,\,\,\forall n, n'\in \Lambda_j\,\,\text{for some}\,\,j\in \{ 1, \dots, \mathtt{N} \} \right\}
\]
where all the members of a generation take the same value. Then $\widetilde{\mathcal{U}_{\Lambda}}$ is invariant under the flow of $\mathcal{N}$.
\end{lemma}
\begin{proof}
This result has been proved for $H_{TM}$ in \cite{CKSTT}. Hence it is sufficient to show that $\widetilde{\mathcal{U}_{\Lambda}}$ is left invariant also by the flow of $G$. We show that for $G$ we are able to prove even a stronger result.
The equations of motion of $G$ are
\[
\mathrm{i} \dot{r}_n=2 \left(8\pi^2-\frac{3}{4\pi^2}  \right) \left(\sum_{k\in \Lambda} |r_k|^2\right)\,r_n \qquad n\in \Lambda.
\]
Consider any $n, {n'}\in \Lambda$ and set $\zeta:=r_n-r_{n'}$. Then $\dot{\zeta}=2 \left(8\pi^2-\frac{3}{4\pi^2}  \right) \left(\sum_{k\in \Lambda} |r_k|^2 \right)\,\zeta$. Hence, if $r_n=r_{n'}$ at some time, this relation remains true for all time (it can be seen, for instance, by using Gronwall lemma). This shows that the set $\{ r_n=r_{n'}\}$ is left invariant by the flow of $G$. Hence, a fortiori, this holds for the set $\widetilde{\mathcal{U}_{\Lambda}}$.
\end{proof}
We set
\begin{equation*}
b_i=r_n \qquad \text{for any}\qquad n\in \Lambda_i, \quad i=1, \dots, \mathtt{N}.
\end{equation*}
By the time reparametrization $\tau=4\pi^2\e^{-2} t$, the equations of motion of $\mathcal{N}$ restricted on $\widetilde{\mathcal{U}_{\Lambda}}$ reduces to the following system
\begin{equation}\label{eq:toy2}
\dot{b}_i=-\mathrm{i} b_i^2 \overline{b_i}+2\mathrm{i} \overline{b_i} (b^2_{i-1}+b^2_{i+1})+\,2^N \mathrm{i} \left( \sum_{k=1}^N |b_i|^2 \right) \,b_i, \qquad i=1, \dots, \mathtt{N}.
\end{equation}
By the above discussion $\sum_{i=1}^N |b_i|^2$ is a constant of motion for \eqref{eq:toy2}. Hence, by setting
\begin{equation}\label{Bb}
B_i(\tau)=\,e^{ \mathrm{i} (2^N\sum_{i=1}^N |b_i|^2) \tau}\,b_i,  \qquad i=1, \dots, \mathtt{N},
\end{equation}
the \eqref{eq:toy2} transforms into the Toy model system of \cite{CKSTT}, that is
\begin{equation}\label{eq:toy}
\dot{B}_i=-\mathrm{i} B_i^2 \overline{B_i}+2\mathrm{i} \overline{B_i} (B^2_{i-1}+B^2_{i+1}), \qquad i=1, \dots, \mathtt{N}.
\end{equation}
The following result concerns the existence of an orbit of \eqref{eq:toy} which displays a forward energy cascade behavior. 
\begin{theorem}{(Theorem $3$, \cite{GuardiaK12})}\label{thm:orbit}
Fix ${\gamma}>0$ large enough. Then for any large enough $\mathtt{N}$ and 
\[
\nu=\exp(-{\gamma} \mathtt{N}),
\]
there exists a trajectory $B(t)$ of the system \eqref{eq:toy}, $\sigma_0>0$ independent of ${\gamma}, \mathtt{N}$ and $T_0>0$ such that
\begin{align*}
	|B_3(0)|>1-\nu^{\sigma_0},& \qquad |B_j(0)|<\nu^{\sigma_0} \quad\,\,\, \mathrm{for}\,\,j\neq 3,\\
	|B_{\mathtt{N}-1}(T_0)|>1-\nu^{\sigma_0},& \qquad |B_j(T_0)|<\nu^{\sigma_0} \quad  \mathrm{for}\,\,j\neq \mathtt{N}-1.
\end{align*}
Moreover $\| B \|_{L^{\infty}}\lesssim 1$ and there exists a constant $\mathbb{K}>0$ independent of $\mathtt{N}$ such that $T_0$ satisfies
\begin{equation}\label{bound:timeT0}
	0<T_0\le  \mathbb{K}\,\mathtt{N}\,\log(\nu^{-1})=\mathbb{K} {\gamma} \mathtt{N}^2.
\end{equation}
\end{theorem}

\begin{remark}\label{rem:polimi}
As noticed in \cite{GuardiaK12}, $\gamma$ can be taken as $\tilde{\gamma} (s-1)$ with $\tilde{\gamma}\gg 1$.
\end{remark}
We consider the solution $b(t)$ of \eqref{eq:toy2} with initial datum $b(0)=B(0)$, where $B(t)$ is the orbit provided by the above theorem. By \eqref{Bb}
\begin{equation}\label{eco}
b(\tau)=e^{-\mathrm{i} \mathtt{M} \tau} B(\tau), \qquad \mathtt{M}:=2^\mathtt{N} \sum_{i=1}^\mathtt{N} |B_i(0)|^2
\end{equation}
and Theorem \ref{thm:orbit} holds by replacing $B(\tau)$ with $b(\tau)$ and system \eqref{eq:toy} with \eqref{eq:toy2}.

We observe that the equation \eqref{eq:toy2} is invariant by the scaling $b(\tau)\mapsto \lambda^{-1} b(\lambda^{-2} \tau)$. Hence, we consider the scaled solution
\begin{equation}\label{lambda}
b^{\lambda}(\tau)=\lambda^{-1} b(\lambda^{-2} \tau), \qquad \lambda^{-1}\ll \eta_0,
\end{equation}
where $\eta_0$ is introduced in Proposition \ref{prop:wbnf} and $b(t)$ is the solution \eqref{eco}.

Then
\begin{equation}\label{def:rlambda}
r^{\lambda}(t, x)=\sum_{n\in \Z_q^2} r^{\lambda}_n(t)\,e^{\mathrm{i} n \cdot x}, \qquad r^{\lambda}_n(t):=\begin{cases}
	b_i^{\lambda}(4\pi^2\e^{-2} t) \qquad n\in\Lambda_i, \quad i=1, \dots, \mathtt{N},\\
	0 \qquad \,\, \,\, \quad n\notin \Lambda
\end{cases}
\end{equation}
is a solution of $\mathcal{N}$ in \eqref{def:N}.
The life-span of $r^{\lambda}(t)$ is $[0, T]$ where
\begin{equation}\label{def:timeT}
T:=\e^2 \lambda^{2} T_0=\e^2 \lambda^2 {} \mathbb{K}\,\gamma\, \mathtt{N}^2,
\end{equation}
where, by an abuse of notation, we absorbed the factor $\frac{1}{4\pi^2}$ into $\mathbb{K}$.
By Theorem \ref{thm:orbit} and arguing as in Lemma $9.2$ in \cite{GuardiaHP16}, we have
\begin{equation}\label{bound:base}
\sup_{t\in [0, T]} \| r^{\lambda}(t) \|_{\ell^1}\lesssim \mathtt{N} 2^{\mathtt{N}-1} \lambda^{-1}.
\end{equation}

Later we will impose that $\lambda$ is such that 
\[
\mathtt{N} 2^{\mathtt{N}-1} \lambda^{-1}\ll \eta_0
\]
in order to remain in the domain of definition of the Birkhoff coordinates given in Proposition \ref{prop:wbnf}. 

\subsection{Approximation argument}\label{sec:approx} In this section we show that solutions of the Hamiltonian system \eqref{bfH} with initial data $\ell^1$-close to $r^{\lambda}(0)$ remain $\ell^1$-close to the trajectory $r^{\lambda}(t)$ over the time interval $[0, T]$, with $T$ in \eqref{def:timeT}. The deviation between these trajectories is essentially caused by the Hamiltonian terms, different by $\mathcal{N}$, appearing in \eqref{bfH}. Since $\mathcal{N}$ is not a resonant Hamiltonian, in the sense that it does not Poisson commute with $\mathcal{H}^{(2)}$, the most dangerous term is $\mathcal{P}'(t)$ in \eqref{bfH}. Actually, $\mathcal{Q}'(t)$ is harmless because its vector field vanishes on the support of $r^{\lambda}$ by Remark \ref{rem:fund}, and $\mathcal{R}'(t)$ has a zero at the origin of high enough order. 

\smallskip

The next lemma shows that $\mathcal{N}$ can be made \emph{close} to resonant by modulating the size of the modes in the $\Lambda$ set through the choice of the parameter $q$.

\medskip

\begin{lemma}\label{lem:U0}
Let $(n_1, n_2, n_3, n_4)\in (\Z^2_q)^4$ be a nuclear family. Then
\[
|\omega(n_1)-\omega(n_2)+\omega(n_3)-\omega(n_4)|\lesssim \e^{-2} \delta^2.
\]
\end{lemma}
\begin{proof}
Since $|n_1|^2-|n_2|^2+|n_3|^2-|n_4|^2=0$, then by \eqref{exp}
\begin{align*}
	|\omega(n_1)-\omega(n_2)+\omega(n_3)-\omega(n_4)|&\le |\omega(n_1)-\omega(n_2)+\omega(n_3)-\omega(n_4)\\
	&- (|n_1|^2-|n_2|^2+|n_3|^2-|n_4|^2)|\lesssim \frac{\e^{-4}}{\min_{i=1, \dots, 4} |n_i|^2}\lesssim \e^{-4} q^{-2}.
\end{align*}

\end{proof}

We are now in position to prove the following.

\begin{proposition}{(Approximation lemma)}\label{prop:approxarg}
Fix $\epsilon\in (0, \frac{1}{4})$ and $\sigma>0$. Then for $\mathtt{N}>0$ large enough the following holds.
Set
\begin{equation}\label{cond:final}
	\lambda=2^{(1+\sigma) \mathtt{N}}, \qquad \delta\le \exp(-2^{(2+\sigma) \mathtt{N}})
\end{equation}
and let $r(t)$ be a solution of $\mathbf{H}$ in \eqref{bfH} such that 
\begin{equation}\label{assump:initial}
	\| r(0)-r^{\lambda}(0)\|_{\ell^{1}}\le \lambda^{-2} \delta.
\end{equation}
Then
\[
\sup_{t\in [0, T]} \| r(t)-r^{\lambda}(t) \|_{\ell^1}\le \lambda^{-1-\epsilon}\,.
\]
\end{proposition}
\begin{proof}
We temporarily use the following notation
\[
\mu:=\mathtt{N} 2^{\mathtt{N}-1}.
\]
We first observe that by the choice of $\lambda$ and \eqref{bound:base}
\[
\sup_{t\in [0, T]}\| r^{\lambda}(t) \|_{\ell^1}\lesssim \lambda^{-1} \mu\lesssim \mathtt{N} 2^{-\sigma \mathtt{N}}.
\]
Thus, taking $\mathtt{N}$ large enough, we can ensure that the above quantity is smaller than $\eta_0$ in Proposition \ref{prop:wbnf}, hence the orbit $r^{\lambda}$ is contained in the domain of definition of the Birkhoff coordinates over the time interval $[0, T]$.
We set $\xi(t)=r(t)-r^{\lambda}(t)$. Then
\[
\dot{\xi}=Z_0+Z_1+Z_2+Z_3 +Z_4,
\]
where, recalling \eqref{giugi},
\begin{align*}
	Z_0&=X_{\mathcal{R}'}(r^{\lambda}+\xi),\\
	Z_1&=D X_{\mathcal{N}}(r^{\lambda}) [\xi],\\
	Z_2&=X_{\mathcal{N}}(r^{\lambda}+\xi)-X_{\mathcal{N}}(r^{\lambda})-D X_{\mathcal{N}}(r^{\lambda}) [\xi],\\
	Z_3&=X_{\mathcal{P}'}(r^{\lambda})+X_{\mathcal{Q}'}(r^{\lambda}),\\
	Z_4&=(X_{\mathcal{P}'}(r^{\lambda}+\xi)-X_{\mathcal{P}'}(r^{\lambda}))+(X_{\mathcal{Q}'}(r^{\lambda}+\xi)-X_{\mathcal{Q}'}(r^{\lambda})).
\end{align*}
By the differential form of Minkowski's inequality we obtain
\[
\frac{d}{d t} \| \xi \|_{\ell^1}\le \|Z_0\|_{\ell^1}+\|Z_1\|_{\ell^1}+\|Z_2\|_{\ell^1}+\|Z_3\|_{\ell^1}+\|Z_4\|_{\ell^1}.
\]
We assume temporarily the following \emph{bootstrap assumption}
\[
\| \xi(t)\|_{\ell^1}\le 2   \lambda^{-(1+\epsilon)} \qquad \forall t\in [0, 
T],
\]
where $T$ is the time introduced in \eqref{def:timeT}.
Note that \eqref{assump:initial} implies that the above inequality holds at least for time $t=0$. A posteriori we will prove 
an improved estimate on the same time interval, and therefore we will drop the bootstrap assumption.

First we need a priori estimates on the terms $Z_i$, $i=1, 2, 3, 4$, defined above.  {We shall 
	use repeatedly Lemma \ref{lem:young}} and the bootstrap assumption. We remark that the change to rotating coordinates 
\eqref{def:rotcoord} does not affect the bounds on the $\ell^1$ norm.

\medskip

\noindent\textbf{Bound for $Z_0$.} By \eqref{bound:base} and the bootstrap assumption
\[
\left\|r^\lambda+\xi\right\|_{\ell^1}\lesssim \mu 
\lambda^{-1}+\lambda^{-(1+\epsilon)}\lesssim \mu \lambda^{-1}.
\]
Hence by the estimate \eqref{bound:remainder} on the Birkhoff remainder
\[
\| Z_0 \|_{\ell^1}\lesssim \e^{-2} \delta^2  \mu^3 
\lambda^{-3}.
\]
\noindent\textbf{Bound for $Z_1$.} By \eqref{bound:base} and considering that $\mathcal{N}$ is a homogenous Hamiltonian of degree $4$ we have 
\[
\| Z_1 \|_{\ell^1}\lesssim \e^{-2}  \mu^2 \lambda^{-2} \| \xi \|_{\ell^1}.
\]

\noindent\textbf{Bound for $Z_2$.} By 
\eqref{bound:base} and using the bootstrap assumption we have
\[
\| Z_2 \|_{\ell^1}\lesssim \e^{-2} 
 \mu \lambda^{-1} \| \xi\|_{\ell^1}^2\lesssim \e^{-2}
\mu\, \lambda^{-2-\epsilon}\, \| \xi \|_{\ell^1}.
\]

\noindent\textbf{Bound for $Z_3$.} We have $X_{\mathcal{Q}'}(r^{\lambda})=0$ because the vector field of $\mathcal{Q}'$ vanishes on $\mathcal{U}_{\Lambda}$. 
Concerning $X_{\mathcal{P}'}(r^{\lambda})$ and recalling the expression of $\mathcal{N}_0$ in \eqref{def:N}, we have that the most dangerous term to bound is, up to constant factors, 
\[
\e^{-2}\sum_{n\in \Lambda} \sum_{\substack{(n_1, n_2, n_3, n)\,\mathrm{is\,a}\\\mathrm{nuclear\,\,family}}} r^{\lambda}_{n_1} \overline{r^{\lambda}_{n_2}} r^{\lambda}_{n_3}\,(1-\exp[\mathrm{i} ( \omega(n_1)-\omega(n_2)+\omega(n_3)-\omega(n)t])\,e^{\mathrm{i} n x}.
\]

Then, by Lemma \ref{lem:U0} and using that $|e^{\mathrm{i} a t}-1|\le |a \,t|$ for all $t, a\in \R$, we have
\[
\| Z_3(t) \|_{\ell^1}\lesssim \e^{-4} \delta^2 \mu^3\,  \lambda^{-3}\,t
\]
for all times $t$ for which $Z_3$ is defined.

\noindent\textbf{Bound for $Z_4$.} By \eqref{bound:base} we have
\[
\| Z_4 \|_{\ell^1}\lesssim   
 \e^{-2} \mu \lambda^{-1}\, \| \xi \|_{\ell^1},
\]
where we have used that $|e^{\mathrm{i} a}-1|\le 2$ for all $a\in \R$. 

We observe that, among $Z_1, Z_2$ and $Z_4$, the latter has the worst estimate.

\medskip

By collecting the previous estimates we have
\[
\frac{d}{dt} \| {\xi}(t) \|_{\ell^1}\le C_0(\e^{-2} \mu \lambda^{-1} \| \xi(t) 
\|_{\ell^1} + \mu^3\, \lambda^{-3}\,\e^{-4}\,\delta^2\,t+\e^{-2} \delta^2  \mu^3 
\lambda^{-3}),
\]
for some constant $C_0>0$. Then, by Gronwall Lemma and  
\eqref{assump:initial},
\begin{align*}
	\| \xi(t) \|_{\ell^1}
	&\le  \left(\lambda^{-2} \delta+ 
	\mu^3\, \lambda^{-3}\,\e^{-4}\,\delta^2\,t^2+\e^{-2} \delta^2  \mu^3 
	\lambda^{-3} t\right)\,\exp 
	\left(C_0\, \mu \e^{-2} \lambda^{-1} t \right).
\end{align*}

Then, for $t\in [0, T]$, where $T=\e^2\, \lambda^2 {}\mathbb{K} \gamma \mathtt{N}^2$ is the time defined in \eqref{def:timeT}, one has  
\begin{align*}
	\| \xi(t) \|_{\ell^1}
	&\le  \left(\lambda^{-2} \delta+ 
	\mathtt{N}^7 2^{3 \mathtt{N}}\, \lambda\,\delta^2\,\mathbb{K}^2 {\gamma}^2+\delta^2  \mathtt{N}^5 2^{3 \mathtt{N}} 
	\lambda^{-1} \mathbb{K} {\gamma}  \right)\,\,\exp(C_0 \mathbb{K} {\gamma} \mathtt{N}^3 2^{(2+\sigma) \mathtt{N}-1}).
\end{align*}
We need to prove that the right hand side of the above inequality is strictly less than $2\lambda^{-(1+\epsilon)}$. Hence we impose the following conditions
\begin{align}
	\delta\lambda^{-2} \exp(C_0 \mathbb{K} {\gamma} \mathtt{N}^3 2^{(2+\sigma) \mathtt{N}-1})\le \frac{1}{2}\,\lambda^{-1-\epsilon}  \label{cond1}\\
	\delta^2\,\mathtt{N}^7 2^{3 \mathtt{N}}\, \lambda\,\mathbb{K}^2 {\gamma}^2\,\exp(C_0 \mathbb{K} {\gamma} \mathtt{N}^3 2^{(2+\sigma) \mathtt{N}-1})\le \frac{1}{2}\,\lambda^{-1-\epsilon}  \label{cond2}\\
	\delta^2  \mathtt{N}^5 2^{3 \mathtt{N}}  \lambda^{-1} \mathbb{K} {\gamma}\,\exp(C_0 \mathbb{K} {\gamma} \mathtt{N}^3 2^{(2+\sigma) \mathtt{N}-1})\le\frac{1}{2}\,\lambda^{-1-\epsilon}\,.  \label{cond3} 
\end{align}
Thanks to the assumptions \eqref{cond:final} the above conditions are satisfied, provided that $\mathtt{N}>0$ is taken large enough. Hence we have the improved estimate
\[
\| \xi( t) \|_{\ell^1}\le \frac{3}{2}  \lambda^{-(1+\epsilon)} \qquad \forall t\in [0, T].
\]
Therefore, we can drop the bootstrap assumption. This concludes the proof.
\end{proof}

\subsection{Conclusion of the proof of Theorem \ref{thm:weak}}

\begin{lemma}\label{lem:ratio}
Fix $s>1$ and take $\gamma$ in Theorem \ref{thm:orbit} as $\gamma=\tilde{\gamma}\, (s-1)$ for $\tilde{\gamma}>0$ large enough. Then, for $\mathtt{N}>0$ large enough and assuming the assumptions of Proposition \ref{prop:approxarg} the following holds.
There exists a solution $w(t)$ of \eqref{HamCal} such that
\begin{equation}\label{ratio}
	\frac{\| w(T) \|_s^2}{\| w(0) \|_s^2}\gtrsim 2^{(s-1) (\mathtt{N}-6)}\,,
\end{equation}
where $T$ is the time introduced in \eqref{def:timeT}. Moreover
\begin{equation}\label{zell1}
	\sup_{t\in [0, T]} \| w(t) \|_{\ell^1}\lesssim \lambda^{-1} \mathtt{N} 2^\mathtt{N}.
\end{equation}
\end{lemma}

\begin{remark}\label{rem:stessodato}
The trajectory $w(t)$ given by the above theorem satisfies $w(0)=r^{\lambda}(0)$, hence the Fourier support of its initial datum is contained in the set $\Lambda$. Moreover, by \eqref{def:R},
\[
\| w(0)\|_s^2=\| r^{\lambda}(0)\|_s^2\sim \lambda^{-2} \,\sum_{n\in \Lambda_3} |n|^{2s}\gtrsim \lambda^{-2} q^{2s} R^{2s} 2^{\mathtt{N}-1}.
\]
\end{remark}

\begin{proof}
The proof follows word by word the proof of Lemma $2.8$ in \cite{Giu} where simply one has to substitute $\mathtt{L}^{1/2}$ with $\delta$. The bound \eqref{zell1} is not stated in Lemma $2.8$ in \cite{Giu}, but it can be deduced from the proof. Indeed, we consider $w(t)$ as the solution of {\eqref{HamCal}} with $w(0)=r^{\lambda}(0)$ and we write
\[
w(t)=\Gamma(\{ r_n\,e^{\mathrm{i} \omega(n) t}  \}_{n\in \Z^2_q}),
\]
where $r(t)=\sum_{n\in \Z^2_q} r_n(t)\,e^{\mathrm{i} n \cdot x}$ is a solution of \eqref{bfH}. Since $\Gamma$ is close to the identity in $\ell^1$ norm (see \eqref{closetoId}) $r(t)$ has initial datum close enough to $r^{\lambda}(0)$ to apply the approximation argument given by Proposition \ref{prop:approxarg}. Hence $r(t)$ satisfies a bound like \eqref{zell1}. Then $w(t)$ satisfies the same bound because the Birkhoff map $\Gamma$ is close to the identity in the $\ell^1$ topology, see \eqref{closetoId}.
\end{proof}

It remains to undo the change of coordinates $S$ given in Proposition \ref{prop:S}.
By \eqref{giala2}, the fact that $\delta\ll 1$ and \eqref{change} we have that
\begin{align*}
|z_n|&=|\mathtt{d}(n) w_n+\mathtt{e}(n) \overline{w_{-n}}|\le |w_n|+|w_{-n}|,
\\
|w_n|&=|\mathtt{d}(n) z_n-\mathtt{e}(n) \overline{z_{-n}}|\le |z_n|+|z_{-n}|.
\end{align*}
Therefore, there exist universal constants $C_1, C_2>0$ such that
\begin{equation}\label{bound:wz}
C_1 \| w \|_s\le \| z \|_s\le C_2 \| w \|_s, \qquad C_1 \| w \|_{\ell^1}\le \| z \|_{\ell^1}\le C_2 \| w \|_{\ell^1}.
\end{equation}
Then, by Lemma \eqref{lem:ratio} there exists a solution of \eqref{ham:pw} satisfying 
\begin{equation}\label{ratio2}
\frac{\| z(T) \|_s^2}{\| z(0) \|_s^2}\gtrsim 2^{(s-1) (\mathtt{N}-6)}\,,
\end{equation}
and
\begin{equation}\label{zell12}
\sup_{t\in [0, T]} \| z(t) \|_{\ell^1}\lesssim \lambda^{-1} \mathtt{N} 2^\mathtt{N}.
\end{equation}

\noindent\textbf{Conclusion of the proof of Theorem \ref{thm:weak}}
Fix $m, \sigma>0$, $s>1$, $\cK> 0$ large enough. We consider $\mathtt{N}>0$ such that
\begin{equation}\label{bobo}
\mathtt{c}\,2^{(s-1) (\mathtt{N}-6)}\,= {\cK^2} \qquad \Rightarrow \qquad \mathtt{N}= \frac{1}{(s-1) \log(2)} \log\left(\frac{\cK^2}{\mathtt{c}}\right)+6,
\end{equation}
for some universal constant $\mathtt{c}>0$.
Note that we can enlarge $\mathtt{N}$ by taking $\cK$ larger.
Set {
\begin{equation}\label{chooseL}
	\lambda=2^{(1+\sigma) \mathtt{N}}=\left(\frac{\mathcal{K}^2}{\mathtt{c}} \right)^{\frac{1+\sigma}{(s-1) \log(2)}} 2^{6 (1+\sigma)}
\end{equation}} 
and $q$ such that 
\[
q\gtrsim \e^{-1} \exp(2^{(2+\sigma) \mathtt{N}}).
\]
In particular, recalling Remark \ref{rem:stessodato}, this bound ensures that
\begin{equation}\label{Poli}
\| w(0)\|^2_s\gtrsim m,
\end{equation}
provided that $\mathtt{N}$ is large enough.
Proposition \ref{prop:approxarg} and
Lemma \ref{lem:ratio} apply, ensuring the existence of a solution $z(t)$ of the Hamiltonian system \eqref{ham:pw} satisfying the norm explosion property \eqref{ratio2}. Recalling \eqref{coord}, we consider the function $\psi(t)$ defined by
\[
\psi(t)=\left(\sqrt{m-\| z(t) \|_{L^2}^2}+z(t)\right)\,e^{\mathrm{i} \theta(t)}, \qquad t\in [0, T],
\]
where $\theta(t)$ solution of \eqref{eq:theta} with $\theta(0)=0$, which is a solution to the cubic NLS equation \eqref{eq:psi}.
Then for $\mathtt{N}$ large enough and by \eqref{zell1}, for all $t\in [0, T]$,
\begin{equation}\label{analisi3}
\begin{aligned}
	\| \psi(t)- \sqrt{m}\,e^{\mathrm{i} \theta(t)}\|_{\ell^1}&=|\sqrt{m-\| z(t) \|^2_{L^2}}-\sqrt{m}|+\| z(t) \|_{\ell^1}=\frac{\| z(t) \|^2_{L^2}}{\sqrt{m-\| z(t) \|^2_{L^2}}+\sqrt{m}}+\| z(t) \|_{\ell^1}\\
	&\lesssim \left(\frac{\| z \|_{L^2}}{\sqrt{m}}+1\right) \lambda^{-1} \mathtt{N} 2^\mathtt{N}\lesssim (\sqrt{m}+1) \mathtt{N} 2^{-\sigma \mathtt{N}}\lesssim  2^{- \frac{\sigma}{2} \mathtt{N}}\lesssim C^{\frac{\sigma}{s-1}} \mathcal{K}^{-\frac{ \sigma \mathtt{N}}{(\mathtt{N}-6)(s-1)}}\\
	&\lesssim \mathcal{K}^{-\frac{ \sigma }{(s-1)}}
\end{aligned}
\end{equation}
where $C>1$ is a universal constant (that we absorbed by taking $N$ large enough). 
By the definition of $\psi(t)$
\begin{equation}\label{bound:zpsi}
\| z(t) \|^2_{H^s}+{m}\geq \| \psi(t) \|^2_{H^s}\geq \| z(t) \|^2_{H^s}.
\end{equation}
By \eqref{Poli} and \eqref{bound:wz}
\[
\| z(0)\|_s^2\gtrsim m.
\]
Then by \eqref{ratio2}, \eqref{bound:zpsi} and the choice of $\mathtt{N}$
\begin{equation}\label{ratiom}
\frac{\| \psi(T) \|_s^2}{\| \psi(0) \|_s^2}\geq \frac{\| z(T)\|_s^2}{\| z(0)\|_s^2+m}\gtrsim \frac{\| z(T)\|_s^2}{\| z(0)\|_s^2} \gtrsim \mathcal{K}^2.
\end{equation}

Let $u(t)$ be the function defined by \eqref{psiu} where $\psi$ is the function defined above. We observe that $\| u(t)\|_{\ell^1}=\| \psi(t) \|_{\ell^1}$ and $\| u(t)\|_{H^s}=\| \psi(t) \|_{H^s}$. Then the bounds \eqref{bound:growth}, \eqref{bound:growth2} hold. Concerning the time $T$ in \eqref{def:timeT}, we have, for $\mathtt{N}$ large enough,
\begin{align*}
T=\lambda^2 T_0&=\e^2 \lambda^2\, {\gamma}\, \mathbb{K}\, \mathtt{N}^2\stackrel{Rem. \ref{rem:polimi}}{=}\e^2 \lambda^2 \tilde{\gamma}\,(s-1)\,\mathbb{K}\,\mathtt{N}^2\\
&\stackrel{\eqref{chooseL}, \eqref{bobo}}{\lesssim} \frac{\e^2}{s-1}\,\left(\frac{\mathcal{K}^2}{\mathtt{c}}\right)^{\frac{2(1+\sigma)}{(s-1) \log(2)}}\,\left(\log\left(\frac{\mathcal{K}^2}{\mathtt{c}}\right)\right)^2\lesssim  \frac{\e^2}{s-1} \, \mathcal{K}^{\frac{3(1+\sigma)}{(s-1)}} .
\end{align*}

We are left to prove the bound \eqref{lastbound}, thus let us fix any constant $\delta_*>0$. We first recall that the $H^1$ norm of the solutions to the cubic NLS equation \eqref{cubicNLS} is controlled for all time by the Hamiltonian \eqref{cubicNLSham}, namely (recall Remark \ref{rem:stessodato})
\begin{align*}
\| u(t)\|_{H^1} &\lesssim\, H(u(0), \bar{u}(0))\lesssim {\e^{-2}} \| u(0)\|_{H^1} \stackrel{\eqref{bound:zpsi}}{\lesssim} {\e^{-2}} \| z(0)\|_{H^1}\stackrel{\eqref{bound:wz}}{\lesssim} {\e^{-2}} \| w(0)\|_{H^1}\lesssim \e^{-2} \| r^{\lambda}(0)\|_{H^1}\\
&\lesssim {\e^{-2}} \| r^{\lambda}(0)\|_{\ell^1} \,\max_{n\in \Lambda}\{ |n| \} \stackrel{\eqref{bound:base}, \eqref{def:R}}{\lesssim} {\e^{-2}} \mathtt{N}\,2^{\mathtt{N}-1}\,3^\mathtt{N}\,\lambda^{-1}\,q\,R.
\end{align*}
Then, recalling \eqref{analisi3}, the bound \eqref{lastbound} is obtained by taking $\lambda>0$ in \eqref{lambda} such that
\[
\left[(\sqrt{m}+1) \lambda^{-1} \mathtt{N} 2^\mathtt{N} \right] {\e^{-2}}  \mathtt{N}\,2^{\mathtt{N}-1}\,3^\mathtt{N}\,\lambda^{-1}\,q\,R\lesssim \delta_* \qquad \Rightarrow \qquad {\e^{-2}}\lambda^{-2} q< \tilde{C}\,\delta_*\, \,e^{-\mathbf{a}^\mathtt{N}}
\]
for some constants $\tilde{C}, \mathbf{a}>0$.
We remark that the scaled time $T$ will be enlarged by the choice of a larger $\lambda>0$. Hence we need to check that the approximation lemma (Proposition \ref{prop:approxarg}) still holds true with this choice of $\lambda$. We note that the conditions \eqref{cond1}, \eqref{cond2}, \eqref{cond3} are satisfied provided that
\[
\lambda^{4/3} {\delta}< C\,e^{-\mathbf{b}^\mathtt{N}}
\]
for some constants $C, \mathbf{b}>0$. The above conditions are compatible, indeed it is sufficient to consider $q$ as 
\[
e^{\mathbf{b}^\mathtt{N}}\,\frac{\lambda^{4/3}}{C}<q<\tilde{C}\,\delta_*\, \,e^{-\mathbf{a}^\mathtt{N}} \lambda^2
\]
and $\lambda=\e^{-5}\,e^{\mathbf{c}^N}$ for $\mathbf{c}>0$ such that
\[
\mathbf{c}^{\mathtt{N}}\geq 2 (\mathbf{a}^\mathtt{N}+\mathbf{b}^\mathtt{N})\,.
\]
The estimate on the time is obtained by using \eqref{def:timeT}.
This concludes the proof of Theorem \ref{thm:weak}.

\appendix

\section{Estimates in Wiener algebra and Sobolev spaces}
In this appendix we collect some useful results, which we used in the proof of Proposition \ref{pr:equiv1} on the equivalence between Sobolev norms of the wave function and the hydrodynamic variables.

We start by recalling that the Wiener space $\ell^1(\T^d)$, defined by \eqref{eq:def_wiener}, is a Banach algebra with respect the pointwise product. Indeed, the Young inequality gives
\begin{equation}\label{eq:algebra}
	\|fg\|_{\ell^1(\T^d)}=\|\widehat{f}*\widehat{g}\|_{\ell^1(\Z^d)}\leqslant\|\widehat{f}\|_{\ell^1(\Z^d)}\|\widehat{g}\|_{\ell^1(\Z^d)}=\|f\|_{\ell^1(\T^d)}\|g\|_{\ell^1(\T^d)}.
\end{equation}
We also point out the basic bound
\begin{equation}\label{eq:emb}
\|f\|_{L^{\infty}(\T^d)}\leqslant\|f\|_{\ell^1(\T^d)}.
\end{equation}
The Wiener algebra behaves well under composition with analytic functions. We have indeed the following result, which is a particular instance of the Levy-Wiener theorem \cite{Wiener,Levy}. For convenience of the reader, we present here a direct proof, based on \cite{Zygmund}.
\begin{proposition}\label{pr:comp-wiener}
	Let $z_0\in\C$, and $h$ be analytic in an open neighborhood of $z_0$. There exists $C>0$ such that the following holds: for any $\delta>0$ sufficiently small, and any $f\in\ell^1(\T^d)$, with
	\begin{equation}\label{small_l1}
	\|f-z_0\|_{\ell^1}\leqslant\delta,
	\end{equation}
	the function $h \circ f$ also belongs to $\mathcal \ell^1(\T^d)$, with
	\begin{equation}\label{comp_small_1}
		\|h \circ f\|_{\ell^1}\leqslant |h(z_0)|+C\delta.
	\end{equation}
\end{proposition}

\begin{proof}
	We can assume without loss of generality $z_0=0$. Let $\{a_k\}_{k\geqslant 0}$ and $R$ be respectively the coefficients and the radius of convergence of the Taylor series of $h$ centered at $0$. Denote moreover by $c^{(k)}_n$, $n\in\Z^d$ and $k\geqslant 0$, the Fourier coefficients of $f^k$.
	
	Owing to the algebra property \eqref{eq:algebra}, for every $k\geqslant 0$ we have $f^k\in\ell^1(\T^d)$, with $\|f^k\|_{\ell^1}\leqslant \delta^k$. Assume $\delta\leqslant\frac{R}{2}$. In view of \eqref{small_l1} and \eqref{eq:emb}, the composition $h\circ f$ is well-defined. The Fourier coefficients $\gamma_n$ of $h\circ f$ are well-defined too, and given by
$$\gamma_n=\sum_{k=0}^{\infty}a_kc_n^{(k)}.$$  
Then we have
	$$\|h\circ f\|_{\ell^1}=\sum_{n\in\Z^d}|\gamma_n|\leqslant\sum_{n\in\Z^d}\sum_{k=0}^{\infty}|a_kc_n^{(k)}|= \sum_{k=0}^{\infty}|a_k|\sum_{n\in\Z^d}|c_n^{(k)}|\leqslant \sum_{k=0}^{\infty}|a_k|\delta^k,$$
	and the last series is finite, since $\delta<R$. In particular, $\|h \circ f\|_{\ell^1}\leqslant |h(0)|+C\delta$ for some constant $C>0$ depending only on $h$.
\end{proof}

Next, we state some useful product and composition estimates in Sobolev spaces.  We first recall the classical fractional Leibniz rule:
\begin{lemma}
	Let $d\in\N^+$ and $s\geq 0$. Then the estimate
	\begin{equation}\label{eq:bil_sym}
		\|fg\|_{H^s}\lesssim_{s,d} \|f\|_{L^{\infty}}\|g\|_{H^s}+\|f\|_{H^s}\|g\|_{L^{\infty}}
	\end{equation}
	holds true for every $f,g\in H^s(\T^d)\cap L^{\infty}(\T^d)$.
\end{lemma}

Estimate \eqref{eq:bil_sym}, and suitable generalizations, are well-known in $\R^d$ \cite{Kato-Ponce, GOH, MUSC}, and hold true also on the torus \cite{Metivier,IOKE,GHO} -- see the recent paper \cite{Tada} and references therein for a general result and a comprehensive discussion.

We also  need an \emph{asymmetric} version of \eqref{eq:bil_sym}.

\begin{lemma}\label{le:asbi}
	Let $d\in\N^+$ and $s\geqslant 1$. Then the estimate
	\begin{equation}\label{eq:asbi}
		\|f g\|_{H^{s-1}}\lesssim_{s,d} \|f\|_{L^{\infty}}\|g\|_{H^{s-1}}+\|f\|_{H^s}\|J^{-1}g\|_{L^{\infty}}
	\end{equation}
	holds true for every $f\in H^s(\T^d)\cap L^{\infty}(\T^d)$ and $g\in H^{s-1}(\T^d)$ such that $J^{-1}g\in L^{\infty}(\T^d)$.
\end{lemma}

Estimate \eqref{le:asbi} could be deduced by paraproduct calculus, see e.g.~\cite{Metivier}. Here we provide instead a direct proof, following the same lines as in \cite[Theorem 1.4]{GK}, where the analogous result (and suitable generalizations) is established on $\R^d$. We preliminary recall a continuity bound for bilinear multiplier operators on $\T^d$. 

\begin{lemma}\label{eq:cm}
	Fix $d\in\N^+$. Let $\lambda\in\mathcal{C}^{\infty}(\R^d\times\R^d\setminus(0,0))$ be a real symbol such that
	\begin{equation}\label{ass:cf}
		\big|\partial^{\alpha}_{\xi}\partial^{\beta}_{\eta}\lambda(\xi,\eta)\big|\lesssim_{\alpha,\beta}(|\xi|+|\eta|)^{-|\alpha|-|\beta|},
	\end{equation}
	for any $\alpha,\beta\in \Z^d_{+}$ and $(\xi,\eta)\neq(0,0)$. If $I_{\lambda}$ denotes the bilinear operator \begin{equation}\label{eq:bilin}
		I_{\lambda}(f,g)(x):=\sum_{\ell_1,\ell_2\in\Z^n}e^{i(\ell_1+\ell_2)\cdot x}\lambda(\ell_1,\ell_2)\widehat{f}(\ell_1)\widehat{g}(\ell_2),
	\end{equation}
	then the estimate
	\begin{equation}\label{eq:bil-est}
		\|I_{\lambda}(f,g)\|_{L^2}\lesssim\|f\|_{L^{\infty}}\|g\|_{L^2}
	\end{equation}
	holds true for every $f\in L^{\infty}(\T^d)$ and $g\in L^{2}(\T^d)$.
\end{lemma}
The analogous statement of Lemma \ref{eq:cm} on $\R^d$ is provided by the Coifman-Meyer estimates \cite{Co-ma}. The result on $\T^d$ can be then deduced by the bilinear transference principle \cite[Theorem 3]{fansato}.

\begin{proof}[Proof of Lemma \ref{le:asbi}]
	Let $\chi:\R\to\R$ be a smooth, decreasing, odd function, with $|\chi(r)|=\frac12$ for $|r|\geqslant 1$. Consider now the function $\varphi:\R_{\geqslant 0}\to \R$ defined by $\varphi(0)=1$ and $\varphi(x)=\frac{1}{2}+\chi(\ln(x))$ for $x>0$. Observe that $\varphi$ is smooth, non-negative, non-increasing, and supported on $[0,e]$. Moreover, $\varphi(x)+\varphi(x^{-1})=1$ for every $x>0$, whence in particular $\varphi(x)=1$ for $x\in[0,\frac{1}{e}]$. Next, we can write
	
	\begin{equation}\label{eq:split-symbol}
		\begin{split}
			J^{s-1}(fg)(x)&=\sum_{\ell_1,\ell_2\in\Z^2}e^{i(\ell_1+\ell_2)\cdot x}\langle \ell_1+\ell_2\rangle^{s-1}\widehat{f}(\ell_1)\widehat{g}(\ell_2)\\
			&=\sum_{\ell_1,\ell_2\in\Z^2}e^{i(\ell_1+\ell_2)\cdot x}\varphi\Big(\frac{|\ell_1|}{|\ell_2|}\Big)\langle \ell_1+\ell_2\rangle^{s-1}\langle \ell_2\rangle^{-s+1}\widehat{f}(\ell_1)\widehat{J^{s-1}g}(\ell_2)\\
			&+\sum_{\ell_1,\ell_2\in\Z^2}e^{i(\ell_1+\ell_2)\cdot x}\varphi\Big(\frac{|\ell_2|}{|\ell_1|}\Big)\langle \ell_1+\ell_2\rangle^{s-1}\langle \ell_1\rangle^{-s}\langle\ell_2\rangle\widehat{J^sf}(\ell_1)\widehat{J^{-1}g}(\ell_2)\\
			&=I_{\lambda_1}(f,J^{s-1}g)(x)+I_{\lambda_2}(J^sf,J^{-1}g),
		\end{split}
	\end{equation}	
	where we set, for $(\xi,\eta)\in\R^d\times\R^d\setminus(0,0)$,
	\begin{equation}
		\begin{split}
			\lambda_1(\xi,\eta):=\varphi\Big(\frac{|\xi|}{|\eta|}\Big)\langle \xi+\eta\rangle^{s-1}\langle \eta\rangle^{-s+1},\\
			\lambda_2(\xi,\eta):=\varphi\Big(\frac{|\eta|}{|\xi|}\Big)\langle \xi+\eta\rangle^{s-1}\langle \xi\rangle^{-s}\langle\eta\rangle.
		\end{split}
	\end{equation}
	Since $\varphi$ and $\varphi'$ are supported, respectively, on $[0,e]$ and $[\frac1e,e]$, 
	a direct computation shows that, for $s\in\N^+$, the symbols $\lambda_1, \lambda_2$ belongs to $\mathcal{C}^{\infty}(\R^d\times\R^d\setminus(0,0))$, and satisfy assumption \eqref{ass:cf}. Using \eqref{eq:split-symbol} and Lemma \ref{eq:cm}, we then obtain
	\begin{equation}
		\begin{split}
			\|f g\|_{H^{s-1}}&=\|J^{s-1}(f g)\|_{L^2}\leqslant\|I_{\lambda_1}(f,J^{s-1}g)\|_{L^2}+\|I_{\lambda_2}(J^sf,J^{-1}g)\|_{L^2}\\
			&\lesssim\|f\|_{L^{\infty}}\|J^{s-1}g\|_{L^{2}}+\|J^sf\|_{L^2}\|J^{-1}g\|_{L^{\infty}}=\|f\|_{L^{\infty}}\|g\|_{H^{s-1}}+\|f\|_{H^s}\|J^{-1}g\|_{L^{\infty}},
		\end{split}
	\end{equation}
	when $s\in\N^+$. The general case is obtained by interpolation, arguing as in the proof of Theorem 1.4 in \cite{GK}.
\end{proof}
Finally, we state a composition lemma in Sobolev spaces.
\begin{lemma}\label{le:compo}
	Let $d\in\N^+$ and $s>0$. Let $I$ be an open interval of $\R$, and $f\in H^s(\T^d)\cap L^{\infty}(\T^d)$, with $f(\T^d)\subseteq I$. Fix moreover  $h\in\mathcal{C}^{\infty}(I)$, such that either $0\not\in I$ or $0\in I$, $h(0)=0$. Then
	\begin{equation}\label{eq:compo}
		\|h\circ f\|_{H^s}\lesssim_{d,s,h,\|f\|_{L^{\infty}}} \|f\|_{H^s}.
	\end{equation}
\end{lemma}
\begin{proof}
	If $0\in I$, $h(0)=0$, \eqref{eq:compo} is a classical composition estimate, see e.g.~\cite[Theorem 2.87]{Danchin}. If $0\not\in I$, we can extend $h$ to a smooth function on $\R$, with $h(0)=0$, and the previous case applies.
\end{proof}


\end{document}